\documentclass[english,11pt]{amsart}
\usepackage{amssymb,amsmath,amsthm}
\usepackage{srcltx}
\usepackage{babel,graphicx}
\usepackage{color}

{{}}

\newtheorem{theorem}{Theorem}[section]
\newtheorem{proposition}[theorem]{Proposition}
\newtheorem{corollary}[theorem]{Corollary}
\newtheorem{lemma}[theorem]{Lemma}

\theoremstyle{definition}
\newtheorem{definition}[theorem]{Definition}
\newtheorem{remark}[theorem]{Remark}

\newcommand{\norm}[1]{\left\Vert#1\right\Vert}

\newcommand{\N}{\mathbb{N}}
\newcommand{\Z}{\mathbb{Z}}
\newcommand{\Q}{\mathbb{Q}}
\newcommand{\R}{\mathbb{R}}

\newcommand{\e}{\epsilon} 
\newcommand{\type}{(*)}

 \newcommand{\liec}{\mathfrak c}
\newcommand{\liel}{\mathfrak l}
 \newcommand{\lieg}{\mathfrak g}
\newcommand{\lieh}{\mathfrak h}
\newcommand{\liek}{\mathfrak k}
\newcommand{\liem}{\mathfrak m}
\newcommand{\lien}{\mathfrak n}
\newcommand{\liea}{\mathfrak a}

\newcommand{\liep}{\mathfrak p}

\newcommand{\HaarK}{\eta_K} 
\newcommand{\HaarG}{\eta_G} 
\newcommand{\mGamma}{\mu_{\Gamma}} 
\newcommand{\mG}{\mu_G} 
\newcommand{\mB}{\mu_B} 
\newcommand{\mBg}{\mu_B^g} 
\newcommand{\mGBM}{\mu_G^{BM}} 
\newcommand{\mX} {\mu_X} 
\newcommand{\mDer}{\mu_{\Der}}
\newcommand{\mL}{\mu_L} 
\newcommand{\PW}{P_{\mathcal{W}}} 
\newcommand{\W}{\mathcal{W}}
\newcommand{\central}{\kappa_{\mG}}

\newcommand{\Homeo}{\text{Homeo}^+(\R)}
\newcommand{\HomeoR}{\text{Homeo}^+(\R)}
\newcommand{\HomeoS}{\text{Homeo}^+(\mathbb{S}^1)}

\newcommand{\Der} {Z} 
\newcommand{\s}{\mathbb{S}^1}

\date{2020}
\title[Non left-orderability of higher rank lattices]{Non left-orderability of lattices in higher rank semi-simple Lie groups}

\author{Bertrand Deroin and Sebastian Hurtado}

\address{CNRS  -- IMPA, Rio de Janeiro, Brasil -- AGM, Cergy-Pontoise, France} 
\email{bertrand.deroin@u-cergy.fr}

\address{Department of Mathematics
University of Chicago
Chicago, Il 60615}
\email{shurtados@uchicago.edu}

\begin{document}
\maketitle

\begin{abstract}
We prove that an irreducible lattice in a real semi-simple Lie group of real rank at least two and finite center is not left-orderable.
\end{abstract}

\section{Introduction}

\subsection{Overview} 
A group $\Gamma$ is said to be left-orderable if it has a total order \(<\) which is invariant by left multiplication, i.e. $\forall f,g,h \in \Gamma , \  f < g \ \text{implies} \  hf < hg$. For a countable group, this is well-known to be equivalent to the existence of an injective group homomorphism from \(\Gamma\) to the group of orientation preserving homeomorphisms of the real line $\HomeoR$, see e.g. \cite{Ghys groups acting on the circle.}. Left orderable groups appear naturally in different branches of mathematics including group theory, low dimensional topology, dynamical systems and foliation theory.

Well known examples of groups which are left-orderable are torsion-free abelian and nilpotent groups, free non-abelian groups and fundamental groups of closed surfaces. Other examples are right angled Artin groups (RAAGs), braid groups, mapping class groups of orientable surfaces with boundary (fixing the boundary) and Thompson's group $F$.

Many other groups are known to not be left-orderable. For example, any left-orderable group is torsion free and in particular finite groups are not left-orderable. Also, random groups in the Gromov density model are not left-orderable, \cite{Orlef}. The group $\text{SL}(n, \Z)$ and its finite index subgroups were shown to be not left-orderable by Witte-Morris \cite{Witte1}.

We consider here the question of the left-orderability of a lattice of a real semi-simple Lie group is a left-orderable group. 

We begin with a brief discussion of what is known in the case where $G$ has real rank one and trivial center (in this case, $G$ is the group of isometries of a symmetric space $S$ with strictly negative curvature). When $S$ is the hyperbolic space ($G = SO_0(n,1), \ n \geq 2$) we have the following: When $n =2 $, every torsion free lattice is left-orderable being either a surface group or free. When $n = 3$, there is an interesting conjecture of Boyer-Gordon-Watson \cite{BGW} relating the left-orderability of the fundamental group of a closed 3-manifold (more precisely a rational homology 3-sphere) and its Heeggard-Floer homology; if the manifold is hyperbolic, its fundamental group is virtually left-orderable, by Agol's solution of the virtual Haken conjecture \cite{AGM}. For $n > 3$, the picture is not well-understood yet but there are examples of closed hyperbolic manifold with left-orderable fundamental group known in every dimension; namely the standard arithmetic manifolds, see e.g. \cite{BHW}. 

For complex hyperbolic space ($G = PU(n,1), \ n > 1$) much less is known, a recent example of Agol-Stover of a lattice in $PU(2,1)$ which is RFRS (See \cite{AS}[Sec. 2]) is an example of a left-orderable lattice in $PU(2,1)$ because RFRS groups are locally indicable and locally indicable groups are left-orderable, as far as we know, this is the only known example of a left-orderable lattice in complex hyperbolic space. For quaternionic hyperbolic space ($G = Sp(n,1)$) and the Cayley plane ($G = F_4^{-20}$), no examples are known, this is related to a well known question in the theory of left-orderable groups asking whether there exists a left-orderable group with Kazhdan's Property $(T)$.

Our main result is about left-orderability of lattices in a semi-simple Lie group of real rank greater or equal to two (this is equivalent to the symmetric space $S$ containing an isometric copy of $\R^2$). Examples of such Lie groups are $\text{SL}(n, \R), n \geq 3$, $SO(n,m), n,m > 1$, and $ \text{SL}(2, \R)\times \text{SL}(2, \R)$ among many others. 


Irreducible lattices in higher rank Lie groups are known to have many rigidity properties, a key example of this fact is Margulis super-rigidity \cite{Margulisbook}, which roughly says that any linear representation of a lattice comes from a linear representation of the Lie group itself, and constitutes the main tool to prove arithmeticity of lattices. Another example is the Zimmer program, which is a series of conjectures by Robert Zimmer in the late 80's attempting to generalize Margulis super-rigidity to a non-linear setting. The program vaguely says that any smooth action of a higher rank lattice in a smooth manifold comes from algebraic actions (linear representations, boundary actions, among others) up to some minor modifications and combinations (Blowing up orbits, gluing, etc.) See \cite{Fisher} and \cite{WZ} for more precise statements and the history of the program.

 Zimmer's program might be even true if one drops the smoothness assumption and consider only topological actions (by $C^0$ homeomorphisms) on manifolds and in that category the simplest question is about understanding actions in the line, which is asking about the left-orderability of higher rank lattices. One should warn the reader that topological actions are much more difficult to deal with than smooth actions and in higher dimensions any progress seems quite difficult to achieve. For example, it is not known whether there exists a topological action of the $p$-adic integers $\Z_p$ in $\R^4$ (this is a case of the Hilbert-Smith conjecture).

In this work, we give a complete classification of left-orders or circular left-orders on irreducible lattices in semi-simple Lie groups of higher rank and correspondingly, of all the possible topological actions of the group on the line and on the circle. This completes works and solve conjectures of Ghys \cite{Ghys} and Witte-Morris \cite{Witte1, Witte2}, see the survey \cite{Witte3}.

\subsection{Statement of results}

Two actions $\phi_1, \phi_2: \Gamma \to \HomeoR$ are said to be semi-conjugate if there exists a proper continuous non-decreasing function $f: \R \to \R$ such that for every $\gamma \in \Gamma$, $f \circ \phi_1(\gamma) = \phi_2(\gamma) \circ f$.  

We denote by \(\widetilde{\text{SL}(2, \R)} \) the universal covering of the real Lie group \(\text{SL}(2, \R) \). There is a standard action $\widetilde{\text{SL}(2, \R)}$ on the real line obtained by lifting the projective action of $\text{PSL}(2,\R)$ from $\mathbb{RP}^1$ to its universal covering.

\begin{theorem} \label{t: no orderable lattice 1}
Let \(\Gamma\subset G\) be an irreducible lattice in a connected real semi-simple Lie group of rank at least two. Then,  \(\Gamma\) is left-orderable if and only if \(\Gamma\) is torsion free and there exists a surjective morphism \( G \rightarrow \widetilde{\text{SL}(2, \R)} \). Moreover, in the latter case every action $\phi: \Gamma \to \HomeoR$ is semi-conjugate to an action obtained in the following way:
\begin{enumerate}
\item The inclusion of $\Gamma$ into $G$.
\item A surjective morphism \( G\rightarrow \widetilde{\text{SL}(2, \R)} \).
\item The standard action $\widetilde{\text{SL}(2, \R)}$ in $\R$.
\end{enumerate}
\end{theorem}

\begin{remark}
The existence of a surjective morphism  \( G\rightarrow \widetilde{\text{SL}(2, \R)} \) is equivalent to $G$ being a non-trivial product with one factor equal to $ \widetilde{\text{SL}(2, \R)}$.
\end{remark}

As a corollary, in the case where $G$ has finite center we have the following:

\begin{theorem}\label{t: no orderable lattice 2}
Let \(\Gamma\) be an irreducible lattice in a connected real semi-simple Lie group $G$ with finite center and with real rank at least two. Then, $\Gamma$ is not left-orderable and moreover every $\Phi: \Gamma \to \HomeoR$ is trivial. In particular, if \(G\) is a real algebraic semi-simple Lie group of rank two, it contains no left-orderable irreducible lattice.
\end{theorem}

\begin{remark}
This statement has been conjectured by Ghys \cite{Ghys} and Dave Witte-Morris (See \cite{Witte3}). In \cite{Witte1}, Witte-Morris solved the case where $G$ is simple and $\Gamma$ has $\mathbb{Q}$-rank at least two. In \cite{LW}, Witte-Morris and Lifschitz proved some cases in $\Q$-rank one, including the case of non-uniform lattices in products of $\text{SL}(2,\R)$, their proof uses the bounded generation by unipotents property proven by Carter-Keller, \cite{Witte2}. The case where $\Gamma$ is co-compact, which in some sense is the most common case among lattices as was shown by Belolipetsky and Lubotzky \cite{BL}, was the main open problem. \end{remark}

As a consequence of Theorem \ref{t: no orderable lattice 1} and a Theorem of Ghys (see Theorem \ref{t: Ghys}), we have the following theorem for actions on the circle.

\begin{theorem}\label{t: circle actions} Let \(\Gamma\subset G\) be an irreducible lattice in a connected real semi-simple Lie group of rank at least two and finite center. If $\phi: \Gamma \to \HomeoS$, then either $\phi(\Gamma)$ is finite or $\phi$ is semi-conjugate up to finite index to an action obtained by composition of the following morphims:

\begin{enumerate}
\item The inclusion of $\Gamma$ in $G$.
\item A surjective homomorphism from $G$ to $\text{PSL}(2,\R)$.
\item The projective action of $\text{PSL}(2,\R)$ on the circle.
\end{enumerate}

\end{theorem}

\begin{remark}
\begin{enumerate}
\item For the definition of semi-conjugacy up to finite index, see the discussion before Theorem \ref{t: Ghys}. 
\item In the case where $|\phi(\Gamma)| < \infty$ the action is conjugate to an action by rotations.
\item In the case where $G$ is infinite center one has an extra-possibility :  $G$ must have a  $ \widetilde{\text{SL}(2, \R)}$ factor and  $\Gamma$ has a finite orbit $F \subset \s$. In this case, if $\Gamma'$ is the finite index subgroup of $\Gamma$ fixing $F$, the action of $\Gamma'$ in $\s \setminus F$ is as in Theorem \ref{t: no orderable lattice 1}.
\end{enumerate}
\end{remark}

Ghys proved the previous Theorem in the case where the action is of class \(C^1\) (See Theorem \ref{t: Ghys}) and Theorem \ref{t: circle actions} is a consequence of Ghys results in \cite{Ghys} and Theorem \ref{t: no orderable lattice 1}. Ghys result was also shown by Burger-Monod \cite{BurgerMonod} via the study of bounded cohomology. Bader-Furman-Shaker \cite{BFS} have a similar Theorem which holds for more general groups including some non-linear groups. Burger unified those approaches in \cite{Burger}. Navas \cite{Navas, Navas1} proved there are no infinite groups with Kazhdan's property $(T)$ in the group of $C^{1 + \frac{1}{2}}$ diffeomorphisms of the circle, strongly inspired by the work of Reznikov  \cite[Ch.2]{Reznikov}.

\subsection{Outline of the proof}



Due to work of Ghys \cite{Ghys}, it is sufficient to prove Theorem \ref{t: no orderable lattice 1} in the case where \(G\) has finite center, 
in which case it claims that no irreducible lattice $\Gamma$ is left-orderable, or equivalently (this is classical, see e.g. \cite[section 1.1.3]{DNR}), that given any lattice \( \Gamma\subset G\), any action of \(\Gamma\) on the real line by orientation preserving homeomorphisms is trivial.  The proof of this statement is by contradiction, assuming there is a non-trivial action of $\Gamma$ in $\R$ that we can assume to be fixed point free, by considering the restriction of the action to a component of the complement of the set of fixed points. 



To understand the proof it is useful to mention some of the recent progress in Zimmer's program, since our result is a special instance of this program in dimension one without any regularity apart from continuity. In \cite{BFH}, Brown-Fisher-Hurtado proved that any smooth action of an irreducible lattice $\Gamma$ of a semi-simple Lie group $G$ on a closed manifold $M$ of dimension at most $\text{Rank}_\R(G)- 1$ is finite. One of the main ingredients of the proof is the recent work of Brown, Rodriguez-Hertz, and Wang \cite{BRHW} where the authors establish the existence of a $\Gamma$-invariant probability measure in $M$ under the same assumptions on the dimension of $M$.

As in \cite{BRHW}, \cite{BFH}, our goal will be to show the existence of a $\Gamma$-invariant Radon (infinite mass) measure  on $\R$, which immediately yields the desired contradiction since a translation number with respect to such a measure would provide a non trivial cohomology class of degree one on \(\Gamma\), which contradicts a theorem of Kazhdan and Margulis. However, in our case two new difficulties arise and make impossible the implementation of the method of \cite{BRHW, BFH}, which essentially rely on the use of the Ledrappier-Young formula. The first one is that $\R$ is not compact and the second is that the action is not smooth. 
So, instead of using Ledrappier-Young formula, our argument consists in establishing a weak form of the stiffness property of Furstenberg \cite{Furstenberg stiffness} for the action, namely that after taking a suspension \(G\)-space encoding the \(\Gamma\)-action, a stationary measure is indeed globally invariant by a large normal subgroup of \(G\). There have been recently a bunch of new ideas and methods to establish stiffness of group actions in many contexts (see e.g. \cite{BFLM, BenoistQuintstiff, EskinMirzakhani, BRH}); the one we develop here is yet another one.  This strategy yields immediately to the contradiction when \(G\) is simple.  We outline this method in \ref{sss: finding invariant measure} in the case \(G= \text{SL}(3, \R)\) which covers most of the arguments.

In the semi-simple case however, the proof of Theorem \ref{t: no orderable lattice 1} requires some additional tools.  We outline this in \ref{sss: the semi-simple case}.
 
\begin{subsubsection}{The almost-periodic space}


One of our main tool is to embed the action of $\Gamma$ in $\R$ into an action of $\Gamma$ in a suitable compact space $Z$, that we call the almost-periodic space. This tool was introduced by Deroin in \cite{almost-periodic}, where it was established that given a fixed point free action of a group $\Gamma$ on the real line, there exists a one dimensional laminated compact metric space $(Z,\mathcal L_Z)$  having a fixed point free action of $\Gamma$ preserving each of the one-dimensional leaves, in such a way that the original action of $\Gamma$ on $\R$ can be chosen to be conjugate to the action on one of the leaves of $Z$. The construction provides an additional regularity, namely there is an $\R$-flow in $Z$ preserving the leaves of \(\mathcal L_Z\) in such a way that in each leaf, $\Gamma$ acts by Lipschitz homeomorphisms (with respect to the usual distance in the time parametrization of the leaves given by the $\R$-flow). The space $Z$ is not unique and can be constructed in different ways, see \cite{almost-periodic}, \cite{DKNP}. Such an almost-periodic space  $Z$ is related to the space of orders (which apparently was introduced informally by Ghys) and we believe it can become a valuable tool for understanding left-orderable groups, as for instance in \cite{MBT}.


Here we will need to  construct an almost periodic space $Z$ having the additional property that in each leaf, the Lebesgue measure defined by the $\R$-flow is a stationary measure for an appropriate probability measure $\mu_\Gamma$ on $\Gamma$: a discretization measure for the Brownian motion on the symmetric space of $G$. We will call informally this property the \textit{harmonicity property}. A consequence of it is that any \(\R\)-invariant probability measure \(\mu_Z\) on \( Z\) is \(\mu_\Gamma\)-stationary. The construction of a space \(Z\) having the harmonicity property relies heavily on the study of random walks on the group \( \text{Homeo}^+ (\R)\) made by Deroin-Kleptsyn-Navas-Parwani in \cite{DKNP}, one of the main results being that given a symmetric finitely supported probability on \(\Gamma\), any action of \(\Gamma\) on the real line without discrete orbit can be semi-conjugated to a harmonic action, namely an action such that the Lebesgue measure is \(\mu_\Gamma\)-stationary. A crucial fact is that harmonic actions are almost-periodic; this observation was established by Klepstyn and is fundamental for our purpose, since we can then define our space \(Z\) as the space of harmonic actions having a displacement bounded by some given constants.

Our proof consists then in showing that any of the \(\R\)-invariant probability measures \(\mu_Z\) on \(Z\), are also \(\Gamma\)-invariant. Since those measures are \(\mu_\Gamma\)-stationary by the  harmonicity property of the almost periodic space, this is equivalent to establishing  a weak form of the stiffness property of Furstenberg in our context of the action of \( \Gamma\) on the almost-periodic space \(Z\), see \cite{Furstenberg stiffness}.


\subsubsection{The suspension space}


As in \cite{BRHW}, \cite{BFH}, to find a $\Gamma$-invariant measure on the almost-periodic space $Z$ we consider the suspension space $X := (G \times Z)/\Gamma$, which is a space with a natural action of $G$ that fibers over $G/\Gamma$ with fiber $Z$. There is a natural correspondence between $\Gamma$-invariant measures on $Z$ (resp. \(\mu_\Gamma\)-stationary measures on \(Z\)) and $G$-invariant measures on $X$ (resp. \(\mu_G\)-stationary measures on \(X\) for symmetric absolutely continuous probability measures on \(G\)) and so we will instead prove that the \(\mu_G\)-stationary measure \( \mu_X\) on \(X\) corresponding to \( \mu_Z\) is in fact $G$-invariant. On the suspension \(X\), there is also an associated unidimensional lamination \(\mathcal L_X\) (in each \(Z\)-fiber it is defined by \(\mathcal L_Z\)) carrying a Lipschitz structure along the leaves,  which is invariant by the \(G\)-action (however in this case \(G\) might permute the \(\mathcal L_Z\)-leaves). One of the reasons we need the harmonicity property for our almost-periodic space is that it implies that on a \(G\)-invariant set of total \(\mu_X\)-measure, any element of \(G\) has a well-defined derivative along the lamination \(\mathcal L_X\).


\begin{remark}
Our construction of the suspension space and its various structures works for a left-orderable lattice in the rank one case as well, assuming it is uniform, and could be interesting for other applications. This is in particular the case for torsion free uniform lattices in \(\text{PSL} (2,\R)\) for instance, and for the known examples of left-orderable lattices in $SO(n,1)$. One can actually carry out a universal construction, by considering as the almost-periodic space the set of \textit{all} harmonic actions of the lattice on the real numbers, with a normalization on the average displacement. Equiped with a natural distance making the \(G\)-action Lipschitz, this space might be infinite dimensional in the rank one case, this is what happens in the case of a lattice in \(\text{PSL}(2,\R)\). We do not develop this point of view here however, since there are technical issues related to the control of the displacement in this generality.  In the non uniform case, there are other difficulties due to the  distortion between the word metric on the lattice and the distance coming from the symmetric space (such a phenomenon does not occur in the higher rank case thanks to work of A. Lubotzky, S. Mozes and M. Ragunathan \cite{Lubotzky Mozes Ragunathan}).  
\end{remark}

\end{subsubsection}

\subsubsection{Outline in the simple case: propagating invariance}\label{sss: finding invariant measure}

We continue with the outline of the proof in the case where \(G= \text{SL} (3,\R)\). Recall  the  Iwasawa decomposition $G = KAN$ of $G = \text{SL}(3,\R)$ where  $K = SO(3, \R)$, $A$ is the group of diagonal matrices with positive entries in $\text{SL}(3,\R)$ (the action of this group in $G$ is called the Cartan flow) and $N$ is the unipotent subgroup that can be chosen to be  the upper (or lower) triangular matrices with ones in the diagonal. Recall also the minimal parabolic group $P:= MAN$, where $M$ is the group of diagonal matrices in $\text{SL}(3, \R)\cap SO(3,\R)$.

Using the Furstenberg correspondence between harmonic functions on the symmetric space \(S\) of $G$ and bounded functions in the Poisson boundary $G/P$, one shows that $\mu_X$ corresponds to a unique $P$-invariant measure $\mu_X^{'}$ on $X$ such that $\int_K k_{*}(\mu_X') dk = \mu_X$. Therefore, to show that $\mu_X$ is $G$-invariant we only need to show that $\mu_X^{'}$ is $G$-invariant. This is done as in \cite{BRHW}, via studying the dynamics of the $A$-action on $X$, but our method is different and does not use the notion of entropy and the Ledrappier-Young entropy formula, which in our setting results problematic because $Z$ could be infinite dimensional. Our method to show $G$-invariance of $\mu_X^{'}$ consists in studying some ergodicity properties of the Cartan flow on $X$ along singular directions in $A$ and by using the action of the Weyl group. We proceed to explain briefly how this argument works.

The choice of the unipotent subgroup $N$ as the upper triangular unipotents is arbitrary, and it actually depends on a choice of a Weyl chamber for the lie algebra $\liea$ of $A$. Indeed, for each Weyl chamber $\mathcal{W}$ in $\liea$ there is a choice of unipotent subgroup $N$ and an associated minimal parabolic subgroup $P^{\mathcal{W}}$ containing $A$. For $G = \text{SL}(3,\R)$ there are six Weyl chambers and so we have six $A$-invariant probability measures $\{\mu_X^{\mathcal{W}} \}$ which are $A$-invariant (one of them will be $\mu_X^{'}$), these measures are related to each other via elements of the Weyl group. The idea of the proof is to show that many of these measures coincide (at least three for $\text{SL}(3,\R)$), this shows that $\mu^{'}_X$ is invariant by three minimal parabolic subgroups, and this is enough to show $G$-invariance of $\mu'_X$. 

In order to achieve this, one first establishes that for any such measure $\mu_X^{\mathcal{W}}$, there exists a non-trivial linear functional $\chi_{\mathcal{W}}: \liea \to \R$, which is a Lyapunov functional describing the rate of expansion of the derivative of the \(A\)-action along the one-dimensional leaves of \(\mathcal L_X\) at generic \(\mu_X\)-points, and which satisfy the following contraction property: for $\mu_X^{\mathcal{W}}$ a.e. $x \in X$, if $x, y \in X$ are in the same one-dimensional leaf \(L\) of $\mathcal L_X$, then $\lim_{t \to \infty} d_L(e^{ta}x, e^{ta}y)  = 0$ provided that $\chi_{\mathcal{W}}(a) < 0$. We call this property the \textit{global contraction property for $a \in \liea$ with respect to $\mu_X^{\mathcal{W}}$}. 

The key lemma \ref{hard1} shows that if we take an element $a \in \liea$ which is singular (i.e. in the wall between two Weyl chambers $\mathcal{W}, \mathcal{W}'$), and has the global contraction property with respect to both $\mu_X^{\mathcal{W}}, \mu_X^{\mathcal{W}'}$, then both measures coincide. The idea of the proof is to show there exists two Birkhoff generic points $x, x'$ for $\mu_X^{\mathcal{W}}, \mu_X^{\mathcal{W}'}$ for which $\lim_{t \to \infty} d_X(e^{ta}x, e^{ta}x' ) = 0$, if this were true we have $\mu_X^{\mathcal{W}} = \mu_X^{\mathcal{W}'}$ using Birkhoff's ergodic theorem. The fact that $a$ is singular is used in the fact that both measures $\mu_X^{\mathcal{W}}, \mu_X^{\mathcal{W}'}$ are invariant by the horospherical unstable subgroup $N_a$ of $e^{a}$ which allow us to take generic points $x, x'$ and ``modify their future" (using the $N_a$-invariance) in such a way that $\limsup_{t \to \infty} d_{G/\Gamma}\big( e^{ta}\pi_{G/\Gamma}(x), e^{ta}\pi_{G/\Gamma}(x') \big)$ is as small as we want (here $\pi_{G/\Gamma}$ is the natural projection from $X$ to $G/\Gamma$).  This, together with the global contraction property give us $x,x'$ with $\limsup_{t \to \infty} d_X(e^{ta}x, e^{ta}x' )$ as small as we want and we conclude $\mu_X^{\mathcal{W}} = \mu_X^{\mathcal{W}'}$ using Birkhoff's ergodic theorem.
 
 \subsubsection{Outline in the semi-simple case}\label{sss: the semi-simple case}
 
 In the case where \(G\) is a non trivial almost-product of simple Lie groups, we can assume up to taking a finite index subgroup of \(\Gamma\) and a finite covering of \(G\), that \(G\) is indeed isomorphic to a product \( G_1 \times \ldots \times G_r\) of simple Lie groups \(G_k\). The strategy described in \ref{sss: finding invariant measure} leads to the conclusion that the stationary measure \(\mX\) is invariant by all factors \( G_2 , \ldots, G_r\), and that \(G_1\) is of rank one.  
  
To find a contradiction in this case, we control the displacement of the elements of \(\Gamma\) along the leaves of the lamination \(\mathcal L _Z\) equiped with their \(\Gamma\)-invariant Lipschitz structures coming from the time parametrization by the \(\R\)-flow. We prove that, denoting by \(p_1 : G\rightarrow G_1\) the projection on the \(G_1\) factor, any element in a compact subset of \( G_1\) can be approximated by a sequence of elements of \( p_1( \Gamma ) \) whose displacements remains bounded by a constant that depends only on the chosen compact subset of \( G_1\).  This step uses the fact that we can approximate elements of \( G_1\) by elements of \(\Gamma\) with an error which is exponentially small in terms of the length function on \(\Gamma\) with respect to a finite system of generators, that we deduce from the local spectral gap estimates in simple Lie groups obtained recently by Boutonnet, Ioana and Salehi-Golsefidy in \cite{Boutonnet-al}.  
 
Furstenberg's Poisson formula and the fact that \( \mu_X\) is \(\mu_G\)-stationary and invariant by \( G_k\) for \(k\geq 2\) permits to write \( \mu_Z \) as an integral over the boundary \(B(G_1)\) of probability measures \(\mu_Z ^{\zeta_1}\) on \(Z\), for \(\zeta_1\in B(G_1)\), with the following equivariance property: \( \mu_Z^{p_1(\gamma) \zeta_1} = \gamma_* \mu_Z^{\zeta_1}\).  This equivariance and the afore mentioned displacement control, permits to show that all the measures \( \mu_Z ^{\zeta_1}\) can be transported mutually to each other along the lamination \( \mathcal L_Z\) by a certain bounded amount of displacement. This notion of transportation of measures along a lamination is reminiscent of Kantorovich's theory, and it seems to us that this is a new tool. Since the measures \(\gamma _* \mu_Z \) are also expressed as an integral of the measures \( \mu_Z^{\zeta_1}\), we infer that all the measures \( \gamma _* \mu_Z\) can be transported from \( \mu_Z\) by a certain amount of displacement, which is independant of \(\gamma\).  This property easily contradicts the global contraction property of the action of \(\Gamma\) in restriction to any leaf of \(\mathcal L_Z\).

 \subsubsection{Organization of the article}
 
 In Section \ref{reductionfinitecenter}, we show that Theorem \ref{t: no orderable lattice 1} and Theorem \ref{t: circle actions} are a consequence of Theorem \ref{t: no orderable lattice 2}. As a consequence, after section \ref{reductionfinitecenter} we assume that the Lie group $G$ has finite center.  In section \ref{Lietheory} and \ref{ergodictheory} we recall some facts from Lie Theory and Ergodic Theory. In Section \ref{almostperiodicspace}, we construct the space of almost periodic actions $Z$ and its suspension space and its various structures. In sections \ref{globalcontraction1} and \ref{globalcontraction2} we describe the global contraction property and the Lyapunov functionals $\chi_{\mathcal{W}}$ for the action of $G$ on the suspension space. In Section \ref{maintheoremsimplecase}, we  discuss the propagation of  invariance of the measures $\mu_X^{\mathcal{W}}$ on the suspension space and prove Theorem \ref{t: no orderable lattice 2} when $G$ is simple. Finally, in Section \ref{maintheoremsemisimplecase} we prove Theorem \ref{t: no orderable lattice 2} in the case where $G$ is a non-trivial almost product.
 
 \vspace{1cm}

\textit{Acknowledgments.} We would like to thank: Dave Witte Morris and \'Etienne Ghys for their inspiring conjecture. Victor Kleptsyn for his crucial observation \cite[Proposition 8.1]{DKNP} establishing that harmonic representations are almost-periodic, and enabling our construction of a useful almost-periodic space. Ian Agol, 
Richard Aoun, 
Uri Bader, Nicolas Bergeron, 
Aaron Brown, Yves de Cornulier, Alex Furman, Victor Kleptsyn, Andr\'es Navas and Cristobal Rivas for fruitful discussions.  IPAM/UCLA for his organization of the program  "New methods in the Zimmer Program" where this project were initiated. UNAM for hosting the conference "Groups of dynamical origin II" where we pursued research on this topic.

\section{Reduction to the case of finite center}\label{reductionfinitecenter}

In this section we show that Theorem \ref{t: no orderable lattice 2} implies Theorem \ref{t: no orderable lattice 1} and Theorem \ref{t: circle actions}.

We first prove the easy implication of Theorem \ref{t: no orderable lattice 1}. If there exists a surjective morphism \( G\rightarrow \widetilde{\text{SL}(2, \R)} \), then we have an action $\phi: \Gamma \to \HomeoR$ via the standard action of $\widetilde{\text{SL}(2, \R)}$ in $\R$. The kernel of $\phi$ is a normal subgroup of $\Gamma$ and from Margulis Normal subgroup Theorem it follows that $\ker{\pi}$ is contained in the center $Z(\Gamma)$ of $\Gamma$, therefore $\ker(\pi)$ is an abelian group. Moreover, $\ker(\pi)$ is torsion free because $\Gamma$ is torsion free. Therefore, we have an exact sequence $\ker{\pi} \to \Gamma \to \pi(\Gamma)$, and as $\ker(\pi), \pi(\Gamma)$ are left orderable groups, $\Gamma$ is left orderable.

We now prove the other implication of Theorem \ref{t: no orderable lattice 1}. Assume that $\Gamma$ is left-orderable and so we have a faithful action $\phi: \Gamma \to \HomeoR$, we can assume that the action $\phi$ is fixed point free. We will make use of the following important statement due to Margulis, based on previous results of Kazhdan, see \cite[19, remark IX 6.20]{Margulisbook}:

\begin{proposition}\label{p: vanishing first Betti number}
Let \(\Gamma\subset G\) be an irreducible lattice in a connected real semi-simple Lie group of rank at least two. Then any homomorphism $\Gamma \to \Z$ is trivial.
\end{proposition}

As $G$ has infinite center, the intersection $Z(G) \cap \Gamma$ is an infinite abelian group, see \cite[Ch. 9, Sect. 6]{Margulisbook}. Let $K= \text{Fix}(Z(G) \cap \Gamma)\subset \R$ be the fixed point set of $Z(G) \cap \Gamma$; $K$ is a $\Gamma$-invariant closed set.  The action $\phi:\Gamma \to \HomeoR$ has no discrete orbits, otherwise we obtain a non-trivial homomorphism of $\Gamma$ to $\Z$. By \cite[Lemma 3.5.18]{DNR}, if $K$ is non-empty, there exists a unique minimal closed $\Gamma$-invariant set $\Lambda$. Collapsing the intervals of $\R \setminus \Lambda$ we obtain an action of $\Gamma/(Z(G) \cap \Gamma)$ on the real line which is minimal. However,  $\Gamma/(Z(G) \cap \Gamma)$ is a lattice in $G/Z(G)$ which is not left orderable by Theorem  \ref{t: no orderable lattice 1}, obtaining a contradiction; therefore $K$ is empty. As $Z(G) \cap \Gamma$ is abelian, one can find an element $\gamma \in Z(G)\cap \Gamma$ which acts freely in $\R$ and so we obtain a circle $\mathbb{S}^1 := \R/\phi(\gamma)$ with an action $\phi_1 : \Gamma \to \HomeoS$ induced by $\phi$. \\

In order to continue, we need to recall the following Theorem of Ghys about actions of higher rank lattices in the circle. We say that two actions $\phi_1, \phi_2: \Gamma \to \HomeoS$ are semi-conjugate up to finite cover if there exists a $f: \mathbb{S}^1 \to \mathbb{S}^1$, continuous, locally non-decreasing, surjective map such that $f \circ \phi_1(\gamma) = \phi_2 \circ f (\gamma)$ for every $\gamma \in \Gamma$. 

\begin{theorem}[\cite{Ghys}, Thm. 3.1.]\label{t: Ghys} Let \(\Gamma\subset G\) be an irreducible lattice in a connected real semi-simple Lie group of rank at least two and $\phi: \Gamma \to \HomeoS$, then $\phi$ either preserves a probability measure on $\mathbb{S}^1$ or it is semi-conjugate up to finite index to an action obtained by composition of the following morphims:

\begin{enumerate}
\item The inclusion of $\Gamma$ in $G$.
\item A surjective homomorphism from $G$ to $\text{PSL}(2,\R)$.
\item The projective action of $\text{PSL}(2,\R)$ in $\HomeoS$.
\end{enumerate}

\end{theorem}

Applying Ghys's theorem, the action  $\phi_1: \Gamma/(\gamma) \to \HomeoS$ described before, either preserves a probability measure on $\mathbb{S}^1$ or it is semi-conjugate up to finite index to a projective action. The action of  $\Gamma/(\gamma)$ does not preserve a probability measure on $\mathbb{S}^1$, otherwise by looking at the rotation number (which is a homomorphism in the case where the action preserves a probability measure) one obtains a non-trivial homomorphism $\Gamma \to \Z$. Therefore the action $\phi_1$ is semi-conjugate to a projective action $\phi_2: \Gamma \to \text{PSL}(2,\R)$ via a surjective continuous map  $f:\mathbb{S}^1 \to \mathbb{S}^1$.\\

We can define a homomorphism $\alpha: \Gamma \to  \widetilde{\text{SL}(2, \R)}$ in the following way, let $\widetilde{f}: \R \to \R$ be a fixed lift of $f$ and for each $\gamma \in \Gamma$, let $\alpha(\gamma)$ be the unique element of $\widetilde{\text{SL}(2, \R)}$  such that $\widetilde{f} \circ \phi(\gamma) = \alpha(\gamma) \circ \widetilde{f}$, the uniqueness of $\alpha(\gamma)$ implies that $\alpha$ is a homomorphism. From Ghys's Theorem we also have a surjective homomorphism $\pi: G \to \text{PSL}(2,\R)$. This homomorphism defines a unique copy of the lie algebra of $\text{SL}(2, \R)$ as an ideal in the Lie algebra  of $G$. As $G$ is semi-simple, the universal covering $\widetilde{G}$ of $G$ is a non-trivial product $\widetilde{G} =  \widetilde{\text{SL}(2, \R)} \times L$, where $L$ is a simply connected Lie group and $G = \widetilde{G}/ \Lambda$, where $\Lambda$ is a subgroup of the center of $\widetilde{G}$. We will prove that $\Lambda$ is contained in $L$ and this will show that $G$ surjects onto $\widetilde{\text{SL}(2, \R)}$.\\

This can be proven as follows: Let $\pi_1, \pi_2, \pi_G$ be the projections of $\widetilde{G}$ onto $\widetilde{\text{SL}(2, \R)}$, $L$ and $G$ respectively. Let $\widetilde{\Gamma}$ be the lift of $\Gamma$ to $\widetilde{G}$. We can define a  group homomorphism $\Pi: \widetilde{\Gamma} \to \widetilde{\Gamma}$ by $\Pi(\gamma) := (\alpha(\pi_G(\gamma)), \pi_2(\gamma))$ for every $\gamma \in \widetilde{\Gamma}$.  Assume for the moment  that  $\Pi$ is the identity homomorphism (we will check this later). This implies that  $\alpha \circ \pi_G = \pi_1$ and so for every  $\gamma \in \Lambda$ (notice that $\Lambda \subset \widetilde{\Gamma}$), we have $\pi_1(\gamma) = \alpha(\pi_G(\gamma))$, the latter being equal to the identity of $\widetilde{\text{SL}(2, \R)}$, therefore $\gamma \in L$ as we wanted to show.

Moreover, observe that $\pi_1$ is now a well defined projection of $G$ and that our original action $\phi$ is semi-conjugate via $\widetilde{f}$ with the action obtained by first projecting $\Gamma$ to $\widetilde{\text{SL}(2, \R)}$ via $\pi_1$ as described in Theorem \ref{t: no orderable lattice 1}. The only remaining thing to check is that $\Pi$ is the identity homomorphism.

\begin{proposition} For every $\gamma \in \widetilde{\Gamma}$, $\Pi(\gamma) = \gamma$. 
\end{proposition}
\begin{proof} 
It is enough to check that $\alpha \circ \pi_G = \pi_1$. Observe that for every $\gamma \in \widetilde{\Gamma}$, the projection of both $\Pi(\gamma)$ and $\gamma$ in $\text{PSL}(2,\R)$ coincide. Therefore there exists $c(\gamma) \in Z(\Gamma)$ such that $\Pi(\gamma) = \gamma c(\gamma)$. Observe that as $\Pi$ is a homomorphism, then $c: \Gamma \to Z(\Gamma)$ is a homomorphism into an abelian group and so it has a finite image, which must be contained in $L$ because $\widetilde{\text{SL}(2, \R)}$ has infinite center. Therefore  $\pi_1(\gamma) = \pi_1(\Pi(\gamma))$ and this is exactly the equation $\alpha \circ \pi_G = \pi_1$.
\end{proof}

\section{Review of Lie theory}\label{Lietheory}

\subsection{Iwasawa decomposition, root systems and strong stable unipotent subgroups}

We recall some classical facts about real semi-simple Lie groups, we will use the same notation as in \cite[Ch.6]{Knapp}.  We refer the reader to that reference for more precise statements and proofs of most of this material.

Let $G$ be a real semi-simple Lie group and $\lieg$ be the corresponding real semi-simple Lie algebra, let $B: \lieg \times \lieg \to \R$ be the Killing form in $\lieg$  given by $B(x,y) = \text{Tr}(\text{ad}(x)\text{ad}(y))$ where $\text{ad}$ denotes the adjoint action of $\lieg$ in itself. We fix a Cartan involution $\theta: \lieg \to \lieg$; this is a Lie algebra homomorphism such that $B_{\theta}(x,y) = -B(x, \theta(y))$ is a positive definite bilinear symmetric form. Any real semi-simple Lie algebra $\lieg$ is isomorphic to a Lie algebra of matrices closed under transpose  and using such an identification one can assume  the Killing form is given for $x , y \in \lieg$ by  $B(x ,y) = \text{Tr}(xy)$ and the Cartan involution is given by $\theta(x) = -x^{*}$. We write $\liek$ and $\liep$, respectively,  for the $+1$ and $-1$ eigenspaces of $\theta$. Denote by $\liea$ a maximal abelian subalgebra of $\liep$. 

We recall the restricted root decomposition of $\lieg$. The adjoint action of $\liea$ in $\lieg$ gives a decomposition  $\lieg = \lieg_{0}\oplus \bigoplus_{\alpha\in \Sigma} \lieg_\alpha$ where $\Sigma$ denotes a finite subset of the dual space $\liea^{*}$ which is known as the set of \emph{restricted roots} of $\lieg$ with respect to $\liea$. For all $a \in \liea $ and $x \in \lieg_{\alpha}$ we have $ad(a) x = \alpha (a) x$. We have the following:

\begin{theorem}[Prop.6.4,\cite{Knapp}]\label{Knapp1} The restricted root space decomposition has the following properties:

\begin{enumerate}

\item $\lieg$ is the orthogonal direct sum $\lieg = \lieg_{0} \oplus \bigoplus_{\alpha\in \Sigma} \lieg_\alpha$.
\item $[\lieg_{\alpha}, \lieg_{\beta}] \subset \lieg_{\alpha + \beta}$ for every \(\alpha, \beta \in \Sigma\),
\item $\theta \lieg_{\alpha} = \lieg_{-\alpha}$,
\item $\lieg_0 = \liea \oplus \liem$, where $\liem$ is the centralizer of $\liea$ in $\liek$.

\end{enumerate}
\end{theorem}

We will identify $\liea$ with $\liea^{*}$ via the Killing form as follows: for every $\lambda \in \liea^{*}$, we let $a_{\lambda}$ be its dual under the Killing form, that is, for every $a \in \liea$ we have: $$B(a_{\lambda}, a) = \lambda(a).$$
This defines a non-degenerate metric in $\liea^{*}$ given by $$(\lambda, \beta) := B(a_{\lambda}, a_{\beta}).$$

We will need the following:

\begin{theorem}[Prop.6.52,\cite{Knapp}]\label{Knapp2} Let $\lambda$ be a restricted root of $\lieg$, and let $x_{\lambda}$ be a nonzero element in $\lieg_{\lambda}$. We have: $$[x_{\lambda}, \theta (x_{\lambda})] =B(x_{\lambda}, \theta (x_{\lambda}))a_{\lambda} \text{ and } B(x_{\lambda}, \theta (x_{\lambda})) < 0.$$
\end{theorem}

\subsubsection{Iwasawa Decomposition}\label{iwasawa}

Recall that a \emph{base} (or a collection of \emph{simple roots}) for $\Sigma$ is a subset $\Pi\subset \Sigma$ which is a basis for the vector space $\liea^*$ and which is such that every non-zero root $\beta\in \Sigma$ is either a positive or a negative integer combination of elements of $\Pi$.  For a choice of $\Pi$, elements $\beta\in \Pi$ are called \emph{simple} (positive) roots.  Relative to a choice of base $\Pi$, let $\Sigma_+\subset \Sigma$ be the collection of positive roots and let $\Sigma_-$ be the corresponding set of negative roots.  Then $\lien =\bigoplus_{\beta\in \Sigma_+} \lieg^\beta$ is a nilpotent subalgebra and we have the Iwasawa decomposition $\lieg = \liek \oplus \liea \oplus \lien$ which gives rise to the Iwasawa decomposition $G = KAN$.  The Lie  exponential $\exp\colon \lieg\to G$ restricts to  diffeomorphisms between $\liea$ and $A$ and $\lien$ and the nilpotent group $N$.

In the case where $G$ has finite center, $K$ is a compact subgroup.  If we let $M$ to be the centralizer of $A$ in $K$, then $M$ is a closed subgroup of $G$ and the Lie algebra $\liem$ of $M$ is the centralizer of $\liea$ in $\liek$. The group $P := MAN$ is the \emph{standard minimal parabolic subgroup}. The identity component of $M$ is compact  and the group $P$ is amenable.

\subsubsection{Weyl group, Weyl chambers and the parabolic subgroups $P_{\W}$}\label{weylgroup}

From \cite[Thm. 6.53]{Knapp}, it follows that $\Sigma$ is an abstract root system. The Killing form on \(\liea\) being positive definite induces an inner product in $\liea^{*}$ by duality. We let $W(\Sigma)$ to be the Weyl group of $\Sigma$; this is a finite group acting on $\liea^{*}$ (and also on $\liea$ by duality) which is generated by reflections across hyperplanes perpendicular to elements in $\Sigma$.

We let $\mathcal{N}_K(\liea)$ (resp. $\mathcal{Z}_K(\liea)$) be the normalizer (resp. centralizer) of $A$ in $K$.

\begin{theorem}[Thm. 6.5, \cite{Knapp}]\label{weyl1} There is a natural isomorphism $W(\Sigma) = \mathcal{N}_{K}(\liea) / \mathcal{Z}_K(\liea)$.
\end{theorem}

For a choice $\Pi$ of simple roots of $\Sigma$ we define the corresponding \emph{Weyl Chamber} by $$\W := \{ a \in \liea | \lambda(a) \geq 0 \text{ for all } a \in \Pi\}. $$ 
This produces a one to one correspondence between the set of Weyl chambers and the set of choices of simple roots. 

The Weyl group acts transitively on the set of Weyl chambers.  For a Weyl chamber $\W$, we let $P_{\W}$ to be the minimal parabolic subgroup associated to the corresponding choice of simple roots defining $\W$.

\subsection{Horospherical subgroups}\label{horospherical}

We warn the reader that the following notation is not standard. For any \(a \in \liea\), we define \(N_a\) to be the \emph{unstable horospherical} subgroup of \(\exp(a)\) in \(G\), that is 
\[N_a := \{ g \in G | \  \lim_{n \to -\infty} d_{G}(\exp(na)g\exp(-an), \text{Id}) = 0 \}.\]

The group \(N_a\) is a closed connected subgroup of \(G\). The Lie algebra of \(N_a\) can be described as follows. 
If we let \[\Sigma_a := \{ \lambda \in \Sigma | \lambda(a)< 0  \},\] then \( \lien_{a} = \bigoplus_{\lambda\in \Sigma_a} \lieg_\lambda \). The Lie algebra $\lien_{a}$ is a nilpotent subalgebra. Moreover if $a$ is not in a wall of a Weyl chamber (i.e. if $\lambda(a) \neq 0$ for all $a \in \Sigma$), then $\lien_a$ can be taken to be $\lien$, the Lie algebra corresponding to $N$. For any choice of $a \in A$, we can choose $N$ in such a way that $G = KAN$ and $N_a \subset N$ by choosing $\Sigma^{+}$ such that $\Sigma_a \subset \Sigma^{+}$.

We can define also the \emph{stable horocyclic group}  $L_a := N_{-a}$ and the  \emph{central group} $C_a$ to be the identity component of the centralizer of $\exp(a) \in G$. If we let $\liel_a, \liec_a$ to be the lie algebra of $L_a$ and $C_a$, 
we have the following decomposition of the Lie algebra $$\lieg = \liel_a \oplus \liec_a  \oplus \lien_a.$$

From the previous discussion, by using the exponential map, we have the following proposition, which is similar to the LU-factorization of matrices:

\begin{proposition}\label{lufactor} The function $\mathcal{F}: L_a \times C_a \times N_a \to G$ given by $\mathcal{F}(l,c,n) = lcn$ is a diffeomorphism of a neighborhood of the identity in $L_a \times C_a \times N_a$ and a neighborhood of the identity in $G$.
\end{proposition}

\subsubsection{KAK decomposition and Cartan projection}\label{KAKdecomp}

\begin{theorem}[KAK decomposition] 
Let $G$ be a real semi-simple Lie group with finite center and let $\W$ be a fixed Weyl Chamber, then for every $g \in G$, there exists $\kappa_{\W}(g) \in \W$, and \(k_\W (g), k'_\W  (g)\in K\)  such that $$g = k'_\W (g)  \exp (\kappa_{\W} (g))k_\W (g).$$ 
Moreover $k_\W (g)$ (and $k'_\W (g)$) are unique up to right (or left) multiplication by an element of  $M$, and \(\kappa_\W (g)\) is unique if it belongs to the interior of \(\W\). 
\end{theorem}

The assignment $g \to \kappa_\W(g)$ is known as the \emph{Cartan projection}. For more general statements, see \cite[Thm. 7.39]{Knapp} and \cite[Sect 6.7.4]{BenoistQuint}.

\subsection{Harmonic functions and Furstenberg's Poisson formula}  \label{ss: harmonic functions}


Let \(G\) be a semi-simple Lie group with finite center. In the sequel we will use left random walks on \(G\) associated to either the distribution of the Brownian motion on \(G\) (for a \(K\)-invariant riemannian metric on \(G\)) at time one, or to a probability measure of type \(\type\): 


\begin{definition}
A probability measure \(\mG\) on \(G\) has type \(\type\) if it is 

\vspace{0.2cm}

1. absolutely continuous with respect to Haar measure on \(G\),

\vspace{0.2cm}

2. symmetric,

\vspace{0.2cm}

3. compactly supported, 

\vspace{0.2cm}

4. \(K\) bi-invariant.
\end{definition}

In this section we review some aspects of the fundational paper \cite{Furstenberg Poisson Formula} of Furstenberg. Most of the results do not need the symmetry nor the compactly supported assumptions.

\begin{definition}\label{d: harmonic function} A left \(\mG\)-harmonic function on \(G\) is a measurable function \( f: G\rightarrow \R \) which is such that the function \( h\mapsto f(hg) \) belongs to \(L^1(\mG) \) for every \(g\in G\) and which satisfies  
\[ f (g) = \int f (h  g ) \mG(d h) \] 
for every \( g\in G\). 
\end{definition}

A left \(\mG\)-harmonic function on \(G\) is then invariant by left multiplications  by \(K\). Furstenberg proved that for a bounded function, being left \(\mG\)-harmonic does not depend upon the choice of the measure \(\mG\) satisfying \(\type\), see \cite[Remark after Theorem 4.1]{Furstenberg Poisson Formula} (notice that while in Furstenberg's paper the convolutions by \(\mG\) are done on the right, here they are done on the left). In particular, bounded harmonic functions on \(G\) are in natural correspondence with bounded harmonic functions on the symmetric space \(K\backslash G\) in a more classical sense, namely the functions of class \(C^2\) belonging to the kernel of any elliptic second order operator on \(G\) which is invariant by right multiplications by \(G\) and by left multiplications by \(K\). 


Let us recall the following classical lemma. 

\begin{lemma}\label{eq: preliminary Poisson formula} Suppose that \(L\subset G\) is a closed subgroup such that \( G=LK\), and that \(f: G\rightarrow \R\) is a bounded measurable function which is left \(L\)-invariant.  Then the function \(F: G\rightarrow \R\) defined by  \begin{equation}\label{eq: poisson formula} F(g) := \int _K f (kg) \HaarK (dk) \end{equation} is left \(\mG\)-harmonic.\end{lemma}

\begin{proof} We have \[ \int F (hg) \mG (dh) = \int F(h'(k')^{-1}g) m'_G(dh') \HaarK(dk') = \]\[=\int f(k h' (k')^{-1}g) \HaarK(dk) \mG(dh') \HaarK(dk'). \]For every \( k\in K\) and \( h'\in G\) there exists \(l = l(k,h') \in L\) and \( k''(k,h') \in K\) such that \( kh' = l k''\). We denote by \( m''\) the image of the measure \( \HaarK\otimes \mG\) on \(G\) by the map \( (k,h') \mapsto k''\); this is a probability measure supported on \(K\). We then have \[ \int F (hg) \mG (dh) = \int f(l k'' (k')^{-1} g) \HaarK(dk) \mG(dh') \HaarK(dk') = \](and using \(P\)-invariance)\[= \int f( k'' (k')^{-1} g) m'' (dk'') \HaarK(dk')=\](using invariance of \(\HaarK\) by right multiplications)\[=\int f(kg) \HaarK(dk) = F(g) \]which ends the proof of the harmonicity of \(F\). \end{proof}

The Furstenberg boundary associated to \(K\) is the space \(B= M \backslash K\). Notice that for any Weyl chamber \(\mathcal W\), $B$ is naturally in bijection with the compact \(G\)-space \( P_{\mathcal W} \backslash G\), the \(G\)-action being on the right, where \(P_{\mathcal W}\) is the standard minimal parabolic subgroup associated to \(\mathcal W\) defined in \ref{iwasawa}. 

\begin{theorem}[Poisson formula]\label{t: Poisson formula}
The formula \eqref{eq: poisson formula} gives a linear bijective correspondance between the space of bounded measurable functions on \(G\) invariant by left multiplication by \(P_{\mathcal W}\) (namely bounded measurable functions on \( B \simeq P_{\mathcal W} \backslash G\)) and the space of bounded left \(\mG\)-harmonic functions on \(G\). 
\end{theorem}

The reference is \cite[Theorem 4.2]{Furstenberg Poisson Formula}. The fact that the maximal boundary is \( B \) is the remark due to C. C. Moore just before \cite[Theorem 1.10]{Furstenberg Poisson Formula}. 

The Poisson formula is presented in another form in Furstenberg's paper. Let \(\{\mBg\}_{g\in G}\) be the family of harmonic measures on \({B_\mathcal W}\) defined in the following way. For every \(g\in G\), \(\mBg\) is the image of the measure \(\HaarK\) by the map \( k\in K\mapsto P_{\mathcal W} kg \in B\simeq P_{\mathcal W}\backslash G\). This is a harmonic family of measures, in the sense that the function \( g \in G\mapsto \mBg \in \text{Prob} (B) \) is harmonic. The Poisson formula can be restated by saying that the bounded harmonic functions on \(G\) are those of the form 
\[ F (g ) = \int f d\mBg \text{ where } f\in L^\infty(B). \]

We end this section with the following result, that will be helpful for us. 

\begin{theorem} \label{t: poisson boundary product}
Given semi-simple Lie groups with finite center \(G_1, \ldots, G_n\), and \( K_1, \ldots, K_n\) corresponding maximal compact subgroups, a bounded function \( f\in L^\infty (G_1\times \ldots \times G_n)\) is harmonic if and only if it is harmonic in each \( G_k\)-coordinates. This produces an identification \(B(G) \simeq B(G_1) \times \ldots \times B(G_n)\). 
\end{theorem}

\subsection{Discretization}\label{ss: discretization}

In this section, we describe a procedure that enables to discretize the Brownian motion on a symmetric manifold by a random walk on its fundamental group. This is originally due to Furstenberg as well, see \cite{Furstenberg discretization}, but we will follow here some developments that appeared later in history,  the works of Lyons-Sullivan \cite{Lyons Sullivan} and Ledrappier-Ballmann \cite{Ballmann-Ledrappier}.

\begin{definition} 
Let \(\Gamma\) be a lattice in \( G\). A discretization measure on \(\Gamma\) is a probability measure \(\mGamma\) on \(\Gamma\) which has the property that for every probability measure \(\mG\) on \(G\) of type \(\type\), the restriction of any bounded left \(\mG\)-harmonic function \( f: G\rightarrow \R\) to \(\Gamma\) is left \(\mGamma\)-harmonic. 
\end{definition}

We equip \(G\) with a right invariant riemannian metric, which is also invariant by left multiplications by \(K\). The Brownian motion generated by the associated Laplacian operator is a continuous time Markov process which is invariant by right multiplications by \(G\). It is associated to a semi-group of probability measures (for the convolution operation) on \(G\) whose time one is \(\mGBM\).


Recall the notion of a balanced Lyons-Sullivan data (for short LS data) on \(G\): this is a family of sets \(  V_\gamma , F_\gamma \subset G\) for \(\gamma\in \Gamma\),  such that 

\vspace{0.2cm} 

 (D1) \(\gamma \in \text{Int} (F_\gamma) \subset V_\gamma\) for every \(\gamma\in \Gamma\)

\vspace{0.2cm}

(D2) \( F_\gamma \cap V_{\gamma'} =\emptyset \) if \(\gamma \neq \gamma '\in \Gamma\). 

\vspace{0.2cm}

(D3) \(F=\cup _\gamma F_\gamma\) is recurrent for the Brownian motion on \(G\)

\vspace{0.2cm}

(D4) there exists a constant \(C\geq 1\) such that for every \(\gamma\in \Gamma\) and every  \(g\in F_\gamma\), we have 
\[ \frac{1}{C} \leq \frac{d\varepsilon (g, V_\gamma) }{d\varepsilon (\gamma, V_\gamma)} \leq C \]
where \( \varepsilon (\cdot, V_\gamma) \) is the distribution of the point of first exist of the domain \(V_\gamma\) for a Brownian trajectory starting at the point \(\cdot\).

\vspace{0.2cm}

(D5) there exists a constant \(D\) such that \(G_{V_\gamma} ( g, \gamma)= D\) for every \(\gamma\in \Gamma\) and every \(g\in \partial F_\gamma\), where \(G_{V_\gamma}\) is the Green function of \(V_\gamma\) (namely the fundamental solution of the Laplacian that vanishes on \(\partial V_\gamma\)).

\vspace{0.2cm}
  
To construct such a Lyons-Sullivan data it suffices to start with any relatively compact open neighborhood \(V_e\) of the identity which is disjoint from all its images by the elements of \(\Gamma\), and then to set \( F_e = \{g\ |\ G_{V_e} (g, e) \geq 1\}\).  The family of sets \( V_\gamma:=  V_e \gamma^{-1}  \), \(F_\Gamma := F_e \gamma^{-1}\) satisfies the axioms D1--D5. 

This data gives rise to a family of Lyons-Sullivan measures, namely a family of probability measures \( \{\mGamma ^g\} _{g\in G} \) on \( \Gamma \) such that for any bounded harmonic function \( H: G \rightarrow \R \) we have 
\begin{equation}\label{eq: restriction Gamma} H(g) = \int H( \gamma ) \mGamma ^g(d\gamma) \text{ for every } g\in G.\end{equation}
We refer to \cite[p. 7]{Ballmann-Ledrappier} for an iterative  construction of such a family. 


Setting \( \mGamma := \mGamma ^e ,\) we get that for every bounded harmonic function \(H: G \rightarrow \R \), its restriction to \(\Gamma\) satisfies  
\[ H(\gamma   ) = \int  H(s \gamma ) \mGamma (ds) \]
In particular, \(\mGamma\) is a discretization measure. 

\begin{theorem} 
\label{t: Martin boundary of discretization} 
We have the following properties:
\vspace{0.2cm} 

1. For any absolutely continuous right \(K\)-invariant probability measure \(\mG\) on \(G\), the restriction map induces an isomorphism between the spaces of bounded left \(\mG\)-harmonic functions on \(G\) and the space of bounded left \(\mGamma\)-harmonic functions on \(\Gamma\), \cite{Kaimanovich, Ledrappier}.

\vspace{0.2cm} 

2. The discretization measure \(\mGamma\) is symmetric, \cite{Ballmann-Ledrappier}. 

\vspace{0.2cm} 

3. Every positive left \(\mGamma\)-harmonic function on \(\Gamma\) can be extended in a unique way to a positive left \(\mGBM\)-harmonic function on \(G\), and the extension map is continuous when we equip the functional spaces with the topology of uniform convergence on compact subsets \cite{Ballmann-Ledrappier}.

\end{theorem}

Given a finite generating set \( S\) on \(\Gamma\), it is straighforward to prove that there exists a positive constant \(\xi= \xi(m_\Gamma , S) \) such that the logarithm of a positive left \(m_\Gamma\)-harmonic function on \(\Gamma\) is \(\xi\)-Lipschitz with respect to a word metric on \(\Gamma\). This property makes the set of positive \(m_\Gamma\)-harmonic functions a compact space. We recall here that the same phenomenon occurs for positive left-\(\mGBM\)-harmonic functions. We denote by \( d_{K\backslash G} \) the distance on the symmetric space \( K\backslash G\) of the semi-simple Lie group \(G\) (the normalisation by a positive constant will not be of importance for us). 

\begin{theorem} [Harnack inequality] \label{t: Harnack inequality}
There exists a constant \( \xi= \xi(G)\) such that the logarithm of every positive left \( \mGBM\)-harmonic function on \(G\) induces a function on the symmetric space \( K\backslash G\) which is  \(\xi\)-Lipschitz with respect to the distance \( d_{K\backslash G}\). \end{theorem}

We warn the reader that if the choice of \(\mG\) among the set of absolutely continuous right \(K\)-invariant probability measures does not matter for the notion of \textit{bounded} left \(\mG\)-harmonic function, it does matter if we consider \textit{positive} left \(\mG\)-harmonic function. 

We will also need to use the following classical moment estimates on the discretization measure. 

\begin{proposition}[Exponential Moment for the discretization measure]\label{p: exponential moment}
There exists \(\alpha>0\) such that the measures \( \mGamma ^g\) satisfy 
\[\int e^{\alpha l(\gamma)}  \mGamma ^g(d\gamma) <+\infty  \]
where \(l: \Gamma\rightarrow \mathbb N\) is the length function wrt to a finite system of generators on \(\Gamma\).  \end{proposition}


\begin{proof} In \cite[Paragraph 2.4]{DeroinDujardin} it was proved that the exponential moment with respect to the hyperbolic distance on the hyperbolic plane is finite for a lattice in \( \text{PSL} (2, \R)\). The same argument holds in the general setting of any complete locally symmetric manifold of finite volume as soon as we know that the diffusion semi-group has a spectral gap in \(L^2\), which in general is a consequence of the works  \cite{BekkaCornulier, Gelander Levit Margulis}. The corresponding finiteness of an exponential moment with respect to the length function is then a corollary of the estimates obtained by A. Lubotzky, S. Mozes and M. Ragunathan between word length and distance in symmetric space \cite{Lubotzky Mozes Ragunathan}, which works under the assumption that \(G\) has rank \(\geq 2\) and \(\Gamma\) is irreducible.
\end{proof}

\subsection{Random walks on semi-simple Lie groups} 

We recall the following known facts about random walks on real semi-simple Lie groups.  We will assume here that the random walk is defined by a probability measure  \(\mG\) on \(G\) which is compactly supported, absolutely continuous with respect to the Haar measure and right \(K\)-invariant. We refer to \cite[Thm. 10.9]{BenoistQuint}, and \cite[Thm.4.5.]{Aoun} for more general statements.

The following Theorems state that in the $KAK$ decomposition of $G$ associated to a Weyl chamber \(\W\), the Cartan projection of the $\mG$-random walk on $G$ is almost surely close to a line in $\liea$ given by a special direction $\central ^\W\in \liea$.  For a sequence \( \omega = (g_n)_{n\in \N} \in G^\N\) of increments,  we denote \( l_n (\omega) = g_n \ldots g_1\) and 
\[ l_n (\omega) = k'_\W (\omega, n )  e^{\kappa _\W (\omega, n )} k_\W (\omega,n) \]
the \(KAK\) decomposition of \(l_n(\omega)\) relative to the Weyl chamber \(\W\) (see \ref{KAKdecomp}).

\begin{theorem}[Thm. 10.9, \cite{BenoistQuint}]\label{centraldirection}
Let $G$ be a real semi-simple Lie group with finite center,  $\mG$ an absolutely continuous right \(K\)-invariant probability measure on \(G\), and \(\mathcal W\) a Weyl chamber. Then there exists $\central ^{\mathcal W} $ in the interior of $\mathcal W$ such that  for $\mG^{\otimes \N^*}$ a.e. sequence  $\omega = (g_1, g_2, ....)$ we have: 
 $$\lim_{n \to \infty}\frac{1}{n}\kappa _\W (\omega , n) \to \central ^\W,$$ 
 where $\kappa_\W $ is the Cartan projection as defined in \ref{KAKdecomp}.
\end{theorem}

Recall that the Furstenberg boundary is defined by  $M\backslash K$ which using the Iwasawa decomposition is naturally identified with $B := P_{\mathcal W} \backslash G$.  Recall that there is a probability measure $\mB$ which is the projection of Haar measure $\HaarK$  via the projection $K \to M \backslash K$. The probability measure $\mB$ is the unique $\mG$-stationary measure on $B$. We choose a metric $d_B$ in $B$ which is $K$ right-invariant.

By the choice of $\mG$, we have $k_\W(, n)_{*} (\mG^{\otimes \N^*}) = \mB$, which says that $k_\W(\omega, n)$ is distributed with the same distribution as $\mB$.  The sequence of random variables $k_\W(\omega, n)$ converges exponentially fast in $B$ to a random variable with the same distribution $\mB$, more precisely we have the following: 

\begin{theorem}[Thm.4.5, \cite{Aoun}]\label{expconvergencek}

There exists $\rho_1 \in (0,1)$  and  a measurable map $k_\W( , \infty):  \mG^{\otimes \N^*} \to B$  such that $k_\W(, \infty)_{*} (\mG^{\otimes \N^*} )= \mB$ and for $\mG^{\otimes \N^*}$ a.e. $\omega = (g_1, g_2, ....)$ we have: 

$$\int_{G^{\otimes \N^*}} d_B (k_\W(\omega,n), k_\W(\omega, \infty))  \ \mG^{\otimes \N^*}(d\omega)  < \rho_1^n$$ for $n$ sufficiently large.

\end{theorem}

\begin{corollary}\label{corollary to aoun's thm}

There exists $\rho_2 \in (0,1)$  and  a measurable map $k(  , \infty):  \mG^{\otimes \N^*} \to B$  with $k(, \infty)_{*} (\mG^{\otimes \N^*})= \mB$ such that for $\mG^{\otimes \N^*}$ a.e. $\omega = (g_1, g_2, ....)$, there is $N_{\omega}$ such that if $n \geq N_{\omega}$: 

$$d_B (k_\W(\omega,n), k_\W(\omega, \infty))  < \rho_2^n.$$

\end{corollary}

We next prove that we can track the \(\mG\)-random walk on \(G\) by the flow generated by the central element \(\central ^\W \in \text{Int} (\W)\), up to some subexponential error. Although this result is presumably well-known, we haven't found reference for this statement in the literature, and we provide a detailed proof.

\begin{theorem}\label{goodtracking} Assume that \(\mG\) is a compactly supported, absolutely continuous and right \(K\)-invariant probability measure on \(G\). Let $d_G$ be a fixed right-invariant Riemannian metric on $G$. For $\mG^{\otimes \N^*}$ a.e. $\omega = (g_1, g_2, ....)$ we have:   $$\lim_{n \to \infty} \frac{1}{n} d_G(   e^{n \central ^\W}k_\W(\omega, n), e^{n \central ^\W}k_\W(\omega, \infty)) = 0 .$$
\end{theorem}

\begin{proof} For simplicity we will omit the reference to the Weyl chamber \(\W\) in the notations. 

We need to prove that the element \(  \text{sub}_n= e^{n\central } k(\omega, n)k(\omega, \infty )^{-1} e^{-n\central } \) grows sublinearly with \(n\), or equivalently that the norm of the endomorphism \( \text{Ad} (\text{sub}_n) \in \text{End} (\lieg ) \) grows subexponentially with \(n\). 

We choose the determinations of \( k(\omega, n) \) modulo \(M\) in such a way that we have an exponential convergence to \( k(\omega, \infty)\). This is possible for \(\mG^\N\)-a.e. \(\omega= (g_n)_n\) by the use of the corollary to Aoun's Theorem \ref{corollary to aoun's thm} and does not affect the result since \( M\) commutes with \( e^{\central}\).

Let us consider the scalar product given by \( \text{Tr}(\theta (x) y)\) on \(\lieg\) and its associated norm \(\norm{x}^2= \text{Tr} (\theta(x) x ) \), where \(\theta\) is the Cartan involution. Notice that those functionals are invariant by the action of \(K\), and that the eigenspaces \(\lieg _\alpha\) are orthogonal. 

Given an element \( g \in G \), with \(KAK\) decomposition \( g = k_g '  \exp (\kappa (g) ) k_g\), let us introduce the quadratic form \( q_g (x) = \norm{\text{Ad}(g)(x) }^2 \). It is diagonal in the decomposition \( \lieg = \oplus _ \alpha (k_g)^{-1} (\lieg _ \alpha ) \), since on \( (k_g)^{-1} (\lieg _ \alpha )\) it is equal to \( e^{2 \alpha (\kappa(g))} \norm{\cdot}^2 \).  For \( H \in \text{End} (\lieg) \), we denote by \( H_{\alpha, \beta} \in \text{End}(\lieg_\beta, \lieg_\alpha) \) the endomorphisms defined according to the formulas \( H_{|\lieg_\beta} = \sum _\alpha  H_{\alpha, \beta} \) valid for any \(\beta\).

\begin{lemma}\label{l: exponential estimate in KAK}
Let \(g , g' \) in \(G\), and \(C\geq 1\) be a constant. Assume that \( q_g \leq C^2 q_{g'}\). Then  
\[  \norm{\text{Ad}( k_g k_{g'}^{-1})_{\alpha, \beta }} \leq  C e^{ \alpha (\kappa(g')) - \beta (\kappa(g)) } . \] 
\end{lemma}

\begin{proof}
Let \(x\in \lieg _\alpha \). We have \(q_{g k_g^{-1}} \leq C^2 q_{g' k_g^{-1}}\), so 
\[ \sum _\beta \exp (2\beta(\kappa(g))) \norm{\text{Ad}(k_g k_{g'}^{-1} )_{\alpha,\beta} (x) } ^2 = q_{g k_g^{-1}} (\text{Ad}(k_g k_{g'}^{-1})  (x)) \leq \]
\[ \leq C^2 q_{g' k_g^{-1}} ( \text{Ad}(k_g k_{g'} ^{-1}) x) = C^2 q_{g'k_{g'}^{-1} } (x) = C^2 \exp (2\alpha (\kappa(g'))) \norm{x}^2.\] 
In particular, 
\[ \exp (\beta(\kappa(g))) \norm{\text{Ad}(k_g k_{g'}^{-1} )_{\alpha,\beta} (x) } \leq C \exp (\alpha (\kappa(g'))) \norm{x}\]
and the lemma follows. 
\end{proof}

Introduce the notations \( \lambda _\alpha = \alpha (\central) \), so that \( \frac{1}{n} \alpha (\kappa (l_n))\rightarrow \lambda_\alpha\) and recall that  \(l_n = g_n \ldots g_1\) and \( k(\omega, n)= k_{l_n}\). Since the increments \(g_n\) belong to a fixed compact subset of \(G\), there is a constant \(C\geq 1\) such that, for every \(n\in \N\),  \( q_{l_n} \leq C^2q_{l_{n+1}}\) and \(q_{l_{n+1}}\leq C^2  q_{l_n}\). Lemma \ref{l: exponential estimate in KAK} applied to \( g=l_n\) and \(g'=l_{n+1}\) shows that, 
\[ \norm{\text{Ad}(k(\omega,n)  k(\omega, n+1) ^{-1} )_{\alpha,\beta} } \leq C e^{\alpha (\kappa(l_{n+1})) - \beta (\kappa(l_n)) } ,\] 
and in particular 
\[ \frac{1}{n} \log \norm{\text{Ad}(k(\omega, n)  k(\omega, n+1)^{-1} )_{\alpha,\beta} } \leq \lambda_\alpha - \lambda_\beta + O (1/n) .\]
Exchanging the roles of \(l_n\) and \(l_{n+1}\) gives 
\[\frac{1}{n} \log \norm{\text{Ad}(k(\omega, n+1), k(\omega, n)^{-1} )_{\alpha,\beta} } \leq \lambda_\alpha - \lambda_\beta + O (1/n) .\] 
Because \( \text{Ad}(k(\omega, n+1) k(\omega, n)^{-1})\) is orthogonal, we have 
\[ \norm{\text{Ad}(k(\omega, n) k(\omega, n+1)^{-1} )_{\alpha,\beta} } = \norm{\text{Ad}(k(\omega, n+1)  k(\omega, n)^{-1} )_{\beta,\alpha}},\] 
and we infer from all this that for every \(\alpha, \beta\)
\[ \frac{1}{n} \log \norm{\text{Ad}(k(\omega,n)  k(\omega, n+1) ^{-1} )_{\alpha,\beta} } \leq - | \lambda_\alpha - \lambda_\beta | + O (1/n) .\]

Denote \(E_n\in \text{End} (\lieg) \) the endomorphism satisfying \( 1+ E_n = \text{Ad}(k_{l_n} k_{l_{n+1}}^{-1} )\). Aoun's theorem shows that for \(\lambda >0\) small enough, \( \norm{E_n} = O ( e^{-n \lambda} ) \). Define the following numbers 
\[ \lambda_{\alpha, \beta}=\max (\lambda , |\lambda_\alpha -\lambda_\beta|-\lambda ) . \]
Notice that as soon as \(\lambda >0\) is sufficiently small, those numbers satisfy the inequalities 
\begin{equation}\label{eq: triangular inequalities} \lambda_{\alpha, \gamma} \leq \lambda_{\alpha, \beta} + \lambda_{\beta, \gamma} \text{ for every } \alpha, \beta, \gamma.\end{equation}
Moreover, the endomorphisms \( (E_n)_{\alpha, \beta} \in \text{End} (\lieg_\alpha, \lieg_\beta) \) satisfy the estimates \( \norm{( E_n )_{\alpha, \beta}} = O ( e^{ - n \lambda_{\alpha, \beta} }) \).  
In the sequel we will assume that \(\lambda >0 \) is chosen so small that this is the minimum of the \(\lambda_{\alpha, \beta}\)'s.

\begin{lemma}\label{l: multi-exponential estimates}
\(\norm{\left( \text{Ad}(k(\omega,n) k(\omega, \infty) ^{-1}) - 1\right)_{\alpha, \beta} } = O(e^{-n \lambda_{\alpha, \beta}} )\). 
\end{lemma}

\begin{proof}
Notice first that 
\[ \text{Ad}(k (\omega, n )  k(\omega, \infty) ^{-1}) = \lim_{k\rightarrow \infty} (1+ E_n) (1+E_{n+1} ) \ldots (1+ E_{n+k}).\]
Let \(C>0\) be a constant such that for every \(\alpha, \beta\), and every \(n\in \N\), we have 
\[| (E_n)_{\alpha, \beta} |\leq C e^{-n\lambda_{\alpha, \beta}} .\] 

We fix \( n\). For every integer \( k\geq 0\), let \( F_{n,k}\in \text{End} (\lieg )  \) defined by
\[ 1+ F_{n,k} = (1+E_n) (1+ E_{n+1}) \ldots (1+ E_{n+k}) .\]
Set \( C_{n,k} >0\) to be the smallest constant such that 
\[ \norm{ ( F_{n,k} )_{\alpha, \beta} } \leq C_{n,k} e^{-\lambda_{\alpha, \beta} n} \text{ for every } \alpha, \beta. \] 
To prove the lemma we just need to prove that \(C_{n,k}\) is bounded independently of \(k\) and \(n\) by a constant. Notice that \( C_{n,0} \leq C\) for every \(n\). 

From the relation \( F_{n,k}= F_{n,k-1} + E_{n+k} + F_{n,k-1} E_{n+k} \), and the inequalities \eqref{eq: triangular inequalities}, we get
\[ \norm{(F_{n,k})_{\alpha, \beta}} \leq e^{-\lambda_{\alpha,\beta} n} (C_{n, k-1} (1+Cd e^{-\lambda k } ) + C e^{-\lambda k}). \] 
Introducing the numbers 
\[ \gamma_k = 1+ C d e^{-\lambda k} \text{  and  } \varepsilon _k = C e^{-\lambda k}  \]
we get the estimates 
\[  C_ {n, k} \leq \gamma_k C_{n, k-1} + \varepsilon_k .\]
Denoting by \( \Gamma \) the infinite product \( \Gamma = \gamma_1 \ldots \gamma_k \ldots \), we get by induction on \(k\) 
\[ C_{n, k} \leq \frac{\Gamma C}{1-e^{-\lambda} }  < \infty \]
and this ends the proof of the lemma since the term on the right hand side of the last inequality does not depend neither on \(n\) nor \( k\). 
\end{proof}

To conclude, notice that 
\[ \left( \text{Ad}(e^{n \central }) (\text{Ad}(k(\omega,n) k(\omega,\infty)^{-1}) -1) \text{Ad}(e^{-n \central })  \right) _{\alpha, \beta} =\]
\[=e^{n(\lambda_\alpha - \lambda_\beta)} (\text{Ad}(k(\omega,n) k(\omega,\infty)^{-1})  -1) _{\alpha, \beta}   \] 
and apply Lemma \ref{l: multi-exponential estimates} to get that  
\[ \norm{\text{Ad}(\text{sub}_n)_{\alpha, \beta} }= \norm{\text{Ad} \left( e^{n \central } k(\omega,n) k(\omega,\infty)^{-1} e^{-n \central } \right)_{\alpha, \beta} } =\]
\[= \norm{\left(1 +  \text{Ad}(e^{n \central }) (\text{Ad}(k(\omega,n) k(\omega,\infty)^{-1}) -1) \text{Ad}(e^{-n \central }) \right) _{\alpha, \beta}} \leq \]
\[\leq  1+ O\left( e^{n(\lambda_\alpha - \lambda_\beta- \lambda_{\alpha, \beta} )} \right) = O ( e^{\lambda n} ).\]
This being valid for every sufficiently small \(\lambda>0\), the theorem follows.
\end{proof}


\begin{section}{Some facts from Ergodic Theory}\label{ergodictheory}

In this section, there are no new results. We only review some facts that will serve in our argument in the next sections.

\subsection{Measurable partitions subordinate to locally free actions}

Suppose $\mu$ is a Radon measure on a Polish space $Y$ (which we do not assume to be compact) and let $\mathcal{B}$ be the $\sigma$-algebra of Borel sets. Given a partition $\mathcal{P}$, for a point $y \in Y$, $\mathcal{P}_y$ denotes the element of $\mathcal{P}$ containing $y$. A partition of the measure space $(Y, \mathcal{B}, \mu)$ is said to be measurable if the sets \(\mathcal{P}_y\) are measurable, up to measure zero, the quotient space $Y /\mathcal{P}$ is separated by a countable number of measurable sets. See \cite{CliKa}  for a discussion about measurable partitions.

We recall the notion of desintegration of measures. Let $Z \subset Y$ be a measurable set of positive finite measure of  $Y$ and let $\mu|_Z$ be the restriction of $\mu$ to $Z$. For any measurable partition $\Pi$ of $(Z, \mathcal{B}\cap Z, \mu|_{Z})$, there exists a family of probability measures $\mu_z^{\mathcal{P}}$ (which depend measurably in the atoms of $\mathcal{P}$) such that $\mu_z^{\mathcal{P}}$ is supported in the atom $\mathcal{P}_z$ for $\mu|_z$ a.e. $z \in Z$ and we have the following equality:

$$\mu(C) = \int_Z  \mu^{\mathcal{P}}_z(C)   \ d \mu (z)$$ for every $\mu$ measurable set $C$.

Suppose $H \times Y \to Y$ is a continuous action of a locally compact second countable topological group $H$  (which for us will be a Lie group). We equip $H$ with a right invariant metric and let $B_r^H(t)$ be the ball of radius $r>0$ centered at $t \in H$. Suppose the action is locally free in the following sense: For every compact set $Z \subset Y$, there exists $\eta > 0 $, such that for all $y \in K, t \in B_{\eta}^H(e)$, if $ty = y$, then $t = e$.

\begin{definition} Let $x \in X$. A set $A \subset Hx$ is an open $H$-plaque if for every $a \in A$, the set $\{t : ta \in A\}$ is open and bounded.
\end{definition}

Given two measures $\mu, \nu$ we write $\mu \propto \nu $ if there exists $c > 0$ such that $\mu = c\nu$. We now recall the notion of leaf-wise measures for a measure $\mu$ in $Y$ with a locally free action. The following theorem can be found in \cite[Thm.6.3]{EinsLind} in the language of $\sigma$-algebras, we refer the reader to \cite[Sec.6]{EinsLind} and the reference therein for more details.

\begin{theorem}[Leaf-wise measures and subordinate partitions]\label{leafwise} Assume the conditions above and moreover that for $\mu$ a.e. $y \in Y$, $\text{Stab}_H(y) := \{t \in H: ty = t\} = \{e\}$. Then there exists  a collection $\{\mu_y^H\}_{y \in Y'}$ of Radon measures on $H$ called the leaf-wise measures which are determined uniquely up to proportionality by the following properties:

\begin{enumerate}

\item The domain $Y' \subset Y$ has full measure. 
\item For every $f \in C_c(H)$, the map $y \to \int f d\mu_y^H$ is Borel measurable.
\item\label{nicepartition} Suppose that $Z\subset Y$ is a compact set and there is a measurable partition of $Z$, such that for  $\mu$-a.e. $y \in Z$, the atom  $\mathcal{P}_y$ is an open $H$-plaque. Then for $\mu$-a.e. $y \in Z$:

$$ (\mu|_Z)^{\mathcal{P}}_y \propto \mu_y^H|_{\mathcal{P}_y} $$

\item The identity element of $e \in H$ belongs to the support of $\mu_y^H$ for $\mu$-a.e. $y \in Y$.
\end{enumerate}

\end{theorem}

\begin{remark}
\begin{enumerate}
\item Observe that the measure $\mu$ is not required to be $H$-invariant, $\mu_Y^H$ is generally not a probability measure and $Y$ is not assumed to be compact.
\item\label{nicepartition2} A partition $\mathcal{P}$ satisfying Condition \ref{nicepartition} will be called a partition of $Z \subset Y$ subordinate to $H$-orbits. Such a partition always exists. See proof of \cite[Thm.6.3]{EinsLind}, it also follows easily from Lemma \ref{l: local product structure}.
\end{enumerate}
\end{remark}

The following proposition also follows from the proof of \cite[Thm.6.3]{EinsLind}, see \cite[Pr. 6.28]{EinsLind}.

\begin{proposition}[T-invariance implies Haar leafwise measures]\label{leafwisehaar} Assume the hypothesis of Theorem \ref{leafwise}. Then, a Radon measure $\mu$ on $Y$ is $H$-invariant if and only if for $\mu$ a.e. $y \in Y$ the measure $\mu_y^H$ coincides with the left Haar measure on $H$.
\end{proposition}

\subsection{Ergodic decomposition}

Suppose $H \times X \to X$ is a continuous action of a second countable locally compact group $H$  (which for us will be either a Lie group or a discrete countable group) on a compact metric space $X$. Let $\mathcal{B}$ be the $\sigma$-algebra of Borel sets and let $\mathcal{M}(X)$ be the set of probability measures on $X$ with respect with the  weak-$*$ topology. Let $\mathcal{M}^{H}(X)$  be the set of probability measures on $X$ which are $H$-invariant. $\mathcal{M}^{H}(X)$ is a  closed subset of $\mathcal{M}(X)$. A probability measure $\mathcal{M}^{H}(X)$ is ergodic if it cannot be written as a convex combination of probability measures $\mathcal{M}^{H}(X)$ in a non-trivial way.

We have the following version of the ergodic decomposition Theorem for $H$-actions:

\begin{theorem}[Ergodic decomposition for invariant measures]\label{ergodicdecompinv}

Suppose $\mu$ is a $H$-invariant probability measure on $X$. Then there exists a unique probability measure $\lambda$ on $\mathcal{M}^{H}(X)$  and a measurable map $\pi: (X, \mu) \to (\mathcal{M}^H(X), \lambda)$ such that for $\mu$-a.e. \(x\in X\), the projection $\pi(x)$ is a $H$-ergodic invariant measure and such that for every Borel function $f$ on $X$ we have:

$$\int_X f d \mu =  \int_{\mathcal{M}^{H}(X)} \left( \int_X f d \pi(x) \right) d \lambda .$$

\end{theorem}

There is a similar theorem for actions preserving a stationary measure which we recall.  Let $\nu$ be a probability measure on a locally compact group $G$. A probability measure $\mu$ on $X$ is said to be $\nu$-stationary if $\nu * \mu = \mu$ ($*$ denotes convolution). We let $\mathcal{M}^{\nu}$ be the set of $\nu$-stationary measures on $X$.

\begin{theorem}[Ergodic decomposition for stationary measures]
Suppose $\nu$ is a probability measure on $G$, and $\mu$ is a $\nu$-stationary measure for an action of $G$ on a compact metric space $X$. Then there exists a unique probability measure $\lambda$ on $\mathcal{M}^{\nu}(X)$ and a measurable map $\pi: (X, \nu) \to (\mathcal{M}^\nu(X), \lambda)$ such that for $\mu$-a.e. \(x\in X\),  $\pi(x)$ is an ergodic $\nu$-stationary probability measure and such that for every Borel function $f$ on $X$ we have:

$$\int_X f d \mu =  \int_{\mathcal{M}^{\nu}(X)} \left( \int_X f d \pi(x) \right)d \lambda .$$
\end{theorem}

See \cite[Lemma B.11]{Brownetal}. 

\subsubsection{Random walks and the suspension space}\label{suspensionrandomwalk}

Suppose that $\nu$ is a probability measure on $G$, $G \times X \to X$ is an action of $G$ on a compact metric space $X$ and $\mu$ is a $\nu$-stationary probability measure on $X$, there is an associated suspension space $$(S(X),\mu_{S(X)}) := (G^{\N^*} \times X, \nu^{\otimes \N^*} \otimes \mu)$$ and a measure preserving map $$S ((g_n)_n , x):= ( (g_{n+1})_n, g_1(x)) .$$

   The following result is classical, we refer to \cite[Thm.2.1.]{Kifer} and references therein.

\begin{theorem} [Random Ergodic Theorem]  \label{t: random ergodic theorem} The measure $\mu$ is ergodic as a $\nu$-stationary measure on $X$  if and only if the probability measure $\mu_{S(X)}$ is $S$-ergodic. 
\end{theorem}

From this and Birkhoff's ergodic Theorem we have:

\begin{corollary} Suppose that $\nu$ is a probability measure on $G$, $G \times X \to X$ is an action of $G$ on a compact metric space $X$ and $\mu$ is a $\nu$-stationary ergodic probability measure on $X$. Then for any $f \in C(X)$ and $ \nu^{\otimes \N^*}$ a.e. sequence $(g_n)_n$ and $\mu$ a.e. $x \in X$, if we let $l_1 := g_1$ and $l_n := g_n l_{n-1}$, then:  
$$\frac{1}{N}\sum_{n=1}^N  f(l_n(x)) \to \int_X f d \mu .$$

\end{corollary}

\end{section}


\section{A laminated \(G\)-space by Lipschitz oriented unidimensional manifolds}\label{almostperiodicspace}

In this section, we associate to a left-orderable lattice in a semi-simple Lie group \(G\) with finite center, a laminated \(G\)-space by Lipschitz unidimensional manifolds. This space appears as the suspension of a \(\Gamma\)-space called the almost-periodic space, whose contruction relies on the combination of the works \cite{almost-periodic} and \cite{DKNP}. We begin by an expository section on laminations by Lipschitz manifolds.

\subsection{Laminated structure}

In this section we prove that any locally free flow has a local product structure. This statement is presumably well-known, but since we did not found any reference we include a proof, filling a gap in  \cite{almost-periodic}.  

\subsubsection{Definition} \label{sss: lamination Lipschitz}

A lamination by Lipschitz manifolds (of dimension \(r\)) of a topological space \(Z\)  is the data of a maximal set of homeomorphisms \( \varphi_i : U_i  \rightarrow B \times T_i\), where \(\{U_i\}_i \) is a covering of \( Z\) by open subsets, \(B= B^r\subset \R ^r\) an open ball in the standard euclidean space \(\R^r\), and where the homeomorphisms \( \varphi _i \circ \varphi_j^{-1} : \varphi _ i (U_i \cap U_j) \rightarrow \varphi_j (U_i \cap U_j)\) preserve the local fibrations by balls \( B \times T_i \rightarrow T_i\) -- namely they are of the form \( (x_j, \tau _j ) \mapsto (x_i(x_j,\tau_j), \tau_i(\tau_j)) \) --  and satisfy that  each \( x_j, \tau_j\) in the domain of \(x_j\) has a neighborhood \(\mathcal V\) on which we have,  for a certain constant \( \xi_{\mathcal V} \), 
\[ d(x_i(x_j ', \tau_j ') , x_i (x_j '',\tau_j ' ) ) \leq \xi_{\mathcal V} \ d( x_j , x_{j'} ) \]
for every \( x_j ', x_j '', \tau_j'\)  such that \((x_j' , \tau_j') ,\  (x_j'', \tau_j ') \in \mathcal V\),
where \(d\) is the euclidean distance on \(B\). The lamination is oriented if the changes of coordinates preserve orientation along the plaques \(B\times \{t_i\}\). Its dimension is \(r\).  

\begin{lemma}[Flow boxes]\label{l: local product structure}
Let \(Z\) be a completely Hausdorff space (namely any two distinct points can be separated by a continuous function), and \( H:\R \times Z\rightarrow Z\) a locally free action of \(\R\) on \(Z\). Then every point \(x\in Z\) admits a neighborhood \(U\) which is homeomorphic the product \( I\times T\) of a neighborhood \(0\in I\subset \R \) of the origin in  \(\R\) with a topological space \(T\) (called transversal space) in such a way that the action of \(\R\) on \(U\) is given by \( t \cdot (s,\tau) = (t+s, \tau)\) in a neighborhood of \( (0,x)\in \R\times Z\). The changes of coordinates are of the form 
 \( (s,\tau) \mapsto (s'= s+ s_\tau  , \tau'= \tau'(\tau)) \) where \( \tau\mapsto s_\tau\) and \(\tau\mapsto \tau'(\tau)\) are continuous. In particular, the set of all these coordinates provides \(Z\) with a structure of an oriented lamination by Lipschitz manifolds of dimension one.
\end{lemma} 

The neighborhood \( U\) and the coordinate system \((s, \tau)\) provided by Lemma \ref{l: local product structure} will be called a \textit{flow box}.


\begin{proof} 
Introduce the space \(C^1_H (Z)\) whose elements are continuous functions \(f: Z\rightarrow \R\) that are \(C^1\) along the \(H\)-orbits, and whose derivative \( H\cdot f:= \lim _{t\rightarrow 0} \frac{d f\circ H(t,.)}{dt} \) is continuous. 

The existence of such functions can be established using convolutions. Namely, let \( m_\R\) be a probability measure on \(\R\) having compact support and a smooth density with respect to Lebesgue measure on \(\R\). For every \(f\in C^0 (Z)\), the function \(F: Z\rightarrow \R\) defined by \( F(x) := \int _\R f(H(t ,x) ) \  m_\R (dt) \) is of class \(C^1\). 

Notice that if \(m_\R\) weakly converges to the Dirac measure at the origin \(0\in \R\), then \( F \) converges to \(f\) uniformly on compact subsets. Hence, if we choose two distinct points in the same \(H\)-orbit, and an \(f\) taking distinct values at these two points, the function \(F\) have the same property as soon as \(m_\R\) is close to the Dirac mass at \(e\). In particular, there exists a point in that \(H\)-orbit at which the derivative \( H \cdot F\) does not vanish. Since that \(H\)-orbit is homogeneous, we have proved the following 

\vspace{0.2cm}

\textit{Given \(x\in Z\), there exists a function \(F\in C^1_H(Z)\) such that \( H\cdot F (x) \neq 0\).}

\vspace{0.2cm}

We can then assume that \( F(x)=0\) and \(H\cdot F=1\).  Take a neighborhood of \(x\) on which the derivative \( H\cdot F\geq 1/2\). Along each piece of trajectory passing sufficiently close to \(x\), there is a unique point where the function \(F\) vanishes. Hence, the Lemma is proved by defining the transversal space \( T\) as being the intersection of \( \{ F=0\}\) with a sufficiently close neighborhood of \(x\), and \(V\) a sufficiently close neighborhood of \(0\in \R\).\end{proof}

\subsubsection{Measurable volume forms and family of measures} \label{sss: measurable volume forms}

Let \(\mathcal L\) be a lamination by Lipschitz manifolds that is oriented. Recall that a locally Lipschitz map between two euclidean domains is differentiable almost everywhere with respect to the Lebesgue measure. Hence, we have a well defined concept of measurable volume form on \(\mathcal L\). Such an object writes in the coordinates \( (x_i, \tau_i)\) as \( \omega = \omega_i dx_i^r\), where \(dx_i^r\) is the euclidean volume of \(\R ^r\) and \( \omega_i : B^r \times T_i \rightarrow \R\)  a positive function, which is measurable, with bounded logarithm, and well-defined modulo changing its values on a set which intersects every plaque \( B^r \times \{\tau_i\}\) on a set of Lebesgue measure zero. These expressions are related one to the other by the change of coordinates rules for volume forms: \( \omega_j (x_j, \tau_j)  = \omega_j (x_j(x_i,\tau_i), \tau_j(\tau_i) ) \cdot \frac{dx_j^r}{dx_i^r}  \). If the lamination is oriented, it is always possible, using partition of unity, to construct a measurable volume form. In the example of the one dimensional lamination associated to a locally free flow (see Lemma \ref{l: local product structure}), the time form expressed by \(ds\) provides an example of measurable volume form. 

A measurable volume form \( \omega \) on a lamination \(\mathcal L\) induces a family of Radon measures \( \{\mu_L ^\omega\} _{L \text{ leaf of } \mathcal L}\) on the leaves of \(\mathcal L\), defined by the formula
\[ \mu_L^\omega (A) =  \int _{\varphi_i(A) \cap B^r\times \{\tau_i\}} \omega_i (x_i, \tau_i) dx_i^r, \]
for any set \(A\subset L\) contained in a plaque \( \varphi_i^{-1} (B\times \{\tau_i\}) \). If the set \(A\) is not contained in a plaque, the definition of its measures requires a partition of unity. We leave this to the reader. When the dependance on \(\omega\) is clear, we will omit it.

We will call a family of leafwise Radon measures \(  \{\mu_L\}_{L\text{ leaf of } \mathcal L}\) continuous (resp. measurable) if in every chart \( (x_i, \tau_i):U_i\rightarrow B\times T_i \), and any continuous function \(f\in C^c(U_i) \) with compact support (resp. every bounded measurable function), the function \( \tau_i \mapsto \int_{\varphi_i^{-1} (\tau_i) } f_{|\varphi_i^{-1} (\tau_i)} d\mu_{L_{\tau_i}} \) is continuous (resp. measurable), where \(L_{\tau_i}\) is the leaf of \( \mathcal L\) containing the plaque \( \varphi^{-1} (B\times \{\tau_i\}) \). This definition does not depend upon the chosen coordinate system. 

Assume now that the lamination \(\mathcal L\) is one dimensional, oriented, and without circular leaf. Let \(\omega\) be a measurable volume form on \(\mathcal L\) and \( \{\mu_L\} _{L\text{ leaf of } \mathcal L} \) the induced family of leafwise measures. Let \( \{ d_L \} _{L\text{ leaf of } \mathcal L}\) be the family of distances defined by 
\[ d_L (x,y) = \mu_ L ( [x,y] ) \text{ for every } x,y\in L\]
where \([x,y]\subset L\) is the interval between \(x\) and \(y\) in \(L\) (recall that \(L\) is assumed not to be a circle). We have the property that in a chart \(\varphi_i : U_i \rightarrow B\times T_i \) of the lamination structure, the function 
\begin{equation} \label{eq: regularity distance} (x,y, \tau_i ) \in B\times B\times T_i \mapsto d_{L_{\tau_i}}(\varphi_i^{-1} (x,\tau_i), \varphi_i^{-1} (y, \tau_i)) \in [0, +\infty)\end{equation} 
is measurable, where in this formula \( L_{\tau_i}\) is the leaf of \(\mathcal L\) containing the plaque \( \varphi_i^{-1} (B\times \{\tau_i\})\). In this situation, we will adopt the following notation: given a point \( x\in Z\) and a real number \( s \in \R\), we denote by \( x+ s\) the unique element such that \( d_{L} (x, x+s) = |s| \), and if \(s\neq 0\), the orientation from \( x \) to \(x+s\) correspond to the orientation of \(\mathcal L\) if \(s\) is positive and to the reversing orientation if \(s\) is negative. Analogously, given two points \(x,y\) lying in the same leaf \(L\), we denote by \( y- x \) the unique real number \( s\in \R\) such that \( x+ s= y\). 

To conclude these generalities, assume that the family of measures \( \{\mu_L\}_{L \text{ leaf of } \mathcal L} \) is continuous. Then the family of distances \(\{d_L\}_{L\text{ leaf of } \mathcal L}\) is also continuous, in the sense that the functions \eqref{eq: regularity distance} are continuous. In particular, one can then define a continuous local flow acting on the lamination freely, by the formula \( T_Z ^t (x) = x+t\), for \( (t, x) \) in a neighborhood of \( \{0\}\times Z\) in \(\R\times Z\). Note that the Lipschitz lamination structure associated to this flow (see Lemma \ref{l: local product structure}) is the original Lipschitz structure, and the volume form \( \omega \) is the form \(ds\) in the atlas given by Lemma \ref{l: local product structure}.

\subsubsection{Transverse invariant measures} 

\begin{definition} Given a lamination \( \mathcal L\) by Lipschitz manifolds of a space \(Z\), a \textit{transverse invariant measure} is a family \(\{ \mu_{T_i}\}_{i}\) of Radon measures on the transversal spaces \(T_i\)'s given by the lamination structure (with the notations of section \ref{sss: lamination Lipschitz}) which is such that, for every \(i,j\), the map \( \tau_j \mapsto \tau_i(\tau_j) \) sends locally \( \mu_{T_j}\) to \( \mu_{T_i}\). It is called ergodic iff for every measurable \(\mathcal L\)-saturated subset \( E\subset X\), there exists \( \delta\in \{0,1\}\) such that for every \(i\), the set of \(t_i\in T_i\)'s such that \( \varphi_i^{-1} (B^r\times \{t_i\}) \subset E\) has measure \(\delta\). \end{definition}

This concept has been introduced by Plante \cite{Plante} and does not need the Lipschitz assumption to be defined. Notice however that given a transverse invariant measure \(m^t_{\mathcal L}\) on a lamination \(\mathcal L\) by Lipschitz manifolds of a space \(Z\), and a measurable family of leafwise measures \(\mu _{\mathcal L} ^l = \{\mu _L \}_{L\text{ leaf of } \mathcal L}\), one can define a Radon measure \(m_Z\) on \(Z\): for each measurable subset \(A\subset U_i\) contained in the domain of the laminated coordinates chart \(\varphi_i\), set  \(m_Z(A) := \int _{ T_i } 
\mu_L (A \cap \varphi_i^{-1} (B\times \{\tau_i\}) \  \mu_i ( d\tau_i) \). We denote this measure by \( \mu_Z= \mu_{\mathcal L}^l  \otimes \mu_{\mathcal L}^t\). 

\begin{lemma} \label{l: invariant and transverse invariant} Let  \(T_Z:\R\times Z\rightarrow Z\) a locally free flow acting on a completely Hausdorff space \(Z\), and \(\mathcal L\) the lamination associated to it as in Lemma \ref{l: local product structure}. We denote by \(\mu^l_{\mathcal L}= \{\mu _L\} _{L \text{ leaf of  }\mathcal L}\) the Lebesgue measure in the time parametrization of the \(T_Z\)-trajectories. Then, the map \( \mu ^t_{\mathcal L}\mapsto \mu ^l _{\mathcal L}  \otimes \mu^t_{\mathcal L}\)  is a bijection between the set of transverse invariant measures on \(\mathcal L\) and the set of Radon measures on \(Z\) invariant by \( T_Z\). This correspondance preverves the concepts of ergodicity on both sides. \end{lemma} 

\begin{proof}  The measure \(\mu ^l _{T_Z}  \otimes m_{\mathcal L}\) is invariant by the flow \(T_Z\), so the map of the lemma is well-defined.  We just need to construct the inverse map. Given a Radon measure \(\mu_Z\) on \(Z\) which is \(T_Z\)-invariant, its desintegration with respect to the plaques in a laminated chart writes \(\mu_Z = \int _{T_i}  \mu _{B^1\times \{\tau_i\}} \ \mu_{T_i} (d\tau_i)\) for some measure \(\mu_{T_i}\) on \(T_i\), and some measurable family \(\{ \mu_{B^1\times \{\tau_i\}}\}_{\tau_i\in T_i}\) of measures supported on \(B^1\times \{\tau_i\}\). By unicity of the desintegration, and the invariance of \( \mu_Z\) by \(T_Z\), the measures \(\mu_{B^1\times \{\tau_i\}}\) are the Lebesgue measure in the time parametrization of the \(T_Z\)-trajectories multiplied by some constant \(c(\tau_i)\). Up to multiplying \(\mu_{T_i}\) by the function \(\tau_i\mapsto 1/c(\tau_i)\), we can assume that \(c(\tau_i)=1\). The family of measures \(\{\mu_{T_i}\}_i\) defines the desired transverse invariant measure \( \mu^t_{\mathcal L}\).\end{proof}


\subsection{The almost-periodic space} 

\subsubsection{Construction}

Let \(\Gamma\) be a finitely generated group, and \(l:\Gamma\rightarrow [0,\infty)\) be the length function on \(\Gamma\) with respect to a finite set of generators $S$. Let \(\mGamma\) be a symmetric probability measure on \( \Gamma\) of full support, and which  satisfies the following \(L^2\) moment condition with respect to the length function:
\[ \int l(\gamma)^2 \ \mGamma (d\gamma) <\infty.\]
Discretization measures on a lattice satisfy such a moment estimates. 

\begin{theorem}\label{t: almost-periodic space}
Suppose that \(\Gamma\) is left-orderable. Then there exists a non empty compact space \(\Der\), a flow \( T_{\Der}=\{ T^t _{\Der} \}_{t\in \R}\) acting freely on \(\Der\), and an action of \(\Gamma\) on \(\Der\), such that 
\begin{enumerate}
\item \(\Gamma\) has no global fixed point, 
\item each \(T_{\Der}\)-trajectory is \(\Gamma\)-invariant (with the freeness of \(T_{\Der}\) this produces a unique function \( t: \Gamma \times \Der \rightarrow \R\) such that \( \gamma (p) = T_{\Der}^{t( \gamma, p)} \) for every \((\gamma, p)\in \Gamma\times \Der\)),
\item  for each \(p\in \Der\) we have the following zero mean displacement property:
\[ \int_\Gamma t(\gamma, p) \ \mGamma (\gamma) = 0.\]
\end{enumerate}
Moreover, we can assume that \(T_{\Der}\) is minimal -- namely that \(T_\Der\)-trajectories are dense in \( \Der \) --  and this is what we will do in the sequel.
\end{theorem}

\begin{proof}
Let \(S\subset \Gamma\) be a symmetric finite subset that generates \(\Gamma\) and that contains \(e\). Fix some constants \( \xi\geq 1\), and \( \delta , \Delta  >0\), and let  \( \Der '= \Der '_{\xi , \delta, \Delta}\)  be the set of representations \( \rho : \Gamma \rightarrow \text{Homeo}^+ (\R)\) satisfying the following four conditions: 
\begin{enumerate}
\item  {\bf Bounded displacement from above:} 
\( |\rho (\gamma) (x) -x |\leq \Delta \) for every \(x\in \R\),
\item {\bf Bounded displacement from below:}
\( \int_\Gamma (\gamma (x) - x ) ^2 \ \mGamma (d\gamma) \geq \delta ^2 \)  for every \(x\in \R\), 
\item {\bf Lipschitz regularity}: for every \(\gamma\in S\), \(\rho (\gamma)\) is \(\xi\)-Lipschitz,
\item {\bf Harmonicity:} \( \int \rho (\gamma) (x) \ \mGamma (d\gamma) = x\) for every \(x\in \R\). 
\end{enumerate} 
Notice that the integral in the second condition converges if we know that the first condition is fulfilled, because of the \(L^2\) moment condition satisfied by \(\mGamma\).  

The space \( \Der '\) embeds in \( \text{Homeo}^+ (\R) ^S \) by the map \( \rho \mapsto (\rho (\gamma) ) _{\gamma\in S}\). We equip \(\Der'\) with the topology induced by the compact open topology on \( \text{Homeo}^+ (\R) ^S \) restricted to the image of this embedding. An application of Arzela-Ascoli's theorem shows that \( \Der '\) is a compact set. 

We define a flow \( T_{\Der '} =\{ T_{\Der '}^t\}_{t\in \R} \) acting on \(\Der'\) by the formula \[ T _{\Der '}^ t (\rho) := \tau _{-t} \circ \rho \circ \tau  _{t}\] where \( \tau _t\) is the translation \(s\mapsto s+t\) on \(\R\), and the \(\Gamma\)-action by 
\[ \gamma (\rho) := T _{\Der '}^{\rho(\gamma) (0) } (\rho). \]
Verifying that this is an action is painful but straighforward.  We refer to \cite[Lemma 2.2]{almost-periodic} for more details on this. These actions satisfy  that each \( T_{\Der '}\)-trajectory is \(\Gamma\)-invariant, and the zero mean displacement property.  

The flow \( T_{\Der '} \) is not necessarily free however. In that situation, we apply the following trick: we consider a flow \( H=\{H^t\}_{t\in \R} \) acting freely on a non empty compact space \( Y \), and define a flow \(T_{\Der} \) and a \(\Gamma\)-action on a space \( \Der\) by the following formulas:
\begin{itemize} 
\item \(\Der : = \Der' \times Y\), 
\item \( T_{\Der }^t ( \rho , y) := (T_{\Der '} ^t  (\rho), H^t (y))\),
\item \( \gamma (\rho, y) := (\gamma (\rho), H^{\gamma(0)}(y))=(T_{\Der '}^{\gamma(0)}(\rho), H{\gamma(0)}(y))\).
\end{itemize}
This flow \(T_{\Der }\) is free since it projects equivariantly to the free flow \(H\). All the properties of Theorem \ref{t: almost-periodic space} are satisfied for the space \(\Der\) together with its flow \(T_{\Der}\), apart from its non emptyness. 

So let us explain that we can choose appropriately the  constants \( \xi\geq 1\) and \( \delta, \Delta >0\) so that  the space \( \Der'_{\xi, \delta, \Delta} \) is non empty. For every positive integer \(n\), set \[ m_n := \frac{(\mGamma)_{| S^n}}{\mGamma (S^n)}.\]  Those are symmetric, finitely supported probability measures on \(\Gamma\). Let \(n_0\) be an integer such that \(\mGamma (S^n) \geq 1/2 \) for every \(n\geq n_0\).  

For every \( n\), it has been established in  \cite{DKNP} that there exists a non trivial action \( \rho _n : \Gamma \rightarrow \text{Homeo}^+ (\R) \) which is such that for each \( x\in \R\), we have 
\begin{equation}\label{eq: harmonic action} \int \rho_n (\gamma)(x) \ m_n (d\gamma)   = x \text{ for every } x\in \R.\end{equation}
Such an action is Lipschitz, more precisely the transformations \( \rho_n (\gamma)\) are \( 1/m_n(\gamma)\)-Lipschitz. In particular, for \(n\geq n_0\), and \(\gamma \in S\),  \(\rho_n(\gamma)\) is \(\xi\)-Lipschitz for \(\xi:= \max _{\gamma\in S}  2/ \mGamma (\gamma)\).

For every homeomorphism \( h\in \text{Homeo}^+(\R)\), and every \(x\in \R\), define \(K(h, x) =\int_{[x, h^{-1} (x) ]} | h(s) - s | \ ds\). Recall the following observation of Victor Kleptsyn: the quantity \( \int K(\rho_n (\gamma) , x) \ m_n (d\gamma) \)  does not depend on \(x\in \R\) for every action satisfying \eqref{eq: harmonic action}, see \cite[Proof of Proposition 8.1]{DKNP}. So up to conjugating \(\rho_n\) by an affine transformation, we can assume that the integral \(\int K(\rho_n (\gamma), x) \ m_n (d\gamma)\) is identically equal to \(1\) for \(x\in \R\). Since for a \(\xi\)-bilipschitz homeomorphism \(h\), we have \[\frac{(h(x)-x)^2 }{\xi} \leq K(h,x ) \leq \xi (h(x) - x)^2 ,\]
we infer first that, for \(\delta = 1/\xi\), we have
\[ \int _ \Gamma (\rho_n (\gamma) (x) - x ) ^2 \ m_n (d\gamma) \geq \delta \text{ for every } x\in \R, \]
and that, setting \( \Delta= \max _{\gamma \in S} \sqrt{2/\mGamma (\gamma)} \), we have that for each \( \gamma \in S\), each \(n\geq n_0\) and each \( x\in \R\)
\[ |\rho_n (\gamma) (x) - x| \leq \Delta. \]

From the sequence \(\rho_n\) choose a subsequence that converges uniformly to a representation \( \rho : \Gamma \rightarrow \text{Homeo}^+ (\R)\). This is possible by applying Arzela-Ascoli, since each \( \rho_n(\gamma)\) for \(\gamma\in S\) is Lipschitz and of bounded displacement  with constants that do not depend on \(n\). The \(L^2\)-moment assumption on \(\mGamma\) shows that the limiting representation \( \rho \) belongs to \(\Der '_{\xi,\delta, \Delta}\), so this latter set is non empty, and Theorem \ref{t: almost-periodic space} is proved.   
\end{proof}


\subsubsection{The Lyapunov cocycle on the almost-periodic space}

Denote by \(\mathcal C_{\Der} \subset \Der\) the set of points \(z\) such that for every \(\gamma \in \Gamma\), the following derivative  
\[ D_{\mathcal L_\Der} (\gamma, z) = \lim _{t\rightarrow 0} \frac{\gamma(z+t) - \gamma(z) }{t}  \]
exists and is positive. 

Notice the following properties: 

\begin{enumerate}

\item  The intersection of \(\mathcal C_{\Der}\) with \textit{any} leaf \(L\) of \(\mathcal L_\Der\) has full \(\mL\)-measure (since in time parametrization of the leaves by the almost-periodic flow \(T_{\Der}\), the \(\Gamma\)-action is Lipschitz).

\item  \(\mathcal C_{\Der}\) is a measurable \(\Gamma\)-invariant set. 

\item \label{eq: cocycle Der} \(D_{\mathcal L_\Der}\) is a multiplicative cocycle, namely \(D_{\mathcal L_\Der} (\gamma '\gamma , z) = D_{\mathcal L_\Der} (\gamma ', \gamma (z)) \cdot D_{\mathcal L_\Der} (\gamma, z) \) for every \( \gamma, \gamma ' \in \Gamma\) and every \( z\in \mathcal C_{\Der} \). 

\item  For every \( z\in \mathcal C_{\Der} \), the function \( D_{\mathcal L_\Der} (\cdot, z) \) is left \(\mGamma\)-harmonic. 

\end{enumerate}

The first two properties are straighforward, and the last one is a consequence of  3. in Theorem \ref{t: almost-periodic space}. 

\begin{lemma}\label{l: radon-nikodym}
For every \( T_{\Der}\)-invariant measure \(\mDer\) on the almost-periodic space \(\Der\), the set \(\mathcal C_{\Der}\) has full \(\mDer\)-measure, and moreover for \(\mDer\)-a.e. \(z\in \Der\) and every \(\gamma\in \Gamma\), the Radon-Nikodym derivative \( \frac{\gamma^{-1}_* \mDer}{\mDer} (z) \) equals \(D_{\mathcal L_\Der} (\gamma, z)\).
\end{lemma}

\begin{proof}
The measure \( \mDer\) can be written as \( \mu _{\mathcal L_\Der} ^l \otimes \mu ^t _{\mathcal L_\Der} \) where \( \mu _{\mathcal L_\Der}^l=\{\mu_L\} _{L\text{ leaf of }\mathcal L_\Der}\) is the family of leafwise Lebesgue measures in the time parametrization of the almost-periodic flow, and where \( \mu^t _{\mathcal L_{\Der}}\) is a transverse invariant measure on \( \mathcal L_{\Der}\). Since \(\Gamma\) preserves each trajectory of the almost-periodic flow, it preserves \( \mu^t _{\mathcal L_{\Der}}\). In particular, for \(\mDer\)-a.e. \(z\in \Der\), denoting by \( L\) the leaf containing \(z\)
\[ \frac{\gamma^{-1} _* \mDer}{\mDer} (z) =  \frac{\gamma^{-1}_* \mL}{\mL} (z) = D_{\mathcal L_\Der} (\gamma, z) .\] 
\end{proof}

\subsubsection{Stationary measures}

\begin{lemma}\label{l: stationarity of T-invariant measure}
Any probability measure \(\mDer\) on the almost-periodic space \(\Der\) which is \(T_{\Der}\) invariant is \(\mGamma\)-stationary, namely it satisfies the equation \(\mGamma * \mDer = \mDer\). Moreover, if the measure \(\mGamma\) is totally invariant, namely \(\gamma_* \mDer=\mDer\) for every \(\gamma \in \Gamma\), then \(\Gamma\) has a non trivial first Betti number.  
\end{lemma}

\begin{proof}
The stationarity of \(\mDer\) is an immediate consequence of Lemma \ref{l: radon-nikodym} and the third property of Theorem \ref{t: almost-periodic space}. Now, assume that \( \mDer \) is globally \(\Gamma\)-invariant and write as in Lemma \ref{l: radon-nikodym} \(\mDer= \mu _{\mathcal L_\Der} ^l \otimes \mu ^t _{\mathcal L_\Der} \). Since \(\mu ^t_{\mathcal L_\Der}\) is \(\Gamma\)-invariant, this proves that \( \mu _L \) is \(\Gamma\)-invariant on \(m_{\mathcal L_\Der}\)-a.e. leaf \(L\). On those leaves the action  of \(\Gamma\) acts then by translations, and this gives the desired non trivial morphim to \(\R\). \end{proof}


\subsection{The suspension space} 

The \textit{suspension space} $X$ is the quotient of \( G\times \Der \) by the diagonal action of \( \Gamma\) defined by \( \gamma (g, z) = (g\gamma^{-1}, \gamma(z) ) \). It is equiped with the natural quotient topology. It is a locally compact space (compact if \(\Gamma\) is uniform) having the following properties: 

\vspace{0.4cm} 

There is a \(G\)-action by homeomorphisms on the suspension space $X$, defined by \( h (g,z)\ \text{mod}\ \Gamma = (hg, z)\ \text{mod} \ \Gamma\). 

\vspace{0.2cm} 

The suspension space $X$ is a  \(G\)-equivariant fibration \( \pi :X \rightarrow G/\Gamma\) with fibers homeomorphic to the almost-periodic space \(\Der\) (the fibration is defined by \( \pi ((g,z)\ \text{mod}\ \Gamma) = g\ \text{mod}\ \Gamma\)).

\vspace{0.2cm} 

The fibers of \( \pi \) are naturally identified with the almost-periodic space, up to the action of \(\Gamma\).

\subsubsection{Laminated structures}\label{sss: laminated structures}

The suspension space carries an oriented lamination \(\mathcal L_X\) by unidimensional Lipschitz manifolds, which is invariant by the \(G\)-action, tangent to the fibers of \( \pi \), and whose restriction to these latter is the lamination \(\mathcal L_{\Der}\) (which is well-defined up to the action of \(\Gamma\)). Formally, \(\mathcal L_X\) is constructed in the following way. Its lift to the universal covering \( \widetilde{X} = G\times \Der\) is the lamination \(\widetilde{\mathcal L_X}\) whose charts are defined by \( \Phi_i ^{\mathcal L_X}: U_i^{\mathcal L_X} = G \times U_i \rightarrow I\times T_i ^{\mathcal L_X}=   I \times (G \times T_i) \) with \( \Phi_i (g,z) = (x_i(z), (g, \tau_i(z)))\), where \( \varphi_i : U_i \rightarrow I \times T_i\) are the charts of \( \mathcal L_{\Der}\). The change of coordinates are of the form postulated in \ref{sss: lamination Lipschitz}. Since the lamination by Lipschitz manifolds \(\widetilde{\mathcal L_X}\) is invariant by the action of \(\Gamma\) on \(G\times \Der\), it induces on the quotient \(X= G\times \Der /\Gamma\) the desired lamination by Lipschitz manifolds \(\mathcal L_X\).  

\vspace{0.2cm} 

The suspension space carries also a lamination \(G\mathcal L_X\) whose leaves are the sets \(GL\) where \(L\) is a leaf of \(\mathcal L_X\). Formally, it can be defined using the same system of charts, but with the transverse space \(T_i^{G \mathcal L_X} \) being the space \(T_i\), and the plaques the Lipschitz manifold \(I\times G\). More precisely, we will define \( \Phi_i ^{G\mathcal L_X} : U_i^{G\mathcal L_X} = G \times U_i \rightarrow B\times  T_i ^{\mathcal L_X}=   ( I \times G) \times T_i) \) with \( \Phi_i (g,z) = ((x_i(z), g) , \tau_i(z)))\), where \( \varphi_i : U_i \rightarrow I \times T_i\) are the charts of \( \mathcal L_{\Der}\). For the same reason as above this system of charts provide the desired structure of lamination by Lipschitz manifolds \(G\mathcal L_X\).

\vspace{0.2cm}

Given a any probability measure \(\mDer\) on \(\Der\) invariant under the almost-periodic flow is associated a transverse invariant measure \(\mu _{\mathcal L _{\Der} }\) on \( \mathcal L_{\Der}\), see Lemma \ref{l: invariant and transverse invariant}. This transverse invariant measure induces a transverse invariant measure \( \mu _{G\mathcal L_X}\) on \(G \mathcal L_X\)  because the holonomy pseudo-groups of \( \mathcal L_{\Der}\) and \(G \mathcal L_X\) coincide. Multiplying the transverse invariant measure \( \mu_{G\mathcal L_ X}\) by the Haar measure on the \(G\)-orbits gives birth to a transverse invariant measure \( \mu _{\mathcal L_X}\) on \(\mathcal L_X\), which is invariant by the \(G\)-action. (Note that this construction produces all the \(G\)-invariant transverse invariant measures on \(\mathcal L_X\), but there are plenty of other non \(G\)-invariant ones.)

\subsubsection{The Lyapunov cocycle on the suspension space}\label{sss: Lyapunov cocycle on suspension} 



Applying 3. of Theorem \ref{t: Martin boundary of discretization} to the family of functions \( \{ D_{\mathcal L_\Der} (\cdot, z) \} _{z\in \mathcal C_\Der}  \), we get a function \( \widetilde{D_{\mathcal L _{\Der}} } : G\times \mathcal C_\Der \rightarrow (0, +\infty) \) such that: 

\begin{enumerate}

\item its restriction to \( \Gamma \times \mathcal C_\Der \) equals \(D_{\mathcal L_\Der}\),

\item it is left \(\mGBM\)-harmonic in the \(G\)-coordinate,

\item it satisfies the cocycle relation 
\begin{equation} \label{eq: first cocycle relation} \widetilde{D_{\mathcal L_\Der}} (g \gamma, z) = \widetilde{D_{\mathcal L_\Der}} ( g, \gamma(z))  D_{\mathcal L_\Der} (\gamma, z) \text{ for every }\gamma \in \Gamma, \ g\in G\text{ and } z\in \mathcal C_\Der.\end{equation}

\item it is measurable.

\end{enumerate}

Let \(\mathcal C_X\subset X= \Gamma \backslash (G\times \Der)\) be the quotient of the \(\Gamma\)-invariant set \( G\times \mathcal C_\Der\subset G\times \Der\). The set \(\mathcal C_X\) is \(G\)-invariant, and moreover, for every leaf \( Y\) of the lamination \(G\mathcal L_X\), \(\mathcal C_X\) contains a.e. \(G\)-orbits of \(Y\) with respect to the Lebesgue measure class on \(Y\). 

\begin{definition}[The multiplicative cocycle] 
The multiplicative cocycle \( D_{\mathcal L_X}: G\times \mathcal C_X\rightarrow (0,+\infty)\) is defined by the formula 
\begin{equation}\label{eq: Lyapunov cocycle} D_{\mathcal L_X} ( g , x) := \frac{\widetilde{D_{\mathcal L_\Der}} (gh, z)}{\widetilde{D_{\mathcal L_\Der}} (h,z) } \text{ for every } g\in G \text{ and every } x=(h, z)  \text{ mod } \Gamma \in \mathcal C_X. \end{equation}
\end{definition}

The fact that the right-hand side does not depend on the point in the \(\Gamma\)-orbit of \( (h,z) \) is a consequence of property 3. above. It clearly satisfies the multiplicative relation 
\begin{equation} \label{eq: cocycle relation} D_{\mathcal L_X} (g' g, x) =D_{\mathcal L_X} (g', gx) \cdot D_{\mathcal L_X} (g, x) ,\end{equation}
and moreover, it is a measurable function which is left \(\mGBM\)-harmonic in the \(G\)-coordinate. 

\begin{definition}[The additive cocycle] \label{d: Lyapunov cocycle}
The Lyapunov (additive) cocycle is the function \( c : G\times \mathcal C_X\rightarrow \R\) defined by \( c:= \log D_{\mathcal L_X}\). 
\end{definition}

\subsubsection{Volume form and leafwise distances/measures on \(\mathcal L_X\)}\label{sss: forms measures distances}

Denote by \( \omega_{\mathcal L_\Der} \) the time form on \( \mathcal L_\Der\) associated to the flow \( T_\Der\): \(\omega\) is equal to \(ds\) in the coordinates given by Lemma \ref{l: local product structure}. Letting \( p : G\times \Der\rightarrow \Der\) be the projection, consider the form \( \widetilde{\omega _{\mathcal L_X}}\) on \( \widetilde{\mathcal L_X}\) defined by the formula 
\begin{equation}\label{eq: volume form on LX} \widetilde{\omega_{\mathcal L_X}}  =  \widetilde{D_{\mathcal L_\Der} } \cdot  p ^* \omega_{\Der} . \end{equation}
The form \( \widetilde{\omega_{\mathcal L_X}}\) can be thought as the family \( \{ \omega _{\mathcal L_\Der} ^g\}_{g\in G}\) of volume forms on \( \mathcal L_\Der\) defined by \( \omega_{\mathcal L_\Der} ^g = \widetilde{D_{\mathcal L_\Der}} (g, \cdot) \omega_{\mathcal L_\Der}\). We have a  corresponding family of distances \(\{ d^g =\{d^g_L\}_{L\text{ leaf of }\mathcal L_\Der} \} _{g\in G}\) on the leaves of \( \mathcal L_\Der\), defined by 
\begin{equation} \label{eq: family of distances on Der} d^g _L (x, y) =\left\lvert \int _x^y \omega _{\mathcal L_\Der} ^g \right \rvert\end{equation}
for two points \(x,y\) that belong to a same leaf \(L\) of \( \mathcal L_\Der\). Notice that the distance \( d^g\) depends harmonically on \( g\in G\), and therefore that \( d^{kg} = d^g\) for  every \(g\in G\) and every \(k\in K\).

We have for every \(\gamma \in \Gamma\),
\[ \gamma^* \widetilde{\omega_{\mathcal L_X}} = \left( \widetilde{D_{\mathcal L_\Der} }\circ \gamma \right)  \cdot  p ^* (\gamma ^* \omega_{\Der}) = \left( \widetilde{D_{\mathcal L_\Der} }\circ \gamma \right) \cdot D_{\mathcal L_Z} (\gamma, z) p^* \omega_{\Der} . \] 
Because of the cocycle relation \eqref{eq: cocycle Der} satisfied by the function \(\widetilde{D_{\mathcal L_\Der} }\), 
\[\widetilde{D_{\mathcal L_\Der} }\circ \gamma (g, z) = \widetilde{D_{\mathcal L_\Der}} (g\gamma^{-1}, \gamma (z) ) = D_{\mathcal L_\Der} (\gamma^{-1},\gamma (z) )  .\]
 In particular, since \( D_{\mathcal L_\Der} (\gamma^{-1},\gamma (z) ) D_{\mathcal L_\Der} (\gamma, z) = 1\) by the cocycle relation \eqref{eq: cocycle Der}, we deduce that the form \( \widetilde{\omega_{\mathcal L_X}}\) is \(\Gamma\)-invariant, hence defines a measurable volume form \(\omega_{\mathcal L_X}\) on \(\mathcal L_X\). Notice that in terms of the distances \( \{ d^g\} _{g\in G}\), this implies the relation 
 \begin{equation} \label{eq: equivariance distance} d^{g\gamma^{-1}} _{\gamma L} (\gamma x,\gamma y)  =  d^g _L (x,y) \end{equation}
for every \(\gamma \in \Gamma\), every leaf \(L\) of \(\mathcal L_\Der\) and every couple of points \((x,y)\in L^2\).

The measurable volume form \(\omega_{\mathcal L_X}\) gives rise to a family of Radon measures \(\mu_{\mathcal L_X} ^l = \{ \mL \} _{L \text{ leaf of } \mathcal L_X}\) on the leaves of \(\mathcal L_X\), \(\mu_L\) being the absolutely continuous measure associated to the restriction of \(\omega_{\mathcal L_X}\) to \(L\) as explained in \ref{sss: measurable volume forms}. 

We will use in the sequel the following notation (see section \ref{sss: measurable volume forms}): given a leaf \( L\) of \( \mathcal L_X\), and two points \( x, y \in L\), we denote
 \[ y- x = \int _x^y \omega_{\mathcal L_X} ,\]
 and define the distances  
 \[ d_L (x,y) = | y-x | = \mu_L ([x,y]) \text{ for } x,y\in L.\]
Given an interval \(I\subset L\), we will sometimes denote \( |I|= \mu_L (I) \). 

\begin{proposition} \label{p: properties of suspension} We have:

\vspace{0.2cm}

0. The family of leafwise measures \( \mu_{\mathcal L_X} ^l =\{\mu_L\} _{L\text{ leaf of } \mathcal L_X}\) is continuous.

\vspace{0.2cm}

1. The family of distances \( \{ d_L \} _{L\text{ leaf of }\mathcal L_X} \) is continuous, in the sense that in a coordinate chart \( \varphi_i : U_i \rightarrow B\times T_i \) of \(\mathcal L_X\), the function \( (x_i , y_i , \tau_i )\in B^2 \times T_i  \mapsto d_{L_{\tau_i}}(\varphi_i ^{-1} (y_i, \tau_i) , \varphi_i^{-1} (x_i, \tau_i)) \in [0,+\infty) \) is continuous.

\vspace{0.2cm}

2. For every leaf \(L\) of \(\mathcal L_X\) and every pair of points \( x< y \), the function \( g\in G \mapsto g(y) - g(x) \in (0, +\infty) \) is a positive \(\mGBM\)-harmonic function on \(G\). 

\vspace{0.2cm} 

3. For every \( x\in \mathcal C_X \), the function \( D_{\mathcal L_X} (\cdot, x) \) is left \( \mGBM\)-harmonic. Moreover, 
\[ D_{\mathcal L_X} (g, x) = \lim _{\varepsilon \rightarrow 0} \frac{g(x+\varepsilon) - g(x) }{\varepsilon}, \]
and the limit is uniform on compact subsets of \(G\).

\vspace{0.2cm}

4. There exists a constant \( \xi= \xi (G) \) such that for every \( x\in \mathcal C_X\), the  function \( c(\cdot, x)\) is \(\xi \)-Lipschitz.  
 \end{proposition}

\begin{proof}
The two first statements are consequences of the continuity property of the extension in the Ledrappier-Ballmann theorem, see item 3. of Theorem \ref{t: Martin boundary of discretization}.

To prove statement 2., notice that for each pair of points \(x= (h,x') , y = (h,y') \in G\times \Der \) in a same \(\widetilde{\mathcal L_X}\)-leaf (namely \(x',y'\) belong to the same \(\mathcal L_\Der\)-leaf \(L\)) 
\[y - x= \int _{x} ^{y} \omega_{\mathcal L_X} = \int _{x'}^{y'} \widetilde{D_{\mathcal L_\Der}}(h, z') \ \mL (dz')\] 
so using the cocycle relation \eqref{eq: Lyapunov cocycle}, we get 
\begin{equation}\label{eq: G action} gy - gx= \int_{x'} ^{y'} \widetilde{D_{\mathcal L_\Der}}(gh, z') \ \mL (dz') = \int _{x}^{y} D_{\mathcal L_X} ( g , \cdot ) \omega_{\mathcal L_X}   \end{equation}
In particular, since the function \( D_{\mathcal L_X}\) is left \(\mGBM\)-harmonic in the \(G\)-variable, the statement follows.

Let us first prove 3. for a point \(x=( e, x') \in G\times  \mathcal C_\Der\).  We then have for each \(\gamma\in \Gamma\)
\[ D_{\mathcal C_Z} (\gamma, x') = \lim _{\varepsilon\rightarrow 0} \frac{\gamma (x'+\varepsilon)- \gamma (x') }{\varepsilon}= \frac{1}{\varepsilon} \int _{x'}^{x'+\varepsilon} D_{\mathcal L_{\Der}} (\gamma, y') \mL (dy) , \]
where in the last formula \(L\) is the \(\mathcal L_\Der\)-leaf containing \(x'\). Again, the continuity property in Ledrappier-Ballmann theorem, see item 3. of Theorem \ref{t: Martin boundary of discretization}, shows that we get for every \(g\in G\)
\[ D_{\mathcal L_X} (g, x) = \widetilde{D_{\mathcal L_\Der}}(g, x') = \frac{1}{\varepsilon} \int _{x'}^{x'+\varepsilon} \widetilde{D_{\mathcal L_{\Der}}} (g, y') \mL (dy)=  \lim _{\varepsilon\rightarrow 0} \frac{g (x+\varepsilon)- g (x) }{\varepsilon} .\]
The claim follows for such an \(x\). Now, let \(y\in G\times \mathcal C_\Der\) and write \(y=hx \) with \(x= (e, x')\) for some \(x'\in \mathcal C_{\Der}\). For every \(\varepsilon \), let \( \eta = \eta (\varepsilon)  \) such that \( h(x+\eta) = y+\varepsilon\). We have \(\varepsilon /\eta = h(x+\eta) - h(x) /\eta \rightarrow_{\varepsilon \rightarrow 0} D_{\mathcal L_X} (h, x)\), and   
\[ \lim_{\varepsilon \rightarrow 0} \frac{g (y+\varepsilon) - g(y) }{\varepsilon}=\lim _{\varepsilon  \rightarrow 0} \frac{gh (x+\eta) - gh (x) }{\eta}\cdot  \frac{\eta} {\varepsilon}  = \frac{D_{\mathcal L_X} (gh, x) }{D_{\mathcal L_X} (h,x) }= D_{\mathcal L_X} (g, y)  \]
by the cocycle relation \eqref{eq: cocycle relation}.

The last statement comes from Harnack inequality \ref{t: Harnack inequality} and the fact that \(D_{ \mathcal L_X} \) is a positive left \(\mGBM\)-harmonic function in the \(G\)-coordinate. \end{proof}

\subsubsection{Construction of the measure \(\mX\)} \label{sss: mX}

For every \(g\in G\), we denote by \(\mDer  ^g\) the finite measure on \(\Der\) defined by 
\begin{equation}\label{eq: mDerg} \mDer ^g := \widetilde{D_{\mathcal L_\Der}} (g, \cdot ) \mDer .\end{equation}
The measure \(\mDer ^g\) is well-defined and finite since \(\widetilde{D_{\mathcal L_\Der}}\) is defined on \( G\times \mathcal C_\Der \),  \(\mathcal C_{\Der}\) has full \(m_\Der\)-measure, and \(\widetilde{D_{\mathcal L_\Der}}(g, \cdot) \) is bounded by Harnack inequality.

\begin{lemma}\label{l: mX}
The family \(\{\mDer^g\}_{g\in G}\) of measures on \( \Der\) satisfies the following properties 

\begin{enumerate}
\item for each \(g\in G\) and each \(\gamma \in \Gamma\), \( \mDer ^{g\gamma^{-1} } = \gamma_* \mDer\), so \(\{\mDer^g\}_{g\in G}\) induces a family on measures \( \{\mu_{\pi^{-1}(g\Gamma)}\}_{g\Gamma\in G/\Gamma}\) on \(X\) supported on the fibers of \( X\rightarrow G/\Gamma\),
\item for each \(g\in G\), \(\mDer^g\) and \( \mu_{\pi^{-1} (g\Gamma)}\) are probability measures, 
\item The probability  measure on \(X\) defined by
\begin{equation}\label{d: mX} \mX := \text{cst} \int _{G/\Gamma} \mDer^{g\Gamma} \HaarG (dg\Gamma) \end{equation}
(the constant is a normalizing constant) gives full Lebesgue measure to \( \mathcal C_X\),
\item for every \(g\in G\), and \(\mX\)-a.e. \(x\in X\) 
\[ \frac{g^{-1} \mX}{\mX } (x) = D_{\mathcal L_X} (g, x) .\]
\item \(\mX\) is \(\mGBM\)-stationary.
\end{enumerate}

\end{lemma} 

\begin{proof}
For the first item, use the cocycle relation \eqref{eq: first cocycle relation} to get 
\[ \mDer ^{g\gamma^{-1}} (dz) =  \widetilde{D_{\mathcal L_\Der}} (g\gamma^{-1}, z ) \mDer (dz) = \]
\[=\widetilde{D_{\mathcal L_\Der}} (g, \gamma^{-1} z ) \mathcal D_{\Der} (\gamma^{-1} , z) \mDer(dz)  = \gamma_* \mDer ^g (dz). \]
We define \(\mu_{\pi ^{-1} (g\Gamma)} \) as the measure on \(X\) which is the image of the measure \( \delta _g \otimes \mDer^g\) on \(G\times \Der\) (here \(\delta_g\) is the Delta measure supported on \(\{g\}\)) by the projection \( G\times \Der \rightarrow X\).

For the second, notice that \( g\in G \mapsto \mDer^g (\Der) \in [0, \infty) \) is \(\Gamma\)-invariant, by the first item. Formula \eqref{eq: mDerg} together with the \(G\)-harmonicity of \(\widetilde{D_{\mathcal L_\Der}} \) shows that this function is \(\mGBM\)-harmonic, hence it needs to be constant as a positive harmonic function on \( G/\Gamma\). This constant is one since \( \mDer^e= \mDer\) is a probability. Hence the second item follows.

The third item follows from the fact that \( \mathcal C_X = G\times \mathcal C_\Der\) and that \( \mathcal C_\Der\) has full \(\mDer\)-measure.

To prove the fourth item, let us observe that the measure \(\mu_X\) lifts to \( G\times \Der\) to the \(\Gamma\)-invariant measure \( \widetilde{\mX} = \int \mDer ^h \ \HaarG (dh) \), so that given a point \( \tilde{x} \) on the \( \pi \)-fiber of the point \(x\) we have \(\frac{g^{-1} _*\mX }{\mX} = \frac{g^{-1} _* \widetilde{\mX}}{\widetilde{\mX}} (\tilde{x}) \). The formula is then an immediate consequence of the definition \eqref{eq: mDerg} of \(\mDer ^g\), and of the cocycle relation \eqref{eq: Lyapunov cocycle}. 

The last item is a consequence of the fourth one, together with the symmetry of \(\mGBM\) and the fact that \( D_{\mathcal L_X}\) is left \(\mGBM\)-harmonic in the \(G\)-coordinates. \end{proof}

The next result is not truly necessary for our whole argument, but it will be useful in order to short cut moment arguments related to the non compacity of the support of \( \mGBM\).  

\begin{proposition}\label{p: harmonicity compact support}
Assume \(T_{\Der}\) is minimal. Let \(\mDer\) be a probability measure on \( \Der\) invariant under \( T_{\Der}\), and \(\mX\) the associated measure on \(X\). For every probability measure \(\mG\) on \(G\) of type \(\type\), the measure \( \mX\) is \(\mG\)-stationary. Moreover, the following properties are also satisfied: 

\vspace{0.2cm}

1. for every leaf \( L\) of \(\mathcal L_X\), and any pair of points \(x,y\in L\), the function \( g\in G \mapsto g(y) - g(x) \in \R\) is left \(\mG\)-harmonic.

\vspace{0.2cm}

2. for every \(x\in \mathcal C_X\), the function \( g\in G \mapsto D_{\mathcal L_X} (g, x) \in (0, \infty) \) is left \(\mG\)-harmonic.
\end{proposition}

\begin{proof}
By the property 3. of Proposition \ref{p: properties of suspension}, the first point of the Proposition implies the second.

So let us prove the first point. The function \( g\in G \mapsto g_* \mX\in \text{Prob} (X) \) is left \(\mGBM\)-harmonic and bounded. Furstenberg's theorem shows that it is also left \(\mG\)-harmonic. In particular, for \( \mX\)-a.e. \(x\in X\), the function \( g\in G\mapsto \frac{g^{-1} _* \mX}{\mX}(x) \in (0, \infty) \) is left \(\mG\)-harmonic. Since we have \( \frac{g^{-1}_* \mX}{\mX}(x) = D_{\mathcal L_X} (g, x)\), for \(\mX\)-a.e. \(x\), the function \( D_{\mathcal L_X} (\cdot , x)\) is left \(\mG\)-harmonic.  In particular, for a \(\mX\)-generic leaf \(L\) of \(\mathcal L_X\), the function \( D_{\mathcal L_X} (\cdot, x) \) is left \(\mG\)-harmonic for \(\mu_L\)-a.e. \(x\in L\).  For every \(x,y\in L\), we have 
\[ g(y) - g(x) = \int _x^y D_{\mathcal L_X} (g, z) \mu_L (dz) ,\]
hence, the function \( g\in G \mapsto g(y) - g(x) \in (0, \infty) \) is left \(\mG\)-harmonic as well. We conclude that for a \(G\)-invariant set of \(\mathcal L_X\)-leaves the property 1. is satisfied. This set being dense in \(X\) by minimality of \(T_\Der\) and by consequence of the lamination \(G\mathcal L_X\), the continuity property 1. of Proposition \ref{p: properties of suspension} shows that the property 1. of our Proposition \ref{p: harmonicity compact support} is satisfied.\end{proof}

We end this section by the following important

\begin{proposition}\label{p: invariance mX}
The measure \(\mX\) is \(G\)-invariant if and only if the measure \(\mDer\) is \(\Gamma\)-invariant. In particular, if the measure \(\mX\) is \(G\)-invariant, the first Betti number of \(\Gamma\) is non zero. 
\end{proposition}

\begin{proof}
The \(G\)-invariance of \(\mX\) together with the fourth item of Lemma \ref{l: mX} imply that for \(\mX\)-a.e. \(x\) the function \( D_{\mathcal L_X} (\cdot, x)\) is identically \(1\). In particular, the family of distances \( \{d_L\}_{L\text{ leaf of } \mathcal L_X}\) is also \(G\)-invariant, and we conclude that \( D_{\mathcal L_X} (\cdot, x)\) is constant equal to \(1\) for every \(x\in \mathcal C_X\). By restriction to \(\Gamma\), this proves that \( D_{\mathcal L_{\Der}} (\cdot , z) \) is identically equal to \(1\) for every \(z\in \mathcal C_{\Der}\), and thus \(\mDer\) is \(\Gamma\)-invariant. The proof of the reciproque is analogous. 

The conclusion of the Proposition follows from Lemma \ref{l: stationarity of T-invariant measure}.
\end{proof}

\subsubsection{\(P_{\mathcal W}\)-invariant measures}

\begin{lemma} \label{l: PW invariant measure}
Let \(\mDer\) be a probability measure on \( \Der\) invariant under \( T_{\Der}\), and \(\mX\) the associated measure on \(X\). For any Weyl chamber \(\mathcal W\), there exists a unique probability  measure \( \mX ^{\mathcal W}\) on \(X\) which satisfies:

\vspace{0.2cm}

1. \(\mX^{\mathcal W}\) is \(P_{\mathcal W}\)-invariant and

\vspace{0.2cm}

2. \( \int _K k_* \mX^{\mathcal W} \ \HaarK(dk) = \mX\).

\vspace{0.2cm}

\noindent Moreover, the set \(\mathcal C_X\) has full \( \mX^{\mathcal W}\)-measure.
\end{lemma}

\begin{proof}
The function \( g\in G\mapsto F(g) = g^{-1} _* \mX\in \text{Prob} (X)\) is left \(\mG\)-harmonic because \(\mG\) is symmetric. It also satisfies the following equivariance with respect to the \(G\)-actions on itself on the right and on \(\text{Prob} (X)\):
\[ F (gh^{-1}) = h_* F(g) \text{ for every } g,h\in G .\]

The Poisson formula, see Theorem \ref{t: Poisson formula}, shows that there exists a unique measurable function \( f :  G\rightarrow \text{Prob} (X) \) which is left \(P_{\mathcal W}\)-invariant and which is such that \( \int_K f( kg) \ \HaarK (dk) = F(g) \) for every \(g\in G\). By unicity of the function \(f\), this latter satisfies the same equivariance property \( f(gh^{-1} ) = h_* f(g)\) for every \(g,h\in G\).

The measure \(\mX ^{\mathcal W}= f(e) \) is then \(P_{\mathcal W}\)-invariant, and the second condition of the Lemma is a consequence of \(\int _K f(k) \ \HaarK (dk) = F(e)\).

The unicity of \( \mX^{\mathcal W}\) is a consequence of the unicity of the solution in the Poisson formula as well. 

For the last assertion of the Lemma, observe that, the set \( \mathcal C_X\) being \(K\)-invariant, we have \(\mX^{\mathcal W} (\mathcal C_X)= k_* \mX^{\mathcal W} (\mathcal C_X)\), hence
\[\mX^{\mathcal W} (\mathcal C_X) =  \int _K k_* \mX^{\mathcal W} \HaarK(dk) = \mX(\mathcal C_X) = 1.\]
\end{proof}


\section{Global contraction: random estimates}\label{globalcontraction1}

This is the first section dedicated in establishing contraction properties of the action of \(G\) along the lamination \(\mathcal L_X\). The following concepts will be central:

\begin{definition}
A sequence \( (g_n) _{n\in \N}  \) of elements of \(G\) has the local contraction property with respect to a measure \( \mu\) on \(X\) if for \(\mu\)-a.e. \(x\), there exists \(\varepsilon _x >0 \) such that for any \(y\in L_x\) satisfying \( d_{L_x} (x,y) \leq \varepsilon_x\), the distance \( d_{L_x} ( l_n (y) , l_n (x) )  \) tends to zero when \(n\) goes to infinity, where we recall the notation \(l_n = g_n \ldots g_1\).

A sequence \((g_n)_{n\in \N}\) has the global contraction property with respect to \(\mu\) if it has the local contraction property and if we can choose \(\varepsilon_x= \infty\) for \(\mu\)-a.e. \(x\in X\). 

We will also use the following terminology: we will say that an element \( g\in G\) (resp. \(a\in \lieg\)) has the local/global contraction property with respect to \(\mu\) if the sequence \((g^n)_{n\in \N}\) of iterates of \(g\) (resp. if the sequence \( ( \exp (na) )_{n\in \N} \)) has the given property.
\end{definition}

Here we establish a global contraction property of the random walk on \(G\) induced by a probability measure \(\mG\) of type \(\type\) along the \(\mathcal L_X\)-lamination, namely we prove that \( \mG^{\N} \)-a.e. sequence \( (g_n)_{n\in \N}\) has the global contraction property with respect to \(\mX\).   

\subsection{Qualitative estimates}

\begin{definition} An action of a group \(\Gamma\) on \(\R\) has the global contraction property if there exists a compact set \(K\subset \R\) such that for every compact interval \(I\subset \R\), and every \(\varepsilon >0\), there exists an element \(\gamma\in \Gamma\) such that \(\gamma (I) \) is contained in \(K\) and its length is bounded by \(\varepsilon\). \end{definition}

\begin{proposition} 
Assume that \(\Gamma\) is an irreducible lattice in a connected semi-simple Lie group of rank \(\geq 2\) and finite center. Then, any action of \(\Gamma\) on the real line by orientation preserving homeomorphism and without a global fixed point has the global contraction property. 
\end{proposition}

\begin{proof}
An action of \(\Gamma\) on the real line by orientation preserving homeomorphisms and without a global fixed point is of one of the following types (see \cite[]{DNR} and the references therein).
\begin{enumerate}
\item It has a discrete orbit. 
\item It commutes with a fixed point free homeomorphism of \(\R\).
\item It has the global contraction property.
\end{enumerate}

The first case cannot occur since otherwise the group \(\Gamma\) would have a non trivial morphism to the integer, and this is impossible by Proposition \ref{p: vanishing first Betti number}. 

The second case cannot happen neither. Indeed, suppose by contradiction that an action of \(\Gamma\) on \(\R\) without global fixed point is of type 2. Notice that semi-conjugation does not alter type 2. or 3., so we can assume that our action is minimal. In such a situation Ghys \cite{Ghys} proved that there exists a surjective morphism \(q : G\rightarrow \text{PSL} (2,\R)\) so that the action of \(\Gamma\) on \(\R\simeq \widetilde{\R P^1}\) is a lift of the action \(q_{|\Gamma}\) of \(\Gamma\) on the projective line \(\R P^1\). The central extension  \( r: \widetilde{\text{PSL}} (2,\R) \rightarrow \text{PSL}(2,\R) \) (the universal covering of \(\text{PSL} (2,\R) \)) lifts to a central extension \( q : \widetilde{G} \rightarrow G\) (defined as \(\widetilde{G}:=\{ (g,x) \in G\times \text{PSL}(2,\R)\ |\ q(g) = r( x) \}\)). The group \(\widetilde{G} \) is connected since otherwise the map \( q\) would lift to a map from a finite covering of \(G\) to \(\widetilde{\text{PSL}} (2,\R)\), and this is impossible since a semi-simple Lie group with finite center do not have non trivial morphisms to \( \widetilde{\text{PSL}}(2,\R)\). The group \(\Gamma\) lifts to \(\widetilde{G}\), and so the lattice \(q^{-1} (\Gamma)\subset \widetilde{G}\) is isomorphic to the product \( \Gamma\times \Z\), in particular has a non vanishing first Betti number. Proposition \ref{p: vanishing first Betti number} then gives the desired contradiction.



Hence any action of \(\Gamma\) on the real line without global fixed point is of type 3., and the proof follows.\end{proof}

\subsection{Quantitative estimates}

\begin{proposition}\label{p: quantitative contraction}
Assume that \(T_\Der\) is minimal and that the restriction of the \(\Gamma\)-action to any \(T_\Der\)-trajectory has the global contraction property. Fix a probability measure \(\mG\) on \(G\) of type \(\type\). Then given a leaf \(L \) of \(\mathcal L_X\), and two points \(x,y\in L\), we have for \(\mG^{\N} \)-a.e. \((g_n)_n\)
\[ d_{l_n L} \left( l_n (x) , l_n (y)\right)  \rightarrow 0 ,\]
where as usual \(l_n= g_n\ldots g_1\).
\end{proposition}

\begin{proof} The argument is similar to \cite[Proof of Theorem 7.2]{DKNP}.

Let \( K\subset G /\Gamma\) be a compact set on which the left random walk induced by \( \mG\) is recurrent, namely for every \( p\in G/\Gamma\), and \(\mG^{\N}\)-a.e. sequence \( (g_n)_n \in G^{\N}\), the point \( g_n \ldots g_1 (p) \) belongs to \(K\) for an infinite number of times. We denote by \(\widetilde{K}\) the preimage of \( K\) in  \( X\).  

Given two numbers \(l , \varepsilon >0 \), there exists a probability \( p= p_{l,\varepsilon}>0\) and an integer \(N\) such that for every point \(x\in \widetilde{K} \),  denoting \(L_x\) the leaf of \(\mathcal L_X\) containing \(x\)
\[ \mathbb P \left( d_{ l_N L }(l_N (y) , l_N (x) ) \leq \varepsilon \right)\geq p \text{ for every } y\in L_x \text{ such that } d_{L_x} (x,y) \leq l.\]
This is due to the compacity of  \(\widetilde{K}\) and the strong contraction property on the support of \(m_\Der\).

For every \(x\in \widetilde{K}\), we denote by \(U_x\subset G\) the subset formed by elements \( h \in G\) such that the length of the interval \[ h([x-l, x+l])\subset L_{h(x)}\] is bounded by \(\varepsilon\). We have \( \mG^{*N} (U_x) \geq p\). We will prove that for every \(x\) in the support of \(\mX\) and \(\mG^{\N} \)-a.e. \((g_n)_n\in G^\N\), there exists an infinite number of \(n\)'s such that \( g_{n+N} \ldots g_{n+1}\in U_{g_n\ldots g_1(x)}\). 

Indeed, introduce the probability \(p(x)\) that this happens at least once, for \(x\) in the support of \(\mX\). We have, by the Markov property 
\[ p(x) = \mG^{*N} (U_x) + \int _{G\setminus U_x} p (g(x))\   \mG^{*N} (g) \] 
and so the infimum \( p_- =\inf p(x)\) for \(x\) in the support of \(\mX\) satisfies \(p_- \geq p + (1-p) p_-\). We infer \(p_-=1\), and so \(p(x) =1\) for every \(x\) in the support of \(\mX\). The Markov property shows that indeed, the same event happens an infinite number of times. 

Now, the Martingale theorem tells us that for every \(x\in X\), every \(y\in L_x\), and \(\mG^\N\)-a.e. \((g_n)_n\), the length of the interval \( [l_n(x), l_n (y) ]\) tends to a finite limit \(l(x,y, (g_n)_n) \) when \(n\) tends to infinity.  Denote by \(\mathcal E _l\subset G^{\N}\) the subset of sequences for which this limit is strictly bounded by \(l\). For \(n\) large enough, the length of the interval \( [l_n(x), l_n (y) ]\) is bounded by \(l\). Applying what precedes, for \(\mG^\N\) a.e. element of \(\mathcal E_l\), there exists an infinite number of values of \(n\) so that \( l_{n+N} \in U_{g_n\ldots g_1(x)}\). We infer that there exists an infinite number of \(n\) so that 
\[ | l_n ([x,y]) | \leq l \text{ and } | g_{n+N} \ldots g_{n+1} ([l_n(x) - l, l_n (x)+l ] ) | \leq \varepsilon.\]
For those \(n\)'s we have \( | l_{n+N} ([x,y] ) | \leq \varepsilon\), and consequently for \(\mG^\N\)-a.e. \((g_n)_n\in \mathcal E_l\) 
\[ l(x,y, (g_n)_n) \leq \varepsilon.\]
This being valid for every \(l>0\) and every \(\varepsilon\), the conclusion follows. \end{proof}

\subsection{Ergodicity of the measures $\mX$ and $\mX ^{\mathcal W}$}

\begin{lemma}\label{ergodicitymx}
Let \(m_\Der\) be a probability measure on the almost-periodic space $\Der$ which is invariant by the flow \(T_\Der\) and ergodic. Assume that \(T_\Der\) is minimal, and that the restriction of the action of \(\Gamma\) to any \(T_\Der\)-trajectory has the strong contraction property.  Let $\mG$ be a probability measure on $G$ of type \(\type\) and $\mathcal W$ a Weyl chamber. Let \(\mX\) be the associated \(\mG\)-harmonic measure on \(X\), and \(\mX^{\mathcal W}\) the unique \(P_{\mathcal W}\)-invariant measure on \(X\) whose \(K\)-average is \(\mX\). Then $\mX$ is ergodic as a \(\mG\)-stationary probability measure on $X$ and  \(\mX^{\mathcal W}\) is $P_{\mathcal W}$-ergodic.
\end{lemma}

\begin{proof}
Let \(f \in L^ 1 (\mX) \). Birkhoff's ergodic theorem for stationary measures shows that there exists a function \( \overline{f} \in L^1( X )\) which is such that for \(\mX\)-a.e. \(x\in X\) and \(\mG^\N\)-a.e. \({\bf g}= (g_n)_n\in G^\N\), if we let $l_n := g_n l_{n-1}$, then $$ \lim_{N \to \infty} \frac{1}{N} \sum_{n = 1}^N f(l_n(x)) =  \overline{f} (x). $$
Moreover, a consequence of the ergodic decomposition for stationary measure is that the restriction of \(\overline{f}\) to \(\mX\)-a.e. \(G\)-orbit is constant \(\mG\)-a.e. 

The key observation is that, by Proposition \ref{p: quantitative contraction}, the function \(\overline{f}\) is constant on \(\mX\)-a.e. leaf of \(\mathcal L_X\). In particular, its restriction to \(\mX\)-a.e. leaf of the lamination \(G\mathcal L_X\) is constant a.e. wrt the product of the Haar measure on \(G\) by the measures \(m_L\) on the leaves \(L\) of \(\mathcal L_X\), namely wrt the desintegrations of the measure \(\mX\) along the leaves of \(G\mathcal L_X\). In particular, up to changing the values of \(\overline{f}\) on a set of \(\mX\)-measure zero, we can assume that \( \overline{f}\) is constant on \(\mX\)-a.e. leaf of \( G\mathcal L_X\). 

Denoting \(p: G\times \Der\rightarrow X\) the universal covering of \(X\), this means that there exists a function \(h\in L^1 (m_\Der)\), such that \( \overline{f} \circ p ( g, z) = h(z)\) for \(\HaarG\)-a.e. \(g\in G\) and \(m_\Der\)-a.e. \(z\in \Der\). Furthermore \(h\) is constant along \(m_\Der\)-a.e. \(T_\Der\). By ergodicity of \( m_\Der\) with respect to the flow \( T_\Der\), the function \(h\) is constant \(m_\Der\)-a.e., and consequently \( \overline{f}\) is constant \(\mX\)-a.e.

\end{proof}

\subsection{Exponential estimates}\label{expocontrawalk}


In this section we establish global exponential contraction along the lamination \(\mathcal L_X\) by the random walk induced by a probability measure \(\mG\) on \(G\) of type \(\type\).

\begin{lemma}\label{globallyapunov}
The Lyapunov exponent with respect to $\mG$ 
\[ \chi_{\mG}  := \int c(g,x) \ \mG (dg) \ \mX (dx) \]
is negative, unless the measure \(\mX\)  is globally \(G\)-invariant. In particular, if \(\Gamma\) has zero first Betti number, \(\chi_{\mG}\) is negative.
\end{lemma}

\begin{proof}
For \(\mX\)-a.e. \(x\in X\), the function \( c(\cdot, x)\) is the logarithm of a \(\mu_G\)-harmonic function, hence  the concavity of logarithm gives 
\[ \int _G c(g,x) \ \mG(dg) \leq 0\]
with equality iff \(c(\cdot, x) \) is constant a.e. In particular, we have \(\chi _{\mG}\geq 0\) with equality iff for \(\mX\)-a.e. \(x\in X\) the function \(D_{\mathcal L_X}(\cdot, x)\) is constant equal to \(1\).  This implies that \(\mX\) is \(G\)-invariant because for \(\mX\)-a.e. \(x\), we have 
\[ D_{\mathcal L_X} (g, x) = \frac{g^{-1}_* \mX }{\mX} (x) \] 
by Lemma \label{l: mX}. The conclusion follows from Proposition \ref{p: invariance mX}.
\end{proof}

\begin{lemma}\label{l: local contraction for mG}
Assume that \(G\) does not preserve \(\mX\). Choose \( \chi_+\) such that \( \chi_{\mG} <\chi_+<0\). Then, for \(\mX\)-a.e. \(x\in X\) and \( \mG^{\N^*}\)-a.e. \((g_n)_n\), there exists \(\varepsilon=\varepsilon (x, (g_n)_n)  >0\) such that 
\[ \limsup _{n\rightarrow \infty} \log |l_n ([x-\varepsilon, x+\varepsilon ]) | \leq \chi _+ ,\]
where \(l_n= g_n\ldots g_1\)
\end{lemma}

\begin{proof}
We introduce the family of functions, for \(\eta >0\),
\[ c_\eta (g, x) = \sup _{0< |\epsilon| < \eta } \log \left( \frac{g(x+\varepsilon)-g(x)}{\varepsilon}\right) .\]
 Those functions belong to \(L^\infty (\mG \otimes \mX) \) by the Harnack inequality. The third property of Proposition \ref{p: properties of suspension} shows that \( c_\eta\) converges simply to \( c\) when \(\eta\) tends to \(0\). Hence the Lebesgue dominated convergence theorem implies 
 \(\int_{G\times X} c_\eta (g,x) \ \mG(dg) \ \mX(dx) \rightarrow _{\eta \rightarrow 0} \chi_{\mG} \). Hence there exists \( \eta_0>0 \) such that 
 \[ \int _{G\times X} c_{\eta_0} (g,x) \ \mG(dg) \ \mX(dx) <\chi _+.\]
 
 Now introduce the measure space  \( (S(X),m_{S(X)}) = (X\times G^{\N^*}, \mX\otimes \mG^{\otimes \N^*})   \) and the measure preserving map \( S (x, (g_n)_n )= ( g_1(x) ,(g_{n+1})_n) \).  Since \(\mX\) is an ergodic \(\mG\)-stationary measure by Lemma \ref{ergodicitymx},  the Random Ergodic Theorem  \ref{t: random ergodic theorem} shows that the system \( (S(X), m_{S(X)}, S)\)  is ergodic. The Birkhoff ergodic theorem applied to the \(L^{\infty}(m_{S(X)})\)-function \( (x, (g_n)_n) \mapsto c_{\eta_0} (x, g_1) \) shows that for \(\mX\)-a.e. \(x\), and \(\mG^{\N^*}\)-a.e. \((g_n)_n\), 
 \[   \frac{1}{n} \left( c_{\eta_0}(x_0, g_1)+\ldots + c_{\eta_0} (x_{n-1}, g_n) \right) \rightarrow \int c_{\eta_0}\  d\mG \ d\mX < \chi_+,\]
 where \( x_0= x\) and \( x_n =l_n (x)\). So there exists a constant such that for every \(n\geq 1\)
 \[ c_{\eta_0}(x_0, g_1)+\ldots + c_{\eta_0} (x_{n-1}, g_n) \leq  C + n\chi_+ .\]

Take a number \(0<\varepsilon <\eta_0\). We have, with the convention \(l_0= e\),
\[ \log \left( \frac{|l_n ([x, x+\varepsilon])|}{\varepsilon}\right)  = \]
\[=  \log \left( \frac{|g_1 ([x,x+\varepsilon])|}{\varepsilon } \right) + \ldots + \log \left( \frac{|g_n(l_{n-1}( [x, x+\varepsilon]))|}{|l_{n-1}([x, x+\varepsilon])|}\right), \]
so as long as none of the lengths of the intervals 
\[ l_1([x, x+\varepsilon]), \ldots , l_{n-1} ([x, x+\varepsilon])\] 
exceed \(\eta_0\), we will get the following bound 
\[ \log \left( \frac{|l_n ([x, x+\varepsilon])|}{\varepsilon}\right) \leq c_{\eta_0}(x_0, g_1)+\ldots + c_{\eta_0} (x_{n-1}, g_n)\leq C+ n\chi_+.\]
or in other words 
\begin{equation}\label{eq: exponential bound}  |l_n ([x, x+\varepsilon])| \leq \varepsilon \exp (C+\chi_+ n) .\end{equation}

Because the function \(n \mapsto \exp (C+\chi_+ n)\) tends to \(0\)  at infinity, it is bounded. It is thus possible to choose \(\varepsilon >0\) so small that the value \( \varepsilon \exp (C+\chi_+ n) \) is bounded by \(\eta_0\) for every \(n\). By induction, we will then have that \eqref{eq: exponential bound} is satisfied for every integer. 

Reasoning similarly with a negative \(\varepsilon \) finishes the proof of the Proposition. \end{proof}

\begin{proposition}\label{p: global exponential contraction}
Let \(m_\Der\) be a probability measure on the almost-periodic space invariant and ergodic under the  flow \(T_\Der\). Assume that all actions in the support of \(m_\Der\) have the global contraction property. Choose \( \chi_+\) such that \( \chi_{\mG} <\chi_+<0\). Then for \(\mX\)-generic leaf \(L \) of \(\mathcal L_X\), and two points \(x,y\in L\), we have for \(\mG^{\N} \)-a.e. \((g_n)_n\)
\[\limsup _{n\rightarrow \infty} \frac{1}{n} \log |[ l_n (x) , l_n (y)]| \leq \chi ^+ .\]
\end{proposition}

\begin{proof} 
For \(\varepsilon _0>0\) small enough, the subset \(E\subset S(X) \) formed by couples \( (x, (g_n)_n)  \) such that \(\varepsilon (x, (g_n)_n)\geq \varepsilon_0\), where \( \varepsilon (x, (g_n)_n) \) is the number constructed in Lemma \ref{l: local contraction for mG}, has positive \( m_{S(X)}\)-measure (for the construction of \( S(X)\) and \(m_{S(X)}\), see the proof of Lemma \ref{l: local contraction for mG}).  

For \(\mX\)-a.e. \(x\in X\), we know by Proposition \ref{p: quantitative contraction} that 
\[\lim _{n\rightarrow \infty} |[l_n (x-s), l_n (x +s) ]| =0,\]
for every number \(s>0\). Moreover, by ergodicity of the system \( (S(X), S, m_{S(X)})\), we also know that the trajectory \( ( S^k(x, (g_n)_n) )_k\) visit infinitely many times the set \(E\). 

Combining these two properties, we see that there exists some integer \(k\) such that at the same time the length of the interval \([l_k (x-s), l_k (x +s) ]\) is smaller than \(\varepsilon_0\), and the iterate \( S^k (x, (g_n)_n) \) belongs to the set \(E\). From this we deduce easily    the estimates 
\[ \limsup _{n\rightarrow \infty} \frac{1}{n} \log |[ l_n (x-s) , l_n (x+s)]| \leq \chi ^+ ,\]
and the conclusion follows.
\end{proof}


\section{Global contraction: deterministic estimates}\label{globalcontraction2}

The goal of this section is to prove that for each Weyl chamber \(\mathcal W\), there exists an open half space contained in \(\liea\) whose elements have the global contraction property with respect to \(\mX^{\mathcal W}\). We prove moreover that it intersects the interior of \(\mathcal W\).  See corollary \ref{c: a half-space of global contraction}.

\subsection{Lyapunov functional for the $A$-action on $X$ and local contraction}\label{s: lyapunovA}

In this section we assume that $\mX$ is the $\mG$-stationary measure on the suspension $X$ corresponding to a probability measure on $\Der$ which is invariant and ergodic for the almost-periodic flow, constructed in section \ref{sss: mX}. Recall that given a Weyl chamber \(\mathcal W \subset \liea\) we have a corresponding measure  \(\mX^{\mathcal W}\) which is $P_{\mathcal W}$-invariant and $P_{\mathcal W}$-ergodic by Lemma \ref{ergodicitymx}.  Recall the Lyapunov cocycle defined in \ref{d: Lyapunov cocycle}, which is well-defined for $\mX ^{\mathcal W}$ a.e. $x \in X$ by Lemma \ref{l: PW invariant measure}. We define the Lyapunov exponent for the $A$-action relative to \(\mX^{\mathcal W}\) to be the function $\chi _{\mathcal W}: \liea \to \R$ defined by: 

\[ \chi_{\mathcal W} (a)  := \int_X c(e^a ,x)  \ \mX^{\mathcal W} (dx) \]

The $A$-invariance of $\mX ^{\mathcal W}$ and the cocycle property implies that $\chi_{\mathcal W}$ is a linear functional on $A$. 


\begin{lemma}\label{localcontra} Suppose \(\mX^{\mathcal W}\) is a $P_{\mathcal W}$-invariant $P_{\mathcal W}$-ergodic measure on $X$  and  assume that the function $ x \to c(g,x)$ is \(\mX ^{\mathcal W}\) integrable for every $g \in G$. Let $b \in \liea $ such that $\chi_{\mathcal W} (b) < 0$. Then for $\mX ^{\mathcal W} $ a.e. $x \in X$, \(e^b\) has the local contraction property with respect to  \(\mX^{\mathcal W}\). 
\end{lemma}

Notice that the integrability condition is fulfilled thanks to the Harnack inequality,  see the fourth item of Proposition \ref{p: properties of suspension}. The proof is similar to the proof of Lemma \ref{l: local contraction for mG}, but as the measure $\mX^{\mathcal W}$ is not  $e^b$-ergodic a priori we need the following:

\begin{proposition}\label{lyapunovergodiccomp} For a.e. \(e^b\)-ergodic component of \(\mX^{\mathcal W}\), we have $$ \int_X c(e^b ,x)  \ d\nu (dx) =  \chi_{\mathcal W} (b) .$$ Therefore, this integral does not depend on the $e^b$-ergodic component of $\mX^{\mathcal W}$.  
\end{proposition}

\begin{proof}


Denote by \( \mathcal{M}^{b}(X) \) the set of \(e^b\)-invariant probability measures, and let  $(X, \mX^{\mathcal W} ) \rightarrow (\mathcal{M}^{b} (X), \lambda)$ be the ergodic decomposition of $\mX ^{\mathcal W}$ given in \ref{ergodicdecompinv}. We have a continuous action of the centralizer $C_b$ of \(e^b\) in \(G\) on $\mathcal{M}^{b} (X)$ coming from the action of $G$ on $X$. 
By unicity of the ergodic decomposition, this action preserves \(\lambda\). 

Observe that $P = (C_b\cap P)N_a$. By the Mautner phenomenon, see for instance \cite[Chapter 11]{Witte4} for this very classical topic, if  $\nu$ is $e^b$-invariant, then $\nu$ is $N_b$-invariant. Therefore as our measure $\mX^{\mathcal W}$ is $P$-ergodic, the measure $\lambda$ must be $P\cap C_b$-ergodic. 

For a \(e^b\)-invariant measure $\nu \in \mathcal{M}^{b}(X)$, define  $$ \chi_{\nu}(b) := \int_X c(e^b , x)  \ \nu (dx).$$
We claim that if \(\nu\) is \(e^b\)-ergodic, then $\chi_{g_{*}\nu}(b) = \chi_{\nu}(b)$ for $g \in C_b$. Indeed, by Birkhoff's ergodic theorem for $\nu$ a.e. $x \in X$, we have $$\chi_{\nu}(b) = \lim \frac{c(e^{nb}, x)}{n} \text{ and } \chi_{\nu_{*}}(b) = \lim \frac{c(e^{nb}, gx)}{n}.$$ Using the cocycle property and the fact that $g$ commutes with $b$,  we have $$ \lim \frac{c(e^{nb} , gx)}{n} = \lim \frac{c(g,e^{nb}(x)) + c(e^{nb}, x) - c(g, e^{nb}(x))}{n} = \lim \frac{c(e^{nb}, x)}{n} $$ so \(\chi_{g_*\nu} (b) =\chi _{\nu} (b)\) as desired.

Therefore the function $\nu \to  \chi_{\nu}(b)$ is $P\cap C_b$-invariant, and as $\lambda$ is $P \cap C_b$ ergodic, it must be constant.
\end{proof}

Lemma \ref{localcontra} will follow immediately from the following, the proof follows closely the proof of Lemma \ref{l: local contraction for mG}.

\begin{lemma}\label{l: local contraction for A}
Suppose \(\mX^{\mathcal W}\) is a $P_{\mathcal W}$-invariant ergodic measure on $X$ and $a \in \liea$ such that $\chi_{\mathcal W}(a) < 0$ . Choose \( \chi_+\) such that \( \chi_{\mathcal W}(a) <\chi_+<0\). Then, for \(\mX^{\mathcal W}\)-a.e. \(x\in X\), there exists \(\varepsilon=\varepsilon (x)  >0\) such that 
\[ \limsup _{n\rightarrow \infty} \log | e^{na} ([x-\varepsilon, x+\varepsilon ]) | \leq \chi _+.\]
\end{lemma}

\begin{proof}
Let $\nu$ be an $e^a$-ergodic component of $\mX^{\mathcal W}$. We introduce the family of functions, for \(\eta >0\),
\[ c_\eta (x) = \sup _{0< |\epsilon| < \eta } \log \left( \frac{e^a(x+\varepsilon)-e^a(x)}{\varepsilon}\right) .\]
Proposition \ref{p: properties of suspension} shows that those functions belong to \(L^\infty (\mX) \) and that they converges simply to \( c\) when \(\eta\) tends to \(0\). The Lebesgue dominated convergence theorem and  Proposition \ref{l: local contraction for mG} show that 
 \(\int_{X} c_\eta (x) \ \nu(dx) \rightarrow _{\eta \rightarrow 0} \chi \). Hence there exists \( \eta_0>0 \) such that 
 \[ \int _{X} c_{\eta_0} (x) \nu(dx) <\chi _+.\]
 
 The Birkhoff ergodic theorem applied to the \(L^{\infty}(\nu)\)-function \( x \mapsto c_{\eta_0} (e^a, x) \) shows that for \(\nu\)-a.e. \(x\), 
 \[   \frac{1}{n} \left( c_{\eta_0}( x)+\ldots + c_{\eta_0} ( e^{(n-1)a}(x)) \right) \rightarrow \int c_{\eta_0}(x)\  d\nu < \chi_+,\]
so there exists a constant such that for every \(n\geq 1\)
 \[ c_{\eta_0}( x)+\ldots + c_{\eta_0} ( e^{(n-1)a}(x)) \leq  C + n\chi_+ .\]

Take a number \(0<\varepsilon <\eta_0\). We have 
\[ \log \left( \frac{|e^{na }([x, x+\varepsilon])|}{\varepsilon}\right)  = \]
\[=  \log \left( \frac{|e^a ([x,x+\varepsilon])|}{\varepsilon} \right) + \ldots + \log \left( \frac{|e^{na}([x, x+\varepsilon])|}{|e^{(n-1)a}([x, x+\varepsilon])|}\right), \]
so as long as none of the lengths of the intervals 
\[ [x, x+\varepsilon], \ldots , e^{(n-1)a} ([x, x+\varepsilon])\] 
exceed \(\eta_0\), we will get the following bound 
\[ \log \left( \frac{|e^{na}([x, x+\varepsilon])|}{\varepsilon}\right) \leq c_{\eta_0}( x)+\ldots + c_{\eta_0} ( e^{(n-1)a}(x))\leq C+ n\chi_+.\]
or in other words 
\begin{equation}\label{eq: exponential bound}  |e^{na}([x, x+\varepsilon])| \leq \varepsilon \exp (C+\chi_+ n) .\end{equation}

Because the function \(n \mapsto \exp (C+\chi_+ n)\) tends to \(0\)  at infinity, it is bounded. It is thus possible to choose \(\varepsilon >0\) so small that the value \( \varepsilon \exp (C+\chi_+ n) \) is bounded by \(\eta_0\) for every \(n\). By induction, we will then have that \eqref{eq: exponential bound} is satisfied for every integer. 

Reasoning similarly with a negative \(\varepsilon \) finishes the proof of the Proposition. \end{proof}


\subsection{Global contraction for the central direction in a Weyl chamber}


We prove here that the element $\central ^{\mathcal W}\in \liea $ (which is determined by the random walk given by $\mG$, see Theorem \ref{centraldirection}) has the global contraction property with respect to  $\mX^{\mathcal W}$ and also that the Lyapunov exponent $\chi_{\mathcal W}(\central ^{\mathcal W})$ is negative.

 
\begin{lemma}\label{relationexponents}
Let $\chi_{\mathcal W}: \liea \to \R$ be the Lyapunov functional defined in \ref{s: lyapunovA} and the Lyapunov exponent $\chi_{\mG}$ with respect to $\mG$ defined by \ref{globallyapunov}. We have: 
$$\chi(\central ) = \chi_{\mG} = \int c(g,x)  \ d\mX \ d\mG.$$
In particular, $\chi_{\mathcal W} \in  \liea^*$ is a non-zero linear functional on $\liea$ and $\chi_{\mathcal W} (\central ^{\mathcal W}) < 0$ if $\mX$ is not $G$-invariant. 
\end{lemma}

\begin{proof}

For convenience, we will denote \(\kappa= \kappa^{\mathcal W}\) and $\central = \central ^{\mathcal W} $.

We have by Proposition \label{lyapunovergodiccomp} and Birkhoff's ergodic theorem that for $\mX^{\mathcal W}$ a.e. $x \in X$, $ \chi_{\mathcal W}(\central) = \lim \frac{1}{n} c(e^{n\central}, x)$.  We will therefore show that  $\chi_{\mG} = \lim_{n\rightarrow +\infty}  \frac{1}{n} c(e^{n\central}, x)$ for $\mX^{\mathcal W}$ a.e. $x \in X$.

We first show the following: 

\begin{proposition}\label{asdasd} For $\mX^{\mathcal W}$ a.e. $x \in X$ and Haar $\HaarG$ a.e. $g \in G$, we have  $$\lim_{n \to \infty} \frac{1}{n} c(e^{n\central}, gx) = \chi_{\mG}.$$
\end{proposition}

\begin{proof}
Using the Iwasawa decomposition of $G$ (see subsection \ref{iwasawa}), we write $G = NAK$ and we can suppose that the horospherical unstable subgroup $L_a$ coincides with $N$. Observe that for $g= nak$, where $n,a,k$ belong to $N, A, K$ respectively we have that  $\lim_{n\rightarrow +\infty} \frac{1}{n} c(e^{n\central}, gx) =  \lim \frac{1}{n} c(e^{n\central}, kx)$ and therefore it is enough to prove our assertion for $\HaarK$ a.e. $k \in K$.

Birkhoff ergodic Theorem applied to the random walk on $G$ (recall that $\mX$ is ergodic) shows that for $\mG^{\otimes \N^*}$ a.e. $\omega = (g_1, g_2, ....)$ we have that for $\mX$ a.e. $x \in X$, denoting as usual \(l_n=g_n\ldots g_1\), 
$$ \lim_{n\rightarrow +\infty} \frac{1}{n} c(l_n , x) = \chi_{\mG}.$$

Combining this fact with propositions \ref{centraldirection}, \ref{expconvergencek} and \ref{goodtracking}, we have that for Haar  $\HaarK$ a.e. $k \in  K $, there exists $\omega = (g_1, g_2,..)$, such that if we let $$l_n = k'(\omega,n) \exp (\kappa(l_n)) k(\omega, n),$$ then $k(\omega, \infty) = k$ and the following three conditions hold:

\begin{enumerate}
\item  $\lim_{n\rightarrow+\infty} \frac{1}{n} c(l_n, x) = \chi_{\mG}$
\item  $\lim_{n\rightarrow +\infty}  \frac{1}{n} \kappa(l_n) = \central $ 
\item  $\lim_{n\rightarrow +\infty} \frac{1}{n} d( e^{n\central} k(\omega, n), e^{n\central} k) = 0$
\end{enumerate}

By Harnack's inequality, there exists $\xi > 0$ such that for any $x \in X$ we have $$c(g,x) \leq \xi d(g, e)$$ so 
 $$ \big | c(e^{n\central} k(\omega, n), x)-  c(e^{n\central} k ,x)  \big |  \leq \xi d( e^{n\central} k(\omega, n), e^{n\central} k)$$
 
Therefore a simple calculation shows that: $$\lim_{n \to \infty }\frac{1}{n}c(e^{n\central}, kx) = \chi_{\mG}$$ as we wanted. Observe that we proved this is true for $\mX$ a.e. $x \in X$, but as $\mX = \int_K k_{*} \mX^{\mathcal W} \  \HaarK (dk) $, this is also true for $\mX^{\mathcal W}$ a.e. $x \in X$.

\end{proof}

We can continue with the proof of Lemma \ref{relationexponents}. Recall that $C_{\central}$ denotes the centralizer in $G$ and $L_{\central}, N_{\central}$  the horospherical stable and unstable groups defined in \ref{horospherical}. By Proposition \ref{lufactor} the function $\mathcal{F}: L_{\central} \times C_{\central} \times N_{\central} \to G$ given by $\mathcal{F}(l,c,n) = lcn$ is a diffeomorphism from a neighborhood of the identity in $L_{\central} \times C_{\central} \times N_{\central}$ onto a neighborhood of  $G$ and as $e^{m\central } (lcn)$ is at bounded distance from $e^{m\central} (n)$ (independent of $m > 0$), we can conclude from proposition \ref{asdasd} that for $\mX ^{\mathcal W}$ a.e. $x \in X$  and $\kappa_{N_{\central}}$ a.e. $n \in N_{\central}$ we have: 

$$\lim_{m \to \infty }\frac{1}{m}c(e^{m\central}, nx) = \chi_{\mG}$$

We can now take a measurable partition $\mathcal{P}$ of $X$ which is subordinate to the $N_{\central}$-orbits whose existence is given by \ref{nicepartition2}, see \ref{leafwise}. Recall that the atoms of $\mathcal{P}_x$ are $N_a$-plaques for $\mX^{\mathcal W}$- a.e. $x \in X$ and because the measure $\mu$ is $N_{\central}$-invariant, the conditional measures $\mu^\mathcal{P}_x$ in $P_x$ coincide with the normalized Haar $\kappa_{N_{\central}}$ measure on $\mathcal{P}_x$ by Proposition \ref{leafwisehaar} (Observe that we are identifying $N_{\central}$ with its $x$-orbit). As $\lim_{m \to \infty }\frac{1}{m}c(e^{m\central}, nx) = \chi_{\mG}$ for $n \in N_{\central}$, we should have this is also true for $\mX ^{\mathcal W}$-a.e. $x \in X$ by desintegration of measures.

\end{proof}

Using a similar argument, we can use the global exponential contraction property for the action of the random walk on $G$ in $\Der$ given by Proposition \ref{p: global exponential contraction} and Theorem \ref{goodtracking} to prove the following:

\begin{proposition}[Global contraction for $\central ^{\mathcal W}$]\label{contractionamg}
The element \(\central ^{\mathcal W}\in \mathcal W\) has the global contraction with respect to \(\mX^{\mathcal W}\).
\end{proposition}



\subsection{Global contraction in a half-space of \(\liea\)}

\begin{proposition}\label{localtoglobalcontraction}
Suppose $\mu$ is an $A$-invariant probability measure on $X$ and let $a, b \in \liea$. Suppose $a$ has the global contraction property with respect to $\mu$. If $b$ has the local contraction property with respect to $\mu$, then $b$ has the global contraction property with respect to $\mu$.
\end{proposition}

\begin{proof}

For every $x \in X$, define the interval in the \(\mathcal L_X\) leaf \(L_x\) of \(x\) by 
$$S_x := \{y \in L_x   \ |  \ y \geq x \text{ and }  \lim_{m \to \infty} d_{e^{mb} L_x} (e^{mb}(x), e^{mb}(y)) = 0\}.$$
Since \( a\) and \(b\) commute, and that the family of distances \( \{d_L\}_{L\text{ leaf of } \mathcal L_X}\) satisfy the Harnack inequality, we have 
\begin{equation}\label{eq: invariance under a} e^a S_x = S_{e^a x} \text{ for every } x\in X .\end{equation}

Consider the measurable function \( l : X\rightarrow [0, +\infty ] \) defined by \(l(x) = m_{L_x} (S_x)\). 
Since \(a\) has the global contraction property, we know that  for \(\mu\)-a.e. \(x\), either \(l (x ) = \infty\), or we have
\[ l(e^{na} x) =|m_{L_{e^{na} x}}( S_{e^{na} x}) |= |m_{L_{e^{na} x}}(e^{na} (S_x)) |\rightarrow _{n\rightarrow +\infty} 0.  \]  
Poincar\'e recurrence applied to the \(\mu\)-preserving transformation \(e^a\) shows that \(l\) takes \(\mu\)-a.s. the values \(0\) or \(\infty\). But we know that \(\mu\)-a.e. \(l\) takes positive value (local contraction property of \(b\) wrt \(\mu\)) so it takes \(\mu\)-a.e. the value \(+\infty\).

Reasoning similarly with the sets \[ \{y \in L_x   \ |  \ y \leq x \text{ and }  \lim_{m \to \infty} d_{e^{mb} L_x} (e^{mb}(x), e^{mb} (y)) = 0\},\] we conclude that \(b\) has the global contraction property. 
\end{proof}

\begin{corollary} \label{c: a half-space of global contraction}
Given a Weyl chamber \(\mathcal W\), the Lyapunov exponent \(\chi_{\mathcal W} \in \liea^*\) is non zero, and every element \( a\in \liea\) lying in the half-space  \( \{\chi_{\mathcal W} <0 \} \) has the global contraction with respect to \( \mX^{\mathcal W}\).
\end{corollary}

\begin{proof}
The element \( \central \in \liea\) satisfies \( \chi _{\mathcal W} (\central ) <0\) by Lemma \ref{relationexponents} and has the global contraction with respect to \(\mX^{\mathcal W}\) by Proposition \ref{contractionamg}. Applying Proposition \ref{localtoglobalcontraction} concludes the proof.
\end{proof}


\section{Propagating invariance}\label{maintheoremsimplecase} 

In this section we will prove Theorem \ref{t: no orderable lattice 2} in the case where $G$ is simple. The proof is by contradiction. 
Suppose $\Gamma$ is a lattice in a semi-simple real Lie group $G$ of finite center and rank at least two without compact factors, and assume $\Gamma$ is left orderable. So far, we have shown this implies the existence of a compact space $\Der$ with a one-dimensional lamination and a $\Gamma$-action preserving each leaf.  As before, we let $X$ be the corresponding suspension space, which is a $G$-space and we have constructed an ergodic $\mG$-stationary measure $\mX$ for any choice of a symmetric compactly supported probability measure $\mG$ on \(G\) of type \(\type\). The measure $\mX$ is not $G$-invariant as discussed in \ref{expocontrawalk}. We will prove the following:

\begin{theorem}\label{simple1}
Under the assumptions above, we have the following:
\begin{enumerate}
\item If $G$ is simple, then $\mX$ is $G$-invariant, obtaining a contradiction.
\item If $G$ is the almost product of $G_1, G_2, ...G_n$ simple real Lie groups with $n\geq 2$, then either $\mX$ is $G$-invariant or up to rearranging the factors, $G_1$ has rank one and $\mX$ is $G_i$-invariant for all $i\geq 2$. 
\end{enumerate}
\end{theorem}

Observe that Theorem \ref{simple1} implies our main result (Theorem \ref{t: no orderable lattice 2}) in the case where $G$ is simple and has finite center. We proceed to give an outline of the proof of Theorem \ref{simple1}:

For any choice of Weyl chamber $\mathcal{W}$ and its corresponding minimal parabolic subgroup $P_{\mathcal{W}}$,  we have a unique $\PW$-invariant probability measure $\mX^{\W}$ satisfying $\mX^{\W} = \int_K k_{*}\mX^{\W} d \kappa_{K}$. This gives a finite set of measures on $X$ (depending on a choice of simple roots of the root system of $\lieg$) and we will show that some of these probability measures coincide, enabling us to deduce that $\mX^{\W}$ is invariant by other parabolic subgroups and having invariance by enough of them will guarantee that $\mX^{\W}$  is  $G$-invariant and that $\mX^{\W} = \mX$.

If we let $\W_1, \W_2$ be two Weyl chambers and its corresponding minimal parabolic subgroups $P_{\W_1}, P_{\W_2}$, the probability measures $\mX^{\W_i}$ are related via elements of the Weyl group in a very simple way: if $k \in K$ is an element in the Weyl group (an element in $\mathcal{N}_K(\liea)/\mathcal{Z}_K(\liea)$ as discussed \ref{weylgroup}) such that $kP_{\W_1}k^{-1} = P_{\W_2}$, then $k_{*}\mX^{\W_1} = \mX^{\W_2}$. Observe also that all such measures on $X$ are $A$-invariant. 

We will show that if $\W_1$ and $\W_2$ are adjacent Weyl chambers, then $\mX^{\W_1} = \mX^{\W_2}$, provided there exists a singular element $a \in A$ in $\W_1\cap\W_2$ which has the global contraction property with respect to at least one of the measures $\mX^{\W_1}, \mX^{\W_2}$, this is the fundamental lemma which we will prove in the following section. Using this lemma and moving across Weyl Chambers we will show that enough of these measures coincide, obtaining Theorem \ref{simple1}.

\subsection{The fundamental lemma: equality of \(P_\mathcal W\)-measures corresponding to adjacent Weyl chambers}

Recall that a Weyl chamber $\W$ corresponds to a choice $\Pi$ of simple roots of the root system of $\lieg$ and that two Weyl chambers are said to be adjacent if the choices of root system $\Pi_1, \Pi_2$ differ only in one element in the root system. More precisely there exists a root $\beta$ such that $\Pi_2 = (\Pi_1 \setminus {\beta}) \cup {-\beta}$ and if we let $k$ be the element of the Weyl group (as an element of $K$) corresponding to reflection in the kernel of $\beta$, we have $k\W_1k^{-1} = \W_2$ and $kP_{\W_1}k^{-1} = P_{\W_2}$.

\begin{lemma}\label{hard1} Let $\W_1, \W_2$ be two adjacent Weyl chambers, and suppose there exists a non-zero $a \in \W_1 \cap \W_2$, such that $\exp a$ has the global contraction property with respect to $\mX^{\W_1}$, then $\mX^{\W_1} = \mX^{\W_2}$.
\end{lemma}

\begin{proof}

Let $k \in \mathcal{N}_K(\liea)/\mathcal{Z}_K(\liea)$ be the element in the Weyl group corresponding to the reflection on the common wall between $\W_1, \W_2$. Observe that as $a \in \W_1\cap\W_2$, then $k\exp(a)k^{-1} = \exp(a)$. We have that $k_{*}\mX^{\W_1} = \mX^{\W_2}$,  and there is a correspondence between the ergodic decomposition of $\mX^{\W_1}$ and $\mX^{\W_2}$, moreover if we let $\nu$ be an $\exp(a)$-ergodic component of $\mX^{\W_1}$, $k_{*}\nu$ is an $\exp(a)$-ergodic component of $\mX^{\W_2}$. So to finish the proof the lemma of \ref{hard1}, we only need to show the following:

\begin{proposition} For a.e. $\exp(a)$-ergodic component $\nu$ of $\mX^{\W_1}$, we have $\nu = k_{*}\nu$

\end{proposition}

Before proving the previous proposition, we will need some basic facts:

\begin{proposition}\label{SRB3}
If  $\nu$ is a probability measure in $X$ which is invariant by $\exp(a)$ and $N_a$, then for $\nu$ a.e. $x \in X$ we have: 
For Haar a.e. $n \in N_a$, the point $nx$ satisfy Birkhoff's ergodic Theorem for $\exp(a)$ and any function $f \in L^1(X)$. 
\end{proposition}

\begin{proof}
The proof is similar to the proof of Propositions  \ref{asdasd} and  \ref{contractionamg}. One takes a measurable partition of $X$ subordinate to the $N_a$-orbits and use the fact that conditional measures with respect to this measurable partition come from the Haar measure $N_a$.
\end{proof}

Recall that for an element $a \in \liea$, $N_a, L_a$ denote the horospherical subgroups and $C_a$ denotes the centralizer of $\liea$. Also, for a closed connected subgroup $H \subset G $ and $\e>0$,  $H^{\e}$ denotes the set of $h \in H$ with $d_G(h, 1_G) < \e$.

\begin{proposition}\label{smoothlu} There exists $\e> 0$ such that for any $a \in A$,  if  $d_G(g, 1_G) < \e$, there are unique smooth functions $\mathcal{N}^g: N_a^{\e} \to N_a$, $\mathcal{L}^g: N_a^{\e} \to L_a$ and $\mathcal{C}^g: N_a^{\e} \to C_a$ that satisfy $pg = \mathcal{L}^g(p)\mathcal{C}^g(p)\mathcal{N}^g(p)$ for every $p \in G$ and $\mathcal{N}_g$ is a diffeomorphism onto its image. Moreover $ \mathcal{L}^g,  \mathcal{C}^g,  \mathcal{N}^g$ depend smoothly on $g$.

\end{proposition}

\begin{proof}

Recall that by Proposition \ref{lufactor} the function $\mathcal{F}: L_a \times C_a \times N_a \to G$ given by $\mathcal{F}(l,c,n) = lcn$ is a diffeomorphism from a neighborhood of the identity in $L_a \times C_a \times N_a$ onto a neighborhood of $G$. Therefore we can define for $g$ fixed $$(\mathcal{L}^g(p),  \mathcal{C}^g(p),  \mathcal{N}^g(p)) := \mathcal{F}^{-1}(pg),$$ this is well defined for $p, g$ close to the identity $1_G$.

Clearly $ \mathcal{L}^g,  \mathcal{C}^g,  \mathcal{N}^g$ depend smoothly on $g$ and moreover if $g = 1_G$, the identity of $G$, then $\mathcal{N}^{1_G}$ is the identity map and $\mathcal{L}^{1_G}, \mathcal{C}^{1_G}$ are the constant function $1_{G}$ and so $\mathcal{N}_g$ is a diffeomorphism onto its image when $g$ is close to $1_G$.

\end{proof}

We can now begin the proof of our key Proposition \ref{hard1}:

\begin{proof}

Let $f$ be a continuous function on $X$. We will show that for any $\e > 0 $, we have $$ \left \lvert  \int_X f d\nu - \int_X f d k_{*}\nu \right \rvert\leq \e$$ and this is enough to show that $\nu = k_{*}\nu$.

Let $\e_2>0$ to be determined later. If $\pi: X \to G/\Gamma$ denotes the natural projection, we have that $\pi_{*}(\nu) = \kappa_{G/\Gamma}$ for a.e. ergodic component $\nu$, this follows from Moore ergodicity Theorem as $\pi_{*}(\mX^{\W_1}) = \kappa_{G/\Gamma}$ and $\kappa_{G/\Gamma}$ is $\exp(a)$-ergodic. Therefore, we can find two points $x_1, x_2 \in X$  which satisfy Birkhoff's ergodic theorem for $\nu, k_{*}\nu$ respectively and such that if we let  $g_i = \pi(x_i)$, then:

\begin{enumerate}

\item For $i=1,2$, $d_{G/\Gamma}(g_i, 1_{G/\Gamma}) < \e_2$ so Proposition \ref{smoothlu} holds.
\item By Proposition \ref{smoothlu}, the subset $\mathcal{N}^{g_1}(N_a^{\e_1}) \cap \mathcal{N}^{g_2}(N_a^{\e_1})$ of $N_a$ contains an open $N_a$-neighborhood of $1_{G}$ for some $\e_1 >\e_2$ depending on $\e_2$ and with \(\lim_{\e_2 \to 0} \e_1 = 0\)).
\item If $\e_1>0$ is small enough, by Proposition \ref{smoothlu}, both $ \mathcal{L}^{g_1} (N_a^{\e_1}) $ and $\mathcal{C}^{g_1}(N_a^{\e_1})$ are contained in $L_a^{\e_{0}/4}$ and $C_a^{\e_{0}/4}$ respectively for some $\e_0 >\e_1$ depending on $\e_1$ and with \(\lim_{\e_1 \to 0} \e_0 = 0\)).
\item There exists $S \subset N_a^{\e_0}$ of full Haar measure, such that for $n \in S$, the points $nx_1, nx_2$ satisfy Birkhoff's ergodic Theorem for $\nu, k_{*}\nu$ respectively.
\item $x_1, x_2$ are in the same $G\mathcal L_X$ leaf of $X$. (This follows because $x_1, kx_1$ are generic points for $\nu$ and  $k_{*}\nu$)

\end{enumerate}

Therefore we can find $n_1, n_2 \in S$, such that $\mathcal{N}^{g_1}(n_1) = \mathcal{N}^{g_2}(n_2)$ by 2) and 4) and such that $d(\mathcal{C}^{g_1}(n_1), \mathcal{C}^{g_2}(n_2)) < \e_0/2$ by 3). We therefore have $n_1g_1 = lcn_2g_2$ for some $l \in L_a^{\e_0/2}$ and $c \in C_a^{\e_0/2}$ and as $d_{G/\Gamma}(1_G, \exp(ma)(lc)) < \e_0/2$ for $m$ sufficiently large, we have $$d_{G/\Gamma}(\exp(ma)(n_1g_1), \exp(ma)(n_2g_2)) < \e_0/2$$ for $m$ large enough.

Moreover, using the global contraction property with respect to $\nu$ and 5) we have that for $m$ sufficiently large $$d_X(\exp(ma)(n_1x_1), \exp(ma)(n_2x_2) ) < \e_0$$

Given $\e > 0$, by choosing $\e_0>0$ small enough and using the uniform continuity of $f$  we can guarantee that $|f(\exp(ma)(n_1x_1)) -  f(\exp(ma)(n_2x_2)) | < \e$ for $m$ sufficiently large, and in particular
$$ \limsup_{N \to \infty} \frac{1}{N} \left \lvert \sum_{m=1}^{N} f(\exp(ma) (n_1x_1)) -   \sum_{m=1}^{N} f(\exp(ma) (n_2x_2)) \right\rvert \leq \epsilon$$
As $n_1x_1$, $n_2x_2$ satisfy  Birkhoff's ergodic Theorem for $\nu, k_{*}\nu$ by (4) we conclude $|\int_X f d \nu - \int_X f d k_{*}\nu | \leq \e$ and we are done.

\end{proof}
\end{proof}

\subsection{Proof of the theorem in the case where $G$ is simple}

We begin the proof of Theorem \ref{simple1}. The following proposition summarizes some of the properties we have proved in previous sections.

\begin{proposition}\label{properties1} For any choice of Weyl Chamber $\W$ (A choice of simple roots $\Pi$ in the root system $\Sigma$ of $\lieg$) and its corresponding minimal parabolic group $\PW$, there exists a unique $\PW$-invariant probability measure $\mX^{\W}$ in $X$ such that $\mX = \int_K k_{*}\mX^{\PW} d\kappa_K$ and a unique linear functional $\chi_{\W}: \liea \to \R$ satisfying the following:

\begin{enumerate}

\item If $b \in \liea$ satisfy $\chi_{\W}(b) < 0$ and $N_{b} \subset \PW$, then $b$ has the global contraction property with respect to $\mX^{\W}$.
\item\label{nontrivial} There exists $a \in \liea$ which lies in the interior of the Weyl chamber $\W$ such that $\chi_{\W}(a) < 0$.
\item For any two choices of Weyl chamber $\W_1, \W_2$, there exists  $k \in \mathcal{N}_K(\liea)/\mathcal{Z}_K(\liea)$ such that $kP_{\W_1}k^{-1} = P_{\W_2}$ and  $k_{*}\mX^{\W_1} = \mX^{\W_2}$ and for all $a \in A$, $\chi_{\W_2}(a) = \chi_{\W_1}(kak^{-1})$.

\end{enumerate}

\end{proposition}




The walls of a Weyl chamber are defined in the following way: A Weyl chamber is in correspondance with a choice of simple roots $\Pi$ of the root system $\Sigma$ of $\lieg$. For each $\lambda \in \Pi$, there is a corresponding wall defined as $$ \W(\lambda) := \ker (\lambda) \cap \W$$

We will need the following fact:

\begin{proposition} For every $\lambda \in \Pi$, $\W(\lambda)$ has non-empty interior in $\ker \lambda$. 
\end{proposition}

We will denote the interior of a subset $C$ of a topological space by $\text{Int}(C)$. We continue with the proof of Theorem \ref{simple1}. Let fix a choice of a Weyl chamber $\W$ of $\lieg$ and $\Pi$ be its corresponding choice of simple roots.

\begin{proposition}\label{simplecase1} Suppose that $\ker(\chi_{\W}) \cap \text{Int} (\W) = \emptyset$ and for every $\lambda \in \Pi$, the intersection $\ker (\chi_\W) \cap \text{Int}(\W(\lambda))= \emptyset$,  then $\mX^{\W}$ is $G$-invariant. 

\end{proposition}

\begin{proof}

By Proposition \ref{properties1}, there exists $a \in \W$ such that $\chi_{\W}(a) < 0$, from our hypothesis this must be true also for every $a$ in the interior of any wall of the Weyl chamber. Let $\beta \in \Pi$ and choose $a \in \text{Int} (W_{\beta})$. Observe that $N_a \subset  P_\W$, moreover the lie algebra of $N_a$ is given by $\lien_a  = \oplus_{\lambda} \lieg_{\lambda}$ where $\lambda$ varies over all roots which are  positive sum of the roots in $\Pi \setminus \{ \beta \} $.\\

Let $k_{\beta} \in \mathcal{N}_K(\liea)/\mathcal{Z}_K(\liea)$ corresponding to the reflection in the hyperplane $\ker (\beta)$ and let $\W_\beta$ be the Weyl chamber adjacent to $\W$ with common wall $\ker{\beta}$. We have that  $P_{\W_{\beta}} =k_{\beta}\PW k_{\beta}^{-1} $. Observe that $N_a = k_{\beta} N_a k_{\beta} ^{-1}$ (this can be seen from the description of the Lie algebra of $N_a$) and also that $k_{\beta} \exp(a) k_{\beta}^{-1} = \exp(a)$.\\

Therefore  $\mX^{\W}$ and $\mX^{\W_{\beta}}$ are $N_a$-invariant probability measures on $X$ and $\exp(a)$ satisfy the global contraction property with respect to both measures. By Lemma \ref{hard1},   $\mX^{\W} = \mX^{\W_{\beta}}$. This in particular implies that if we let $U^{\beta}$ be the unipotent subgroup of $G$ whose lie algebra correspond to the lie subalgebra $\lieg_{\beta} \oplus  \lieg_{2\beta}$, then $\mX^{\PW}$ is both $U^{\beta}$ and $U^{-\beta}$ invariant as $U^{\beta} \subset \PW$ and $U^{-\beta} \subset P_{\W_{\beta}}$ and this is true for every $\beta \in \Pi$.\\

It is easy to see that $G$ is generated by the subgroups $U^{\beta}$ and $U^{-\beta}$ where $\beta$ varies over all $\beta \in \Pi$, this follows from the fact that roots are positive or negative sums of $\Pi$ and $[\lieg_{\alpha}, \lieg_{\gamma}] \subset \lieg_{\alpha + \gamma}$, see Theorem \ref{Knapp1}. Therefore $\mX^{\W}$ is $G$-invariant.

\end{proof}

We will need the following basic fact about Weyl chambers:

\begin{proposition}\label{nicechamber} For $\chi \neq 0 \in \liea^{*}$, there exists a Weyl chamber $\W'$ satisfy $$\text{Int}(\W') \subset \{a \in \liea\ |\ \chi(a) < 0\}.$$
\end{proposition}

The following is a more general version of Proposition \ref{simplecase1}.

\begin{proposition}\label{simplecase2} If $\chi_{\W} \neq c\alpha$ for every $\alpha \in \Sigma$ and $c \neq 0$, then $\mX^{\W}$ is $G$-invariant.
\end{proposition}

\begin{proof}

Set $H := \{a \in \W\ |\ \chi_\W(a) < 0\}$. By Proposition \ref{nicechamber} there exists a Weyl Chamber  $\W'$ satisfy $\text{Int}(\W) \subset H$. We will show that  $\mX^{\W} = \mX^{\W'}$. We know from Proposition \ref{properties1} there exists some $a_0 \in \W$ with $\chi_\W(a_0) < 0$. We will move along adjacent Weyl chambers intersecting $H$ from $\W$ to $\W'$ showing that the corresponding probability measures for each of the Weyl chambers must coincide, the equality of these is shown via Proposition \ref{hard1} in the same way we proved Proposition \ref{simplecase1}. The formal proof goes as follows:

We choose  $a_1 \in \text{Int}(\W')$ and a smooth path $a(t): [0,1] \to \liea$ such that: 

\begin{enumerate}
\item $a(0)= a_0$, $a(1) = a_1$ and $a(t) \in H$ for all $t \in [0,1]$.
\item For all $\lambda \in \Sigma$, $a(t)$ intersects transversely each hyperplane $\ker (\lambda)$ and the number of such intersections is finite.
\item  For all $\alpha, \beta \in \Sigma$, which are not proportional, $a(t)$ is disjoint from $\ker(\alpha) \cap \ker(\beta)$.
\end{enumerate}

The existence of such path follows easily because the intersections $\ker(\alpha) \cap \ker(\beta)$ have codimension $2$ in $H$. We will show the following statement:

\begin{proposition} For every $t \in [0,1]$, if $a(t)$ belongs to the interior of a Weyl chamber $\W_t$ then, $\mX^{\W_t} = \mX^{\W}$.

\end{proposition}

\begin{proof}

This statement is clearly true for $t = 0$. If the statement is true for $t \in [0,1]$ such that  $a(t)$ is in the interior of a Weyl chamber, then it is obviously true as we increase $t$ until $a(t)$ intersects a wall of the Weyl chamber. So we only need to show that after crossing any such wall our statement remains true, more precisely, suppose that $a(t) \in \ker(\beta)$ for some $\beta \in \Sigma$ and that there exists $\e>0$ such that $a(s)$ is contained in a Weyl chamber $\W_1$ for $ t - \e \leq s  \leq t$ and $a(s) \subset \W_2$ for  $t  \leq s \leq t+\epsilon$; we have that $\mX^{\W_{t-\e}} = \mX^{\W}$ by hypothesis and we want to show that $\mX^{\W_{t-\e}} = \mX^{\W_{t+\e}}$.\\

If we let $k_{\beta} \in \mathcal{N}_K(\liea)/\mathcal{Z}_K(\liea)$ be the element of the Weyl group corresponding to the reflection in the hyperplane $\ker (\beta)$ we have  $P_{\W_{t + \e}} =k_{\beta}P_{\W_{t - \e}}k_{\beta}^{-1} $. Observe that  $a(t)$ satisfies the global contraction property with respect to $\mX^{\W_{t - \e}}$ because $\mX^{\W_{t -\e}} = \mX^{\W}$ and so $\chi_{\W_{t - \e}}(a(t)) = \chi_{\W}(a(t))$. By Proposition \ref{hard1},   $\mX^{\W_{t - \e}} = \mX^{\W_{t + \e}}$ and we are done.

\end{proof}

We continue with the proof of Proposition \ref{simplecase2}. We have by the previous proposition that  $\mX^{\W} = \mX^{\W'}$ and therefore $ \chi_{\W'} = \chi_\W$ (by the uniqueness of the exponents). Observe that $\ker{\chi_{\W'}} \cap \text{Int} (\W') = \emptyset$ by construction and moreover $\ker(\chi_{\W'}) \cap \text{Int}(\W'(\lambda))= \emptyset$ for every $\lambda \in \Sigma$ because $\ker(\chi_{\W'}) \neq c\lambda$ for every $\lambda \in \Sigma$ by hypothesis. Therefore, we can apply Proposition \ref{simplecase1} to conclude that  $\mX^{\W'}$ is $G$-invariant.

\end{proof}

To conclude the proof of Theorem \ref{simple1} we only need to show the following:

\begin{proposition}\label{semi-simplecase} If $\chi_\W = c\alpha$ for some  $c \neq 0$ and some $\alpha \in \Sigma$, then one of the following holds:

\begin{enumerate}
\item $\mX^{\W}$ is $G$-invariant.
\item $G$ is an almost product of at least two simple Lie groups, and $\mX^{\W}$ is invariant by all the factors of $G$ perhaps except for one simple factor $H$ with $\text{Rank}(H) = 1$.
\end{enumerate}

In particular, if $G$ is simple then $\mX^{\W}$ is $G$-invariant. 
\end{proposition}

\begin{proof}

As before, let $\Pi$ be the choice of simple roots determining $\W$, for a root $\lambda \in \Sigma$ and let $U^{\lambda}$ be the unipotent subgroup corresponding to the Lie sub-algebras $\lieg_{\lambda} \oplus \lieg_{2\lambda}$. From the proof of Proposition \ref{simplecase1}, it follows that if $\beta$ is a simple root in $\Pi$ which is not proportional to $\alpha$,  then $\mX^{\W}$ is invariant by $U^{\beta}$ and $U^{-\beta}$.

We let $$\lieh := \lieg_{-2\alpha} \oplus \lieg_{-\alpha} \oplus \lieg_{\alpha} \oplus \lieg_{2\alpha} \oplus [\lieg_{-\alpha},\lieg_{\alpha}] \oplus  [\lieg_{-2\alpha},\lieg_{2\alpha}]$$

\begin{proposition} $\lieh$ is an ideal of $\lieg$ which as a real Lie subalgebra has rank equal to one.
\end{proposition}

\begin{proof}

One can prove using the fact that $[\lieg_{\beta}, \lieg_{\gamma}] \subset \lieg_{\beta+ \gamma}$ and the Jacobi identity that $\lieh$ is a Lie subalgebra.

We will show that for every $\beta \in \Pi$ not equal to $\alpha$, we have $(\alpha, \beta) = 0$, where the inner product $(,)$ comes via identifying $\liea$ and $\liea^{*}$ via the killing form. From the proof of Proposition \ref{simplecase1}, it follows that if $\beta \neq c\alpha$ is a simple root in $\Pi$, and if we let $k_{\beta} \in K$ be the element of the Weyl group corresponding to the reflection in $\ker (\beta)$, then $\mX^{\W} = \mX^{\W'}$ where $\W' := k_{\beta}\W k_{\beta}^{-1}$.

 Therefore $\chi_\W(a) = \chi_P(k_{\beta} ak_{\beta}^{-1})$ for all $a \in \liea$ and $\ker{\chi_\W}$ is invariant by the reflection on $\ker{\beta}$. The only way this can happen is if either $\chi_\W$ is proportional to $\beta$ (which can't happen because $\alpha$ is not proportional to $\beta$) or if $\ker(\chi_\W)$ contains the one dimensional subspace which is invariant under the involution; this subspace is spanned by $H_{\beta}$, the element of $\liea$ which is dual to $\beta$ via $B(H_{\beta}, .) =\beta(.)  $.

Therefore we have that $$(\beta, \alpha) = B(H_{\alpha}, H_{\beta}) = \alpha(H_\beta) = 0, $$ where the last equality comes from $\alpha$ being proportional to $\chi_\W$.

This implies in particular that for every root $\beta$ in $\Pi$ different of $\alpha$,  $\alpha + \beta$ is not a root of $\Sigma$, because otherwise $(\alpha, \alpha + \beta) = (\alpha, \alpha) \neq 0$ and therefore we have $[g_{\alpha}, g_{\beta}] \subset g_{\alpha + \beta}= \{0\}$ and so $ [g_{\alpha}, g_{\beta}] = \{0\}$.

As a consequence, because every root $\lambda \in \Sigma$ is a sum of positive or negative roots of $\Pi$, then for every root of $\Sigma$ which is not proportional to $\alpha$ we have $[\lieg_{\beta}, \lieg_{k\alpha}] = 0$ for every integer $k \neq 0$. We also have that $[\lieg_0, \lieg_{k\alpha}] \subset \lieg_{k\alpha}$, and from the Jacobi identity one can show that $[\lieg_0, [\lieg_{-k\alpha}, \lieg_{k\alpha}]] \subset [\lieg_{-k\alpha}, \lieg_{k\alpha}]$ which concludes the proof that $\lieh$ is an ideal.

It remains to show that rank $\lieh$ is equal to one. Observe that $H_{\alpha}$ is contained in $[\lieg_{-\alpha},\lieg_{\alpha}]$, by \ref{Knapp2}, and that $\lieh$ is invariant by the Cartan involution by \ref{Knapp1}. Therefore it remains to show that $\lieg_0 \cap \liea$ has dimension one. If $X, Y$ are in $g_{-\alpha}$ and $g_{\alpha}$ (or  $g_{-2\alpha}$ and $g_{2\alpha}$) then for every $H \in \liea$, we have $$B([X,Y], H)  = B(X, [Y, H]) =  B(X, \alpha(H)Y) = B(X,Y) \alpha(H)$$ and from this it follows that the subspace $[\lieg_{-\alpha},\lieg_{\alpha}] \cap \liea$ (or $[\lieg_{-2\alpha},\lieg_{2\alpha}]\cap \liea$) is equal to the one dimensional subspace determined by $H_\alpha$.   

\end{proof}

Therefore $\lieh$ is the lie algebra corresponding to a simple factor $H$ of $G$ of rank one. Also, as $\lieg$ is semisimple, we must have a decomposition  $\lieg = \lieh \oplus \liel$ where $\liel$ is the Lie algebra of a semisimple normal subgroup $L$ of $G$ and one can show easily that $\Pi' = \Pi \setminus \{\alpha\}$ is a choice of simple roots for  $\liel \cap \liea$. From the proof of Proposition \ref{simplecase1}, it follows that for every $\beta \in \Pi'$, the measure $\mX^{\W}$ on $X$ is invariant by $U^{\beta}$ and $U^{-\beta}$ and as $L$ is generated by all these subgroups $\mX^{\W}$ is $L$-invariant and we are done.

\end{proof}


\section{The semi-simple case}\label{maintheoremsemisimplecase}

In this section, we prove that the second case in Theorem \ref{simple1} leads to a contradiction, hence ending the proof that an irreducible lattice in a semi-simple Lie group with finite center is not left-orderable. Notice that, up to taking a finite cover, we can assume that \(G\) is the direct product \(G_1\times G_2 \ldots \times G_n\), \(n\geq 2\), and that the measure \(\mX\) is invariant by the action of \(G_2\times \ldots \times G_n\). 

\subsection{Controlling displacements along sequences}

Recall that we have a family of leafwise distances \( \{d^g\}_{g\in G}\) on \(\mathcal L_Z\) defined in the following way: for two points \(x,y\) belonging to a same leaf \(L\) of \(\mathcal L_\Der\), we set \[ d^g (x,y) := \left \lvert \int_x^y  \widetilde{D_{\mathcal L_\Der}} (g, z ) \ \mL (dz) \right \rvert .\] 

\begin{lemma}\label{l: only dependance g1}
The map \(g\in G\mapsto d^g\) only depends on the coordinates \(g_1\). 
\end{lemma} 

\begin{proof} Item \eqref{d: mX} in Lemma \ref{l: mX}  implies that for \( \mX\)-a.e. \(x\), the function \( D_{\mathcal L_X} (\cdot, x)\) only depends on the \(g_1\)-coordinate. In particular,  for \(x\in \Der\) belonging to a set of \( \mDer\) total measure the function \( g\mapsto \widetilde{D_{\mathcal L_Z}} (g, x)  \) only depends on the \(g_1\)-coordinate. Hence,  for a set of leaves whose union has total \(\mDer\)-measure, we have that given two points \(x,y\)  in a leaf \(L\) of that set, the function \( g\in G\mapsto d^g_L (x,y )\in [0, \infty) \) depends only on the \(g_1\)-coordinate. The  claim is a consequence of the fact that the family of leafwise distances \( d^g = \{d^g_L\}_{L \text{ leaf of } \mathcal L_\Der} \) is continuous, and by minimality of \(\mathcal L_\Der\).
\end{proof}

We denote by \( d_{S_1} \) the distance on the symmetric space \(S_1 = K_1 \backslash G_1\) associated to \( G_1\), where \(K_1\) is the maximal compact subgroup of \( G_1 \), and by \( \xi_1= \xi (G_1) \) the constant given by Theorem \ref{t: Harnack inequality}. Note that \( G_1\) acts on its symmetric space by \( g_1 (Kh_1) = Kh_1 g_1^{-1}\).

We also denote by \( p_1 : G\rightarrow G_1\) the projection \( p_1 (g_1, \ldots, g_n ) = g_1\), and for \( q = K_1 g_1 \in S_1\), we set \( d ^q  = d ^{(g_1, g_2,\ldots, g_n)}\) for any choices of \( g_2, \ldots, g_n\). This is well-defined thanks to Lemma \ref{l: only dependance g1}.

\begin{lemma} \label{l: displacement}  For every \(q \in S_1\), every \(\gamma\in \Gamma\) acts on each leaf \(L\) of \(\mathcal L_\Der\) equiped with its distance \( d^q _L \) as a \( e^{\xi_1 d_{S_1} (q ,p_1(\gamma)(q)) } \)-bi-Lipschitz map. In particular, \(\gamma\) acts as a genuine translation along the leaf \(L\) of \(\Der\) equiped with the distance \( d^q \) if \(q\) is a fixed point of \(p_1(\gamma)\).
\end{lemma}

\begin{proof}
Harnack inequality \ref{t: Harnack inequality} applied to the positive harmonic function \( q \in S_1 \mapsto d_L ^q \in \text{Prob} (\Der) \) shows that given any pair of points \( q , q' \in K\backslash G\), and any leaf \(L\) of \(\mathcal L_\Der\), we have 
\begin{equation} \label{eq: harnack distance}      e^{-\xi_1 d_{S_1} (q , q ')}  d_L^{q'}  \leq d_L^q \leq   e^{\xi_1 d_{S_1} (q, q ')}  d_L^{q'}.\end{equation}
Take \(  q'=   p_1 (\gamma) (q) \). Letting \( g=(g_1, e,\ldots, e) \), and \( g' = g \gamma ^{-1} \), we have \( g_1 = p_1 (g')\). Equation \eqref{eq: equivariance distance} shows that for every leaf \(L\) of \(\mathcal L_\Der\) and every pair of points \(x,y\in L\), we have 
\[ d^{q'}_{\gamma L} (\gamma x,\gamma y) = d^{g'} _{\gamma L} (\gamma x,\gamma y)  =  d^g _L (x,y) = d_L^{q} (x,y) . \]
Together with \eqref{eq: harnack distance}, this proves the Lemma.
\end{proof} 

\begin{definition}[Displacement] The displacement of an element \( \gamma\in \Gamma\) with respect to a continuous leafwise family of distances \( d= \{d_L\} _{L\text{ leaf of } \mathcal L_\Der}\) on \(\mathcal L_\Der\) is defined as   
\[ \delta_{d} (\gamma) := \sup _{z\in \Der}  d_{L_z} (z, \gamma (z) ) .\]   
Unless specified, we choose the distance \( d:= d^e\) to measure the displacement of an element of \(\Gamma\). 
\end{definition}

The following result is the main first step in order to establish Theorem \ref{t: no orderable lattice 2} in the case where $G$ is semi-simple. 

\begin{lemma} \label{l: bounded displacement of sequences}
There exists a neighborhood of the identity in \(G_1\) and a constant \( l >0\), such that any element in this neighborhood is a limit \(\lim _n p_1(\gamma_n) \), where \(\gamma_n\) is a sequence of elements of \(\Gamma\) whose displacements along the leaves of the lamination \(\mathcal L_\Der\) is uniformly bounded by \(l\). 
\end{lemma}

\begin{proof} The proof is divided into two cases, that embrace all the possibilities.

\vspace{0.4cm}

\noindent \textit{First case: there exists a non trivial element \( \gamma\in \Gamma\) such that \(p_1(\gamma)\) fixes a point \(q\) in the symmetric space \( S_1\). The existence of such an element can be established if \(\lieg_1\) is isomorphic to \(\mathfrak{sl}(2,\R)\).} 

\vspace{0.4cm} 

If \(\lieg_1\) is isomorphic to \(\mathfrak{sl}(2,\R)\),   the set of non trivial elements in \(G_1 \) having a fixed point in \( K_1 \backslash G_1\) form an open set, so there is such an element in the image of \(p_1(\Gamma)\) since this latter is dense in \(G_1\). 

Let us now prove that under the existence of such an element \(\gamma\in \Gamma\), Lemma \ref{l: bounded displacement of sequences} is true. Up to conjugating we can assume that \(q\) is the base point of the symmetric space \( S_1\), namely that \( q= K_1\in K_1 \backslash G_1\).

It suffices to prove that there exists a (non trivial) one parameter subgroup \( \{ M_t\} _{ t\in \R} \) of \(G_1 \) and two positive constants \(\varepsilon, l >0\), such that any \(M_t\) with \(|t|\leq \varepsilon\) is a limit \(M_t =\lim _n p_1(\gamma_n) \), where \(\gamma_n\) is a sequence of elements of \(\Gamma\) whose displacements along the leaves of \(\mathcal L_\Der\) is bounded by \(l\). 

Let \(\eta_n \in \Gamma \) be a sequence so that \(p_1(\eta_n)\) is a \(o(1/n) \)-approximation of an element of the form \( e^{ a/n } \) with \(a\in \liea_1 \), satisfying \( \text{Ad} ( p_1 (\gamma) ) a \neq a\). We then have that \(d_{G_1} (p_1(\eta_n), 1)\) and \( d_{S_1} (p_1(\gamma) p_1 (\eta_n) (q) , p_1 (\eta_n) q) \) have the order of magnitude of \(1/n\) up to multiplicative constants. 

By Lemma \ref{l: displacement}, the element \(\gamma\) acts as a translation along the leaves of the lamination \(\mathcal L_\Der\) equiped with the distance \(d^{q}\), and with respect to that distance the elements \(\eta_n\) are \(\exp(\xi_1 d_{G_1} (p_1(\eta_n), 1))\)-bilipschitz maps along each leaf of \(\mathcal L_\Der\). Hence, the commutator \([\eta_n, \gamma]= \eta_n \gamma \eta_n^{-1} \gamma^{-1} \) has a displacement bounded by 
\[ \delta _{d^q} ( [\eta_n, \gamma]) \leq (e^{\xi_1 d_{G_1} (p_1(\eta_n), 1)} - 1) \delta_{d^q}(\gamma) = O (1/n) .\]

Notice that  \(d_{G_1} ( p_1([\eta_n , \gamma)] ) , 1)  \) has the order of magnitude (up to multiplicative constants) of \(1/n\). Indeed, we have \( p_1([\eta_n,\gamma) ]) \sim _{n\rightarrow \infty} \frac{1}{n} ( a - Ad(\gamma) a) \) and \(Ad(\gamma) a - a \neq 0\). Hence we can write \( p_1([\eta_n, \gamma]) = \exp ( v_n /n)\) where \(v_n\in \lieg _1\) has a norm bounded from above and below by positive constants. Up to extracting a subsequence, we can assume that \(v_n \) converges to some non zero \( v\in \lieg _1\). 

For any \(|t|\leq 1\), we have 
\[ p_1\left( [\eta_n, \gamma] ^{\left \lfloor{n t}\right \rfloor} \right)= \exp (tv_n ) \rightarrow _{n\rightarrow \infty} \exp (t v) , \]
and 
\[ \delta_{d^q } ( [\eta_n, \gamma] ^{\left \lfloor{n t}\right \rfloor} ) \leq t n  \delta_{d^q } ([\eta_n, \gamma]) \leq O(1).\]
This ends the proof of the Lemma in this case.

\vspace{0.4cm}

\noindent \textit{Second case: the commutator group \([K_1, K_1]\) has positive dimension (We recall that the commutator subgroup of a compact Lie group is closed). In the class of rank one simple Lie groups with finite center, all have this property apart those locally isomorphic to \( \text{SL}(2,\R)\).} 

\vspace{0.4cm}

 The following statement explains that all rank one simple Lie groups apart from those locally isomorphic to \( \text{SL} (2,\R)\) satisfy the assumption of the second case.

\begin{proposition} If $G$ is a real simple Lie group with finite center of rank one and $K$ is the maximal connected compact subgroup of $G$, then either:

\begin{enumerate}
\item The Lie algebra of $G$ is $\mathfrak{sl}(2, \R)$ and $[K,K]$ is trivial.
\item $[K, K]$ is a non-trivial connected Lie subgroup of $K$.
\end{enumerate}
\end{proposition}

\begin{proof}
This is a consequence of the classification of real rank one simple Lie groups, see \cite[Appendix C.3]{Knapp}. We can assume $G$ has trivial center by taking a finite quotient and then we have the following possibilities:

\begin{enumerate}
\item If $G = SO_0(n,1), n > 2$, then $[K,K] = K = SO_0(n)$.
\item If $G = SU(n,1), n > 1$, then $K = S(U(n)\times U(1))$ and $[K,K] = SU(n)$.
\item If $G = Sp(n,1)$, then  $[K,K] = K = Sp(n)$.
\item If $G = F_4^{-20}$, then $[K,K] = K = Spin(9)$.
\end{enumerate}

\end{proof}

Now let us prove that if \([K_1,K_1]\) has positive dimension, then Lemma \ref{l: bounded displacement of sequences} is true. We will use exponential approximation of elements of \(G_1\) by elements of \( p_1(\Gamma)\), essentially due to the work \cite{Boutonnet-al}, and prove the following result, which easily implies Lemma \ref{l: bounded displacement of sequences} as we will see. 

\begin{lemma} \label{l: commutator K} 
For every \( g,g'\in K_1\), there exists a sequence \( \gamma_n \in \Gamma\) such that \(p_1(\gamma_n) \rightarrow [g,g']\), and such that 
\( \delta( \gamma_n ) \rightarrow 0\).
\end{lemma}

\begin{proof} We recall the following definitions.

\begin{definition} \label{d: length function} Let $\Gamma$ be a finitely generated group and $S$ a generating set of $\Gamma$. For $g \in \Gamma$, $l_S(g)$ denotes the word length of $g \in \Gamma$ with respect to $S$, that is, $l_S(g)$ is the smallest positive integer $n$ such that $g$ is a product of $n$ elements of $S$.
\end{definition}

\begin{definition} Let $G$ be a connected simple real Lie group. Denote by $\lieg$ the Lie algebra of $G$ and by $\text{Ad} : G \to GL(\lieg)$ its adjoint representation. We say that a subgroup $\Gamma$ has algebraic entries if there is a basis $\mathcal{B}$ of $\lieg$ such that the matrix of $\text{Ad}(g)$ in the basis $\mathcal{B}$ has algebraic entries for any $g \in \Gamma$.
\end{definition}

\begin{definition}
To avoid cumbersome notation in the proof of Proposition \ref{spectral2}, we let $d_G(,)$ denote a left-invariant Riemannian metric in $G$ and not a right-invariant as in the rest of the article.  We also let $B(g, r)$ denote a ball with center $g \in G$ and radius $r \geq 0$. For a subset $B$ of $G$, let $1_B(.)$ be the characteristic function  of $B$.
\end{definition}

\begin{proposition}\label{spectral2} Suppose $G$ is a connected real simple Lie group and $\Gamma$ is a dense finitely generated subgroup of $G$ with algebraic entries. Let $B$ be a bounded open subset of $G$. Then, there exists $\delta > 0$  and a finite generating set $S$ of $\Gamma$ such that for any $ g \in B$ and any positive integer $n$, there exists $g_n \in \Gamma$ with:

\begin{enumerate}
\item $l_S(g_{n}) < n$.
\item $d_G(g, g_n) < e^{-\delta n}$.
\end{enumerate}

\end{proposition}

\begin{remark} Proposition \ref{spectral2} is an easy corollary from Corollary H in \cite{Boutonnet-al}, where the authors generalized the work of Bourgain and Gamburd \cite{Bourgain Gamburd} in the spectral gap for finitely generated subgroups of $SU(2)$ to the setting of a non-compact simple Lie group. We should point out that we do not need the full strength of \cite{Boutonnet-al} results and for us an approximation of the order of $o(n^{-1})$ for some $\epsilon >0$ would suffice (instead of an exponentially small approximation).
\end{remark}

\begin{proof}[Proof of Proposition \ref{spectral2}] In \cite{Boutonnet-al}, they consider the (delayed) random walk in $B$ defined by a finite set $S = \{s_1,\dots s_k\}$ of $\Gamma$ as follows:  Given $x \in B$ move with probability $\frac{1}{k}$ to each of the points $h_1(x), h_2(x), \dots, h_k(x)$, where $h_i = s_i$ if $s_i(x) \in B$, and $h_i = e$, if $g_i(x) \not\in B$.\\

There is an associated transition operator $P_S: L^2(B) \to L^2(B)$ given by

$$P_S(f) := \frac{1}{k} \sum_{i=1}^{k}  1_{s_i(B)\cap B}  \ f \circ s_i  \  +   \ 1_{B \setminus S_i(B)} \  f. $$\\

It is proved in \cite{Boutonnet-al} that there exists a choice of a symmetric generating set $S$ of $\Gamma$ such that $P_S$ has a spectral gap in $L^2_0(B)$, i.e. there exists $\delta > 0$ such that for any $f \in L^2_0(B)$, $\|P_S(f)\|_{L^2(B)} \leq e^{-\delta} \| f\|_{L^2(B)}$.\\

Our desired result then holds by taking a constant $D>0$ such that \begin{equation}\label{cali}\HaarG(B(e, e^{-\delta n/D})) > e^{-n\delta},\end{equation} considering the functions   $$f_1 := 1_{B(g, e^{-\delta n/D})} - \frac{\HaarG(B(g, e^{-\delta n/D}))}{\HaarG(B)},$$ $$f_2 := 1_{B(e, e^{-\delta n/D})} - \frac{\HaarG(B(e, e^{-\delta n/D}))}{\HaarG(B)}$$ which lie in $L^2_0(B)$ and satisfy $\|f_1\|, \|f_2\| \leq 1$ and observing two facts: First, the spectral gap of $P_S$ implies \begin{equation}\label{jovita}|\langle P_S^n(f_1), P_S^n(f_2) \rangle | \leq e^{-2\delta n} \|f_1\|\|f_2\| \leq e^{-2\delta n/D}\end{equation} and second, if for any $h_1, h_2 \in \Gamma$ with $l_S(h_1), l_S(h_2) \leq n$ we have $h_1(B(g, e^{-\delta n/D})) \cap h_2(B(e, e^{-\delta n/D})) = \emptyset$, then $$|\langle P_S^n(f), P_S^n(g) \rangle |  = \HaarG(B(e, e^{-\delta n/D}))^2 $$ by an easy computation. Therefore, this last inequality cannot hold because of \eqref{cali} and \eqref{jovita} and so $h_1(B(g, e^{-\delta n/D})) \cap h_2(B(e, e^{-\delta n/D})) = \emptyset$ for some $h_1, h_2 \in \Gamma$ with $l_S(h_1), l_S(h_2) \leq n$, which implies that $h_1^{-1}h_2 \in \Gamma $ is at distance at most  $2e^{-\delta n/D}$ of $g$.

\end{proof}

Let us now conclude the proof of Lemma \ref{l: commutator K}. Take two points \(g, g' \in K_1\), and let \( \gamma_n , \gamma_n' \in \Gamma\) two sequences of elements such that 
\[     d_ G (\gamma_n , g) = O (e^{-\delta n} ) , \ l_S (\gamma_n ) = O (n) \text{ and } d_ G (\gamma_n ' , g') = O (e^{-\delta n} ) , \ l_S (\gamma_n ' ) = O (n). \]  
In particular, we have, denoting \(q= K_1 \in K_1\backslash G_1\), that 
\[ d_{S_1} (\gamma_n q, q) =O (e^{-\delta n} ) \text{ and } d_{S_1} (\gamma_n ' q, q) =O (e^{-\delta n} )\]
and this implies that \( \gamma_n \) and \(\gamma_n '\) are \( (1+ O (e^{-\delta n}))  \)-bilipschitz along each \(\mathcal L_\Der\) leaf equiped with the metric \(d= d^q\). The bound on  the translation length of \(\gamma_n \) and \(\gamma_n '\) implies that their displacements is  at most linear 
\[ \delta _q(\gamma_n ) = O (n ) \text{ and } \delta_q (\gamma_n ') = O (n)  .\]

\begin{proposition}\label{p: estimates displacement of a commutator}
Let \( h, h ' \in Homeo ^+ (\R) \) two bi-Lipschitz homeomorphisms of bounded displacements. Then, denoting by \(\xi \geq 1\) (resp. \(\delta\geq 0\)) a common Lipschitz constant (resp. bound for the displacement) for \(\{h,h', h^{-1}, (h')^{-1} \}\), the displacement of  the commutator \( [h, h'] \) is bounded by   \(2 (\xi - 1 ) \delta \). 
\end{proposition}

\begin{proof}
Given \(y,z\in \R\), we have 
\[ | (z- y ) - ( h(z) - h(y) ) | \leq  (\xi  - 1) |z-y|,\]
and the similar inequality for \(\{h',h^{-1}, (h')^{-1}\} \). For every \(x\in \R\), we then have  
\[ | (h^{-1} (h')^{-1} (x )- (h')^{-1} (x) ) - (h' h^{-1} (h')^{-1} (x) - x ) | \leq (\xi - 1)   \delta .\]
Moreover, 
\[ | (h^{-1} (h')^{-1} (x) -  h' h^{-1} (h')^{-1} (x) ) - ( (h')^{-1} (x)- h h' h^{-1} (h')^{-1} (x)) | \leq (\xi - 1) \delta\]
so  summing these two inequalities and using triangular inequality we get 
\[ |  h h' h^{-1} (h')^{-1} (x) - x |  \leq 2 (\xi - 1) \delta ,\]
and the result follows. 
\end{proof}

From Proposition \ref{p: estimates displacement of a commutator}, we deduce that 
\[ \delta ([\gamma_n , \gamma_n' ]) \leq O (n e^{-\delta n}) ,\]
while \(p_1([\gamma_n , \gamma_n ']) \rightarrow [g, g']\). 
Hence, Lemma \ref{l: commutator K} is proved. 
\end{proof}

To conclude the proof of Lemma \ref{l: bounded displacement of sequences} in the second case, notice that the subset \(H\subset G_1\) consisting of elements \(g_1\in G_1\) such that 
\begin{equation}\label{eq: approximative zero displacement} \text{there exists a sequence } \gamma_n\in \Gamma \text{ such that } p_1(\gamma_n) \rightarrow g_1\text{ and } \delta (\gamma_n ) \rightarrow 0\end{equation} 
is a closed subgroup of \(G_1\) (by subadditivity of the displacement function). Moreover, it is normalized by the lattice \(\Gamma\). 

Lemma \ref{l: commutator K} shows that commutators \([g,g']\) with \(g,g'\in K\) belong to \( H\). Hence \([K_1,K_1]\subset H\), and this latter is a closed subgroup of positive dimension. Its Lie algebra is invariant by \( \Gamma\), hence by \(G_1\) since \(\Gamma\) is a lattice. Because \(G_1\) is a simple connected Lie group, we conclude that \(H=G_1\). Thus, every element of  \(G_1\) satisfies \eqref{eq: approximative zero displacement}, and we are done.

\end{proof}

\subsection{Transportation of measures along a flow}

In the sequel we will discuss a notion of transportation of measures along the trajectories of a flow. Our point of view is based on Kantorovich's theory.

\begin{definition}  Given a pair of probability measures \(m_1, m_2\) on a compact space \( Z\) equiped with a free flow \( T=\{T^t\}_{t\in \R}\), and a non negative number \(L\), we will say that \(m_1\) can be \(L\)-transported along the flow \(T\) to  \(m_2\) if there exists a probability measure \( M\) on the space \( Z \times z\) whose support is contained in the closed subset 
\begin{equation}\label{eq: L neighborhood} \Delta^L:= \{ (x_1,x_2)\in Z \times Z \ :\ \exists l\in [-L,L] \text{ such that } x_2 = T^l (x_1) \} ,\end{equation} 
and such that \( m_i = (\text{pr}_i)_* M \) for \(i=1,2\), where \( \text{pr}_i (x_1,x_2) = x_i\). 

Equivalently, by desintegrating \(M\) along the fibers of the projection \( \text{pr}_1 : \Delta^L \rightarrow Z\),  \(m_1\) can be \(L\)-transported to \(m_2\) along the flow \(T\) iff there exists a measurable family of probability measures \( \{p_{1,2} ^{x_1}\} \) on \(Z\) defined for \(m_1\)-a.e. \(x_1\in Z\), the  support of \( p_{1,2}^{x_1}\) being contained in the interval \(T^{[-L,L]}(x_1)\), and such that 
\[ m_2 = \int p_{1,2} ^{x_1} \ m_1 (dx_1). \]
\end{definition} 

\begin{proposition}\label{p: basic transportation along a flow}  
The following properties are satisfied:
\begin{enumerate}
\item (symmetry) if \(m_1\) can be \(L\)-transported to \(m_2\), then \(m_2\) can be \(L\)-transported to \(m_1\) as well,
\item (closedness) the set of pairs of probability measures \( \{ m_1, m_2\}   \) on \(\Der\) that can be \(L\)-transported to each other is closed in the weak topology.
\item (linear combination) if \( (P, m_P)\) is a probability space, and \( p\in P\mapsto m_p\in \text{Prob} (\Der) \) is a measurable family of probability measures on \(\Der\) that can be \(L\)-transported to the measure \(m\), so is the measure \( \int _P m_p \ m_P (dp) \).  
\item (transitivity) given three probability measures \(m_1,m_2,m_3\) on \(\Der\),  if \( m_1\) can be \(L_1\)-transported to \(m_2\), and \(m_2\) can be \(L_2\)-transported to \(m_3\), then \( m_1\) can be \((L_1+L_2)\)-transported to \( m_3\). 
\end{enumerate}
\end{proposition}

\begin{proof}
The first point is a consequence of the fact that the set \( \Delta^L\) is invariant by the involution \( (x_1,x_2)\mapsto (x_2,x_1)\).

The second a consequence of the fact that \( \Delta^L\) is closed in \( \Der \times \Der\).  

The third point is immediate if \(P\) is finite. If \( P\) is infinite, use the law of large numbers which shows that 
\[ \int m_p \ m_P (dp) = \lim _{n \rightarrow \infty} \frac{1}{n} \sum_{1\leq k\leq n} m_{p_k} \text{ for } m_P ^{\N ^*}\text{-a.e. sequence } (p_n)_n \in P^{\N^*} ,\]
and conclude using the second point.   

For the fourth point, denote by  \( \{ p_{1,2} ^{x_1} \} _{x_1\in \Der}\) (resp. \( \{p_{2,3} ^{x_2} \}_{x_2\in \Der}\)) a measurable family of probability measures defined for \(m_1\)-a.e. \(x_1\in \Der\) (resp. \(m_2\)-a.e. \(x_2\in \Der\)), whose support are contained in \(T^{[-L_1, L_1]}(x_1)\) (resp. \(T^{[-L_2, L_2]}(x_2) \) ) and such that \[ \int p_{1,2} ^{x_1} \ m_1(dx_1) = m_2 \ \ \text{  (resp. }  \int p_{2,3} ^{x_2} \ m_2(dx_2) = m_3\text{)}.\] In particular we have that for \(m_1\)-a.e. \( x_1\in \Der\), \(p_{1,2}^{x_1}\)-a.e. \(x_2\in \Der\) the measure \(p_{2,3} ^{x_2}\) is well-defined,  and that denoting \( p_{1,3} ^{x_1} = \int _{\Der} p_{2,3}^{x_2} \ p_{1,2}^{x_1} (dx_2) \), we have \[ m_3 = \int p_{2,3} ^{x_2} \ m_2 (dx_2) = \int \left(\int  p_{2,3} ^{x_2}  \ p_{1,2}^{x_1} (dx_2)\right) \ m_1(dx_1) =\int p_{1,3} ^{x_1} \ dm_1(dx_1).\] Hence, the fourth point follows from the fact that for \(m_1\)-a.e. \(x_1\), the probability measure  \( p_{1,3} ^{x_1}\) is supported on the interval \( T^{L_1+L_2} (x_1)\).   \end{proof}

For the next proposition, let us introduce the following concept. 

\begin{definition} [Length of a flow box] A flow box \(B\)  of the translation flow is said of length at least \( \alpha \) if the connected components of the intersection of \(B\) with the \(T\)-trajectories are intervals of length bounded from  below by \(  \alpha\). The following claim will be interesting for long flow boxes, namely in the regime \( \frac{L}{\alpha}=o(1)\). \end{definition}

\begin{proposition} \label{p: claim 5} Assume \(T\) is free. Let \(m_1\) be a probability measure on \(\Der\) which is \(T\)-invariant, and let \( m_2\) be a probability measure on \( \Der\) which is \(L\)-transported from \(m_1\). Then for every flow box \(B\) of length bounded from below by \(\alpha\) and of positive \(m_1\)-measure, we have 
\[ \left\lvert \frac{m_2(B)}{m_1(B)}-1 \right\lvert \leq \frac{L }{\alpha}. \]
\end{proposition}

\begin{proof}
Introduce the two sets  
\[ B^{L}:= \bigcup_{|l|\leq L} T^l (B) \text{ and } B^{-L} := \bigcap _{|l|\leq L} T^l (B).\]

Let \( M\) be a probability measure supported on \( \Delta^L\) which is such that \( (\text{pr}_i) _* (M) = m_i\) for \(i=1,2\). Since \( \text{pr}_1^{-1} (B^{-L})) \subset \text{pr}_2^{-1} (B) \subset \text{pr}_1^{-1} (B^L)\), we have 
\[ m_1 ( B^{-L} )  = M( \text{pr}_1^{-1}(B^{-L}))  \leq  M ( \text{pr}_2^{-1} (B) )  =\]
\[=m_2 ( B )  \leq M( \text{pr}_1^{-1}(B^L) ) =m_1(B^L).\]

Now, the measure \(m_1\) being invariant by \(T\), and the length of \(B\) being bounded from below by \(\alpha\), we get the two inequalities
\[  \left\lvert \frac{m_1(B^{\pm L})}{m_1(B)}-1 \right\lvert \leq \frac{L }{\alpha},\]
and Proposition \ref{p: claim 5} follows. \end{proof}

\subsection{The contradiction}

Recall that for every \(g\in G\) we define the measure  \( \mDer^g  := \widetilde{D_{\mathcal L_\Der}} (g, \cdot ) \mDer\), and that a lift of \( \mX\) to \( G\times \Der\) is given by \( \widetilde{\mX} = \int \delta_g \otimes \mDer^g \ \HaarG (dg) \) (see section \ref{sss: mX}). In particular,  the family \( g\mapsto \mDer ^g\)  only depends on the  \(g_1\)-coordinate as well; we will then denote in the sequel \( \mDer^{g_1} = \mDer ^{(g_1, \ldots, g_n)}\). 

\begin{lemma}
The bounded function \( g_1 \in G_1 \mapsto \mDer ^{g_1} \in \text{Prob} (\Der ) \) is left \(\mu_{G_1}\)-harmonic for any measure \(\mu_{G_1}\) of type \(\type\) on \(G_1\).  In particular, there exists a measurable family \(\{ \mDer ^{ \zeta_1} \} _{\zeta_1 \in B(G_1) } \) of measures on the almost-periodic space \( \Der\) such that 
\begin{equation}\label{eq: harmonic family of measures}  \mDer^{g_1}  = \int_{\zeta_1\in B(G_1)}   \mDer ^{\zeta_1} \ \mu_{B(G_1)} ^{g_1} (d\zeta_1) \end{equation}
where \(\mu_{B(G_1)} ^{g_1} \) is the harmonic measure on \(B(G_1)\) from the point \(g_1\). 

\end{lemma}

\begin{proof} The first part of the Lemma is an immediate consequence of Theorem \ref{t: poisson boundary product} applied to the bounded function \( (g_1, \ldots, g_n) \in G_1 \times \ldots \times G_n \mapsto \mDer ^{g_1} \in \text{Prob} (\Der) \). The second part is a formulation of the Poisson formula due to Furstenberg, see the discussion after Theorem \ref{t: Poisson formula}. \end{proof}

\begin{proposition}\label{p: claim 4} There exists \(L\geq 0\) such that for Lebesgue a.e. pair \( (\zeta_1, \zeta_2)\in B(G_1) \), the measure \(\mDer ^{\zeta_1}\) can be \(L\)-transported to \(\mDer ^{\zeta_2}\).\end{proposition}

\begin{proof} 
First notice that there exists an integer \(p\) such that the \(p\)-th power of the neighborhood constructed in Lemma \ref{l: bounded displacement of sequences} contains the compact group \( K_1 \). In particular, we infer that with \(L= pl\),  any element of \(K_1\) 
is a limit \(\lim _n p_2(\gamma_n) \), where \(\gamma_n\) is a sequence of elements of \(\Gamma\) whose displacements along the leaves of the lamination \(\mathcal L_\Der\) has a modulus  uniformly bounded by \(L\).

Fix \(\varepsilon >0\). Lusin's theorem furnishes a measurable subset \( E\subset B(G_1) \) whose \(\mu_{B(G_1)}^e\)-measure (the unique probability measure on \(B(G_1)\) invariant under the group \( K_1 \)) is at least \( 1- \varepsilon\) and such that the restriction of the measurable equivariant map \( \zeta_1 \mapsto  \mDer ^{\zeta_1} \) to \(E\) is continuous (with the induced topology). This means that if \( \zeta_ n \rightarrow \zeta \), with \( \zeta_n \in E\) for each \(n\) and \(\zeta \in E\), then \(  \mDer ^{\zeta_n} \rightarrow  \mDer ^{\zeta})\). 

Fix \( k \in K_1\) and let \((\gamma_n)_n\) be the sequence in \(\Gamma\) whose displacements along \(\mathcal L_\Der\) is uniformly bounded by \(L\) and such that \( p_2 (\gamma_n) \) converges to \(k \).  Given \( \zeta_1, \zeta_2\in B(G_1) \) such that \( k (\zeta_1)= \zeta_2\), suppose that \( \gamma_n (\zeta_1) \) belongs to \(E\) for an infinite number of \(n\)'s.  Since the displacement of \( \gamma_n\) is uniformly bounded by \(L\), one can \(L\)-transport the measure \( \mDer ^{\zeta_1}\) to the measure \( (\gamma_n)_* \mDer ^{\zeta_1} =\mDer ^{ \gamma_n (\zeta_1)}\). Since \( \mDer ^{\gamma_n (\zeta_1)} \) converges to \( \mDer ^{\zeta_2} \) along the subsequence of \(n\)'s such that \( \gamma_n (\zeta_1) \) belongs to \(E\), we can \(L\)-transport the measure \( \mDer ^{\zeta_1} \) to the measure \( \mDer ^{\zeta_2}\). 

The set of \( \zeta_1\)'s in \( B(G_1) \) which are such that \(\gamma_n (\zeta_1) \) belongs to \(E\) for an infinite number of \(n\)'s is the complementary of \( \bigcup _{N\in \N }  G_N\) with \( G_N = \bigcap_{n\geq N}  \gamma_n ^{-1} E ^c \). We have \( G_0 \subset G_1 \subset \ldots\) and each \( G_N\) has \(\mu_{B(G_1)}^e\)-measure bounded by the one of \(\gamma_N^{-1} E^c\), namely 
by \(\varepsilon\). In particular, the measure of \( \bigcup_N G_N\) is bounded by \(\varepsilon\) as well. As a conclusion, we have proved that apart from a set of \( \zeta\)'s of \(\mu_{B(G_1)}^e\)-measure bounded by \(\varepsilon\), the measures \( \mDer ^{\zeta_1} \) and \( \mDer ^{\zeta_1} \) can be \(L\)-transported to each other, for \(\zeta_2=k (\zeta_1)\). 

This being valid for every \(k\in K_1\) and every \(\varepsilon >0\), The proposition follows. 
\end{proof}

The basic properties of \(L\)-transportation stated in Proposition \ref{p: basic transportation along a flow}, Proposition \ref{p: claim 4}, and equation \eqref{eq: harmonic family of measures} show that for every couple of points \(g_1, g_2\in G\), the measure \(  \mDer^{g_1} \) can be \(L\)-transported to \(  \mDer^{g_2}\), where \(L\) is the constant appearing in \ref{p: claim 4}.  

We choose \(g_1=e\) and \( g_2= \gamma^{-1} \) for \(\gamma \in \Gamma\) to be chosen later on, and denote \(m_1= \mDer^{g_1}\) and \(m_2=   \mDer^{g_2 } =\gamma_* m_1\). The measure \(m_1\) is invariant by the translation flow \(T\), and can be \(L\)-transported to any of its images \(m_2=\gamma_* m_1\), \(\gamma\in \Gamma\). 

Now take any large number \(\alpha \), and \(x\in \Der\) a point in the support of \(m_1\). By the global contraction property, there exists an element \(\gamma\in \Gamma\) which is such that \( \gamma ^{-1} (T^{[x-\alpha ,x+\alpha ]}(x) ) \) is a piece of trajectory of length bounded from above by \(1\). By continuity, there exists a flow box \(B\) of length at least \(\alpha\) containing \(T^{[x-\alpha ,x+\alpha ]}(x)\), such that \(\gamma^{-1} (B)\) has length bounded from above by \(2\) (namely the component of the intersection of the trajectories with \(\gamma^{-1} (B)\) have a length bounded by \(2\) in the time parametrization). In particular, we have 
\[ \frac{m_2(B) }{m_1(B)} = \frac{m_1(\gamma^{-1}(B))}{m_1(B)} \leq \frac{2}{\alpha} \]
and this contradicts Proposition \ref{p: claim 5} if \(\alpha\) is chosen sufficiently large since \(m_1\) can be \(L\)-transported to \( m_2\).  

This contradiction concludes the proof of Theorem \ref{t: no orderable lattice 2}.

\end{document}